\documentclass[11pt,reqno]{amsart}

\usepackage[all]{xy}
\usepackage{graphicx}
\usepackage{amsmath,amsopn,amssymb,amsfonts,stmaryrd,latexsym,amsxtra, nccmath}
\usepackage{verbatim}
\usepackage{amsthm}
\usepackage{mathtools}
\usepackage{xcolor}
\usepackage{simplewick}
\usepackage{enumitem}
\usepackage[framemethod=TikZ]{mdframed}
\usepackage{chngcntr}
\usepackage{mathrsfs}
\usepackage{booktabs}
\usepackage{caption}
\usepackage{enumitem}
\usepackage{todonotes}
\usepackage{float}
\usepackage{cancel}
\usepackage{tabu}
\usepackage{comment}

\usepackage{tikz}
\usepackage{verbatim}
\usepackage{pgfplots}
\pgfplotsset{compat=1.12}
\usepgfplotslibrary{fillbetween}
\usetikzlibrary{intersections}
\usetikzlibrary{patterns}
\usetikzlibrary{quotes,angles}

\newcommand{\R}{{\mathbb R}}
\newcommand{\C}{{\mathbb C}}
\newcommand{\Z}{{\mathbb Z}}
\newcommand{\N}{{\mathbb N}}
\newcommand{\Q}{{\mathbb Q}}
\newcommand{\IH}{{\mathbb H}}

\newcommand{\CA}{\mathcal{A}}

\newcommand{\CE}{\mathcal{E}}

\newcommand{\CI}{\mathcal{I}}

\newcommand{\sgn}{\mbox{sgn}}

\newcommand{\SL}{\mathrm{SL}}

\newcommand{\zz}{\mathfrak{z}}

\allowdisplaybreaks
\theoremstyle{plain}
\newtheorem{theorem}{Theorem}[section]

\newtheorem{lemma}[theorem]{Lemma}

\theoremstyle{definition}

\newtheorem*{remark}{Remark}
\newtheorem*{remarks}{Remarks}

\numberwithin{equation}{section}

\newcommand{\pmat}[1]{\left( \smallmatrix #1 \endsmallmatrix \right)}

\renewcommand{\sgn}{\textnormal{sgn}}

\setlist[enumerate]{leftmargin=*, listparindent=\parindent, parsep=0pt, 
	font=\upshape} 
\setlist[itemize]{leftmargin=*} 

\def\lp{\left(}
\def\rp{\right)}

\def\d{\delta}

\def\x{\xi}

\def\e{\varepsilon}

\def\t{\tau}

\renewcommand{\sgn}{{\rm sgn}}

\def\wh{\widehat}

\def\slashchar#1{\setbox0=\hbox{$#1$}           
   \dimen0=\wd0                                 
   \setbox1=\hbox{/} \dimen1=\wd1               
   \ifdim\dimen0>\dimen1                        
      \rlap{\hbox to \dimen0{\hfil/\hfil}}      
      #1                                        
   \else                                        
      \rlap{\hbox to \dimen1{\hfil$#1$\hfil}}   
      /                                         
   \fi}                                        %

\setlist[itemize]{noitemsep, topsep=0pt}

\newcounter{exercise}
\renewcommand{\theexercise}{\thesection.\arabic{exercise}}

\mdfdefinestyle{exercise}{%
    linecolor=blue!40,
    outerlinewidth=2pt,
    leftline=false,rightline=false,
    skipabove=\baselineskip,
    skipbelow=\baselineskip,
    frametitle=\mbox{},
}
\newmdenv[%
    style=exercise,
    settings={\global\refstepcounter{exercise}},
    frametitlefont={\bfseries Exercise~\theexercise\quad},
]{exercise}
\newmdenv[%
    style=exercise,
    frametitlefont={\bfseries Exercise~\quad},
]{exercise*}

\newmdenv[%
    backgroundcolor=gray!8,
    linecolor=violet,
    outerlinewidth=1pt,
    roundcorner=3mm,
    skipabove=\baselineskip,
    skipbelow=\baselineskip,
]{boxes}
  
\makeatletter
\newcommand{\vast}{\bBigg@{2}}
\newcommand{\Vast}{\bBigg@{5}}
\makeatother

\renewcommand{\pmod}[1]{\  \,  \left(  \mathrm{mod} \,  #1 \right)}
\newcommand{\erf}{{\mathrm{erf}}}

\begin{document}

\title[Fourier Coefficients of Weight Zero Mixed False Modular Forms]{Fourier Coefficients of Weight Zero Mixed False Modular Forms}

\author{Giulia Cesana}
\address{Giulia Cesana, University of Cologne, Department of Mathematics and Computer Science, Weyertal 86-90, 50931 Cologne, Germany}
\email{gcesana@math.uni-koeln.de}

\begin{abstract}
	In this paper we employ the Circle Method to give exact formulae for Fourier coefficients of an infinite family of weight zero mixed false modular forms using and extending the techniques of Bringmann and Nazaroglu as well as Rademacher. To do so we additionally provide a bound on a Kloosterman sum of modulus $k$.
\end{abstract}


\maketitle

\section{Introduction and statement of results} 
In \cite{bringmann2019framework} Bringmann and Nazaroglu embedded \textit{false theta functions}, functions that resemble theta functions but do not have modular transformation properties, into a modular framework. An example is given by
\begin{align*}
	\psi(z;\t)\coloneqq i\sum_{n\in\Z} \sgn\left(n+\frac 12\right) (-1)^n q^{\frac 12\left(n+\frac 12\right)^2} \zeta^{n+\frac 12},
\end{align*}
where here and throughout $\zeta\coloneqq e^{2\pi i z}$ for $z\in\C$, $q\coloneqq e^{2\pi i\tau}$, with $\t\in\mathbb{H}$, and \begin{align*}
\sgn(n)\coloneqq \begin{cases}
0 & \text{ if } n=0,\\
1 & \text{ if } n>0,\\
-1 & \text{ if } n<0,
\end{cases}
\end{align*}
as usual.
They found the modular completion\footnote{These are modular objects from which the original function can be easily recovered, here for example by taking the limit (see \eqref{limit completion}).} of those false theta functions, with $w\in\IH$, given by (see \cite[equation (1.2)]{bringmann2019framework}) 
\begin{align*}
	\widehat{\psi}(z;\t,w) \coloneqq i\sum_{n\in\Z} \erf\left(-i\sqrt{\pi i(w-\t)} \left(n+\frac 12 +\frac{\operatorname{Im}(z)}{\operatorname{Im}(\t)}\right)\right) (-1)^n q^{\frac 12\left(n+\frac 12\right)^2} \zeta^{n+\frac 12},
\end{align*}
and repaired the modular invariance, where $\erf (z)\coloneqq\frac{2}{\sqrt{\pi}} \int_{0}^{z} e^{-t^2}dt$ denotes the {\it error function} and where $\widehat{\psi}$ satisfies (see \cite[equation (1.3)]{bringmann2019framework}) 
\begin{align}\label{limit completion}
	\lim_{t\rightarrow\infty} \widehat{\psi}(z;\t,\t+it+\varepsilon) = \psi(z;\t)
\end{align}
if $-\frac 12 < \frac{\operatorname{Im}(z)}{\operatorname{Im}(\t)} <\frac 12$ and $\varepsilon>0$ arbitrary. Note that here and in the following we define the square-root on a cut-plane excluding the negative reals and imposing positive square roots for positive real numbers.

As an application from this framework they considered the false theta functions at rank one (see \cite[equation (1.6)]{bringmann2019framework}) 
\begin{equation*} 
	F_{j,N}(\tau) \coloneqq \sum_{\substack{n\in \Z \\ n \equiv j \pmod{2N}}} \sgn (n) \, 
	q^{\frac{n^2}{4N}},
\end{equation*}
with $j\in\Z$ and $N\in\N_{>1}$ and showed how the quantum modularity\footnote{For a so-called quantum set $\mathcal{Q}\subset \Q$ we call a function $f:\mathcal{Q}\to\C$ \textit{quantum modular form} of weight $k$, if its obstruction to modularity, namely $f(\tau) - (c\tau +d)^{-k} f(M\tau),$ for $M=(\begin{smallmatrix}
	a & b \\ c & d
	\end{smallmatrix})\in \Gamma \subset \operatorname{SL}_2(\Z)$, behaves ``nice'' in some analytical sense. See e.g.\@, \cite{zagierQMF} for more background on quantum modular forms.} of these functions follows from the construction of their completions.

The motiviation for looking at this function comes from W-algebraic characters, see for example \cite{W-algBringMil, creutzing2014false, creutzing2017quantum, milas2014character}.
Characters of modules of rational vertex operator algebras are often of the form 
$$\frac{f(\tau)}{\eta(\t)^k},$$ 
where $\eta(\tau)\coloneqq q^\frac{1}{24}\prod_{n\geq1}\left( 1-q^n\right)$ is {\it Dedekind's eta function}. In \cite{creutzing2014false} the authors observed that some numerators of atypical characters of the so-called $(1,p)$-singlet algebra are false theta functions of Rogers (see \cite{Andrews2009notebookII}). In particular, the functions 
$$\CA_{j,N} (\tau) \coloneqq \frac{F_{j,N}(\tau)}{\eta(\tau)}$$ 
show up as characters of  the atypical irreducible modules of the $(1,p)$-singlet vertex operator algebra $M_{1,s}$, for $1\leq s \leq p-1$ and $p\in\N_{\geq2}$, that have been studied in \cite{W-algBringMil, creutzing2014false, creutzing2017quantum}.

In 1937 Rademacher \cite{rademacher1938partition} proved the following exact formula for the partition function
\begin{align*}
	p(n) 	=& \frac{2\pi}{\left(24n-1\right)^\frac{3}{4}} \sum_{k\geq 1} \frac{A_k(n)}{k} I_{\frac{3}{2}}\left(\frac{\pi\sqrt{24n-1}}{6k}\right)  ,
\end{align*}
where $A_k(n)$ is a Kloosterman sum given by
\begin{align*}
	A_k(n)\coloneqq\sum_{\substack{h\pmod{k} \\\gcd(h,k)=1}} e^{\pi i s(h,k)} e^{-2\pi i n\frac{h}{k}},
\end{align*}
with $s(h,k)$ the Dedekind sum defined in \eqref{Dedekind sum}
and where $I_\alpha$ denotes the $I$-Bessel function of order $\alpha$, which in the special case of order $\frac{3}{2}$ can be written as
\begin{align*}
	I_{\frac{3}{2}}(z) = \sqrt{\frac{2z}{\pi}} \frac{d}{dz} \left(\frac{\sinh(z)}{z}\right).
\end{align*}

Our goal is to find Rademacher-type exact formulae for the Fourier coefficients of the infinite family of weight zero \textit{mixed false modular forms}\footnote{These are in general linear combinations of false theta functions multiplied by modular forms (see \cite[Section 4]{bringmann2021false}).} $\CA_{j,N} (\tau)$.
Note that it requires considerably more work to obtain an exact formula for a weight zero function than for a function of negative weight. In contrast to negative weight functions, as for example in the work of Bringmann and Nazaroglu \cite{bringmann2019framework}, we have to take special care of the bound of the Kloosterman sum occuring to ensure that the error term in the Circle Method vanishes. In comparison to \cite{rademacher1938fourier} for example, where Rademacher studied the coefficients of the modular invariant $j(\tau)$ of weight zero, we have the additional problems that the Kloosterman sum showing up in our work is much more complicated and can not be immediately bounded by the famous Weil bound and that the transformation behavior of our family of functions is not as simple as the one of a modular form.

In this paper we let
\begin{align}\label{equation: series Aj,N}
\CA_{j,N} (\tau) \eqqcolon q^{\frac{j^2}{4N}-\frac{1}{24}}\left(a_{j,N}(0) + \sum_{n\geq1} a_{j,N}(n)q^n\right). 
\end{align}
Extending the techniques presented in \cite[Section 3]{bringmann2019framework} and \cite{rademacher1938fourier} we prove the following theorem, which, to the best of the author's knowledge, is the first example of an exact formula of a weight zero mixed false modular form.
\begin{theorem} \label{main Thm}
	For all $n\geq1$ and $\sqrt{\frac{N}{6}}\notin\Z$ we have 
	\begin{small}
		\begin{align}\label{equation: coefficients final with P.V.}
	a_{j,N}(n) =& -\frac{2\pi i }{\sqrt{n+\frac{j^2}{4N}-\frac{1}{24}}}\sum_{k\geq1} \sum_{r=1}^{N-1} \sum_{\kappa=0}^{k-1} \frac{ K_{k,j,N}(n,r,\kappa)}{k^2}  \\
		&\times  \mathrm{P.V.} 	\int_{-\sqrt{\frac{1}{24N}}}^{\sqrt{\frac{1}{24N}}} \sqrt{\frac{1}{24}-Nx^2} \cot\left(\pi \lp -\frac xk+ \frac{\kappa}{k} + \frac{r}{2Nk} \rp \right)  I_1\left(\frac{4\pi\sqrt{n+\frac{j^2}{4N}-\frac{1}{24}}}{k} \sqrt{\frac{1}{24}-Nx^2}\right)  \, dx , \notag
		\end{align}
	\end{small}
\hspace{-0.25cm} where $ K_{k,j,N}(n,r,\kappa)$ is a {\it Kloosterman sum} defined as
	\begin{align}\label{Kloostermansum}
	 K_{k,j,N}(n,r,\kappa)\coloneqq& \sum_{\substack{0\leq h < k \\ \gcd(h,k)=1 }}  \chi_{j,r} (N,M_{h,k})  \zeta_{24k}^{\left(24N\left(\kappa+\frac{r}{2N}\right)^2-1\right)h'-24\left(n+\frac{j^2}{4N}-\frac{1}{24}\right)h  },
	\end{align} 
	with $h'$ a solution of $hh'\equiv -1 \pmod{k}$, $M_{h,k}=\left(\begin{smallmatrix}
	h' & -\frac{hh'+1}{k}\\ k & -h
	\end{smallmatrix}\right)$, $\chi_{j,r}(N,M)$ the multiplier defined in \eqref{multiplier}, and $\zeta_\ell\coloneqq e^{\frac{2\pi i}{\ell}}$ with $\ell \in\N$ an $\ell$-th root of unity.
\end{theorem}

\begin{remarks}
	\begin{enumerate}[wide, labelwidth=!, labelindent=0pt]
		\item Although this representation as a convergent series does not hold for $n = 0$ we obtain that $a_{j,N} (0) = 1$, independent of $j$ and $N$.
				
		\item Note that we are able to split the principal value integral in \eqref{equation: coefficients final with P.V.}, which gives us a more explicit but also more complicated version of our main result, as can be seen in \eqref{equation: coefficients final without P.V.}.
	\end{enumerate}
\end{remarks}

As a second result, which will be extremely helpful in the proof of Theorem \ref{main Thm}, we are able to give a bound on the Kloosterman sum defined in \eqref{Kloostermansum}. In particular, we show the following theorem.
\begin{theorem}\label{Thm Kloosterman bound}\label{THM KLOOSTERMAN BOUND}
	For $\varepsilon>0$ we have that
	\begin{align}\label{bound: Kloosterman sum}
		K_{k,j,N}(n,r,\kappa) = O_N\left(n k^{\frac 12+\varepsilon}\right)
	\end{align}
	as $k\to \infty$.
\end{theorem}
As the main tool to prove this theorem we use the following result by Malishev.
\begin{lemma}\label{Lemma: Malishev}{\normalfont(see \cite[page 482]{knopp1964kloosterman})}
	Let 
	\begin{align*}
	K_\rho(\mu_*,\nu_*;G) \coloneqq \sum_{\substack{h\pmod{G} \\ \gcd(h,G)=1}} \left(\frac{h}{\rho}\right) \exp\left(\frac{2\pi i}{G} (\mu_* h+\nu_*h')\right),
	\end{align*}
	where $\mu_*$ and $\nu_*$ are integers, $G$ is a positive integer, and $\rho$ is an odd positive integer all of whose prime factors devide $G$. Furthermore $h'$ is any integral solution of the congruence $hh'\equiv 1 \pmod{G}$ and $(\frac{h}{\rho})$ is the Jacobi symbol. Then
	\begin{align*}
	\left|K_\rho(\mu_*,\nu_*;G)\right| \leq A(\varepsilon) G^{\frac{1}{2}+\varepsilon} \min\left(\gcd(\mu_*,G)^\frac{1}{2}, \gcd(\nu_*,G)^\frac{1}{2}\right)
	\end{align*}
	for each $\varepsilon>0$, where $A(\varepsilon)>0$ depends only on $\varepsilon$.
\end{lemma}

The paper is structured as follows. In Section \ref{section: Modular Transformations with Mordell-type Integrals} we use the modular completion of $F_{j,N}$ and its modular transformation behavior to determine the ``false'' modular behavior of $\mathcal{A}_{j,N}$. With this we rewrite the obstruction to modularity term, respectively the error of modularity plus the holomorphic part of our function, as a Mordell-type integral. In Section \ref{section: Bound Kloosterman sum} we go on by proving Theorem \ref{Thm Kloosterman bound} and use the Circle Method, to prove Theorem \ref{main Thm} in Section \ref{section: Circle Method}. We end the paper with some numerical results in Section \ref{section: Numerical results}.

\section*{Acknowledgments}
The author thanks Kathrin Bringmann and Caner Nazaroglu for suggesting the topic. The author thanks Kathrin Bringmann in addition for many helpful discussions and valuable comments on earlier versions of this paper. Further the author wants to thank Caner Nazaroglu for many helpful discussions, double checking parts of the math, and comments that helped to improve the paper a lot. Last but not least the author wants to thank Johann Franke for his help with some numerical experiments.

\section{Modular Transformations with Mordell-type Integrals}\label{section: Modular Transformations with Mordell-type Integrals}

\subsection{Modular transformations}\label{subsection: Modular Transformations} 
We first note that a simple straight-forward calculation shows that
$$F_{0,N}(\tau) =0$$
for every $\tau\in\mathbb{H}$. Thus we from now on assume that $j\neq0$. Furthermore we can restrict to $1\leq j\leq N-1$, since $F_{j,N}(\tau)=-F_{-j,N}(\tau)$ and $F_{j\pm2N,N}(\tau)=F_{j,N}(\tau)$.

According to \cite[Section 4]{bringmann2019framework} a modular completion of $F_{j,N}$ can be written as (for $\tau,w \in \mathbb{H}$)
\begin{equation*}
	\wh{F}_{j,N}(\tau, w) \coloneqq \sum_{\substack{n\in \Z \\ n \equiv j \pmod{2N}}} 
	\erf \lp - i \sqrt{\pi i (w -\tau)}  \frac{n}{\sqrt{2N}} \rp \, q^{\frac{n^2}{4N}}.
\end{equation*}
This modular completion can conveniently be rewritten as (see \cite[equation (4.2)]{bringmann2019framework}) 
\begin{align}\label{completion F_j,N}
	\wh{F}_{j,N}(\tau, w)  = \pm F_{j,N}(\tau) - \sqrt{2N} \int_w^{\tau + i \infty \pm \e}
	\frac{f_{j,N} (\zz) }{\sqrt{i(\zz - \tau)}}  d\zz,
\end{align}
where $\e > 0$ and $f_{j,N} $ are the vector-valued cusp forms of weight $\frac{3}{2}$
\begin{equation*}
	f_{j,N} (\tau) \coloneqq \frac{1}{2N} 
	\sum_{\substack{n\in \Z \\ n \equiv j \pmod{2N}}} n \, q^{\frac{n^2}{4N}} = \sum_{n\in\Z} \left(n+\frac{j}{2N}\right)q^{N\left(n+\frac{j}{2N}\right)^2}.
\end{equation*} 
Equation \eqref{completion F_j,N} can also be understood from the writing of $F_{j,N}(\tau)$ as a holomorphic Eichler integral\footnote{For a cusp form $f$ of weight $k\in 2\N$, Eichler introduced in \cite{eichler1957integrale} the integral
	$ \int_{\tau_0}^{\tau} (\tau-\zz)^{k-2} f(\zz) d\zz,$
	which is independent of the path of integration. Integrals of this shape are now called Eichler integrals.}.

The modular transformations of $F_{j,N}$ can be deduced from \cite[equation (4.5)]{bringmann2019framework} 
\begin{equation} \label{modular transformation completion F}
	\wh{F}_{j,N} \lp \frac{a \tau + b}{c \tau + d}, \frac{a w + b}{c w + d} \rp = 
	\chi_{\tau, w} (M) \, (c \tau+d)^{\frac{1}{2}} \, \sum_{r=1}^{N-1} \psi_{j,r} (N,M) \, 
	\wh{F}_{r,N} (\tau, w),
\end{equation}
where $M \coloneqq \pmat{a & b \\ c&d} \in \SL_2 (\Z)$,
\begin{equation*}
	\chi_{\tau, w} (M) \coloneqq \sqrt{\frac{i(w-\tau)}{(c\tau +d)(c w +d)}}
	\frac{\sqrt{c\tau+d} \sqrt{c w +d}}{\sqrt{i(w-\tau)}},
\end{equation*}
and, for $j,r\in\{1,\dots, N-1\}$,
\begin{equation*}
	\psi_{j,r} (N,M) \coloneqq  \begin{cases}
		e^{2 \pi i a b \frac{j^2}{4N}}  \, e^{- \frac{\pi i}{4} (1 - \sgn(d))} \d_{j,r}
		 &\mbox{if } c=0,  \\
		e^{- \frac{3\pi i}{4}  \sgn(c)} \sqrt{\frac{2}{N |c|}} 
		\displaystyle\sum_{\ell=0}^{|c|-1} e^{\frac{\pi i}{2 N c} \lp a (2N\ell + j)^2 + d r^2 \rp}
		\, \sin \lp  \frac{\pi r (2 N \ell + j)}{N |c|} \rp
		 &\mbox{if } c \neq 0.
	\end{cases}
\end{equation*}

For reference let us also note the modular transformation of the eta function given by
\begin{equation} \label{modular transformation eta}
	\eta \lp \frac{a \tau + b}{c \tau + d} \rp = \nu_\eta (M) \, (c \tau+d)^{\frac{1}{2}} \,
	\eta (\tau),
\end{equation}
where for $c>0$ we have
\begin{equation*}
	\nu_\eta (M) \coloneqq  \exp \lp \pi i  \lp  \frac{a+d}{12c} -\frac{1}{4} + s(-d,c) \rp \rp,
\end{equation*}
with the \textit{Dedekind sum} given by
\begin{align}\label{Dedekind sum}
s(h,k) \coloneqq \sum_{r=1}^{k-1} \frac{r}{k} 
\lp \frac{hr}{k} - \left\lfloor\frac{hr}{k}  \right\rfloor - \frac{1}{2} \rp .
\end{align}

\begin{remark}
	For $M \coloneqq \pmat{a & b \\ c&d}$ an alternate representation of the eta-multiplier is given by (see e.g.\@, \cite[Lemma 2.1]{bringmann2018modularity})
	\begin{align*}
	\nu_\eta(M) = \begin{cases}
	\left(\frac{d}{|c|}\right) e^{\frac{\pi i}{12} \left((a+d)c-bd\left(c^2-1\right)-3c\right)} &\text{ if $c$ is odd},\\
	\left(\frac cd\right) e^{\frac{\pi i}{12}\left(ac\left(1-d^2\right)+d(b-c+3)-3\right)} &\text{ if $c$ is even,}
	\end{cases}
	\end{align*}
	where $(\frac{\cdot}{\cdot})$ is the {\it extended Legendre symbol}, also known as {\it Kronecker symbol}.
\end{remark}

By combining \eqref{completion F_j,N} and \eqref{modular transformation completion F} we can write the modular transformation of $F_{j,N}$ as 
\begin{align*}
&F_{j,N} (\tau) - \chi_{\tau, w} (M) (c\tau +d)^{-\frac{1}{2}} \sum_{r=1}^{N-1}
\psi_{j,r} \left(N,M^{-1}\right) \, F_{r,N}  \lp \frac{a \tau + b}{c \tau + d} \rp 
\\ 
=& \sqrt{2N}  
\int_w^{\tau + i \infty + \e} \frac{f_{j,N} (\zz) }{\sqrt{i(\zz - \tau)}} d\zz
\\
&- \sqrt{2N}  \chi_{\tau, w} (M) (c\tau +d)^{-\frac{1}{2}}
\sum_{r=1}^{N-1} \psi_{j,r} \left(N,M^{-1}\right)
\int_{\frac{a w + b}{c w + d}}^{\frac{a \tau + b}{c \tau + d} + i \infty + \e} \frac{f_{r,N} (\zz) }{\sqrt{i(\zz - \frac{a \tau + b}{c \tau + d})}} d\zz .
\end{align*}
Assuming $c >0$ and taking $w \to \tau + i \infty + \e$ we get $\chi_{\tau, w} \to 1$ and hence 
\begin{align}\label{transformation formula F_jN}
&F_{j,N} (\tau) -  (c\tau +d)^{-\frac{1}{2}} \sum_{r=1}^{N-1}
\psi_{j,r} \left(N,M^{-1}\right) \, F_{r,N}  \lp \frac{a \tau + b}{c \tau + d} \rp 
\notag \\
 & \qquad
=- \sqrt{2N}  (c\tau +d)^{-\frac{1}{2}}
\sum_{r=1}^{N-1} \psi_{j,r} \left(N,M^{-1}\right)
\int_{\frac{a}{c}}^{\frac{a \tau + b}{c \tau + d} + i \infty + \e} \frac{f_{r,N} (\zz) }{\sqrt{i(\zz - \frac{a \tau + b}{c \tau + d})}} d\zz .
\end{align}

Now define 
\begin{equation}\label{multiplier}
\chi_{j,r} (N,M) \coloneqq \nu_\eta (M) \, \psi_{j,r} \left(N,M^{-1}\right).
\end{equation}
Also for $\varrho \in \Q$ we define
\begin{equation*}
\CE_{j,N,\varrho} (\tau) \coloneqq \sqrt{2N}
\int_{\varrho}^{\tau + i \infty + \e} \frac{f_{j,N} (\zz) }{\sqrt{i(\zz - \tau )}} d\zz .
\end{equation*}
Using this together with \eqref{modular transformation eta} and \eqref{transformation formula F_jN} immediately gives the modular transformation equation for $\CA_{j,N}$.
\begin{lemma} 
	For $M \coloneqq \pmat{a & b \\ c&d} \in \SL_2 (\Z)$ with $c>0$ we have
	\begin{align} \label{modular transformation A_jN}       
\CA_{j,N} (\tau) =& 
\sum_{r=1}^{N-1} \chi_{j,r} (N,M) \left( \CA_{r,N}  \lp \frac{a \tau + b}{c \tau + d} \rp 
-  \eta \lp \frac{a \tau + b}{c \tau + d}\rp^{-1} \,
\CE_{r,N,\frac{a}{c}} \lp \frac{a \tau + b}{c \tau + d} \rp \right) .
\end{align}
\end{lemma}

\subsection{Mordell-type integrals}\label{subsection: Mordell-type Integrals}
Next we want to rewrite the obstruction to modularity term as a Mordell-type integral. If $\varrho \in \Q$ and $V \in \C$ with $\operatorname{Re}(V) > 0$, then we have,
\begin{align*}
\CE_{j,N,\varrho} \left(\varrho + iV\right)
&=
\sqrt{2N}\int_{\varrho}^{\varrho + iV + i \infty + \e}
\frac{f_{j,N} \left(\zz\right) }{\sqrt{i\left(\zz - \left(\varrho + iV\right) \right)}} d\zz
\\
&=
i \sqrt{2N}\int_{-V}^{\infty - i \e}
\frac{f_{j,N} \left(\varrho + i \left(\zz+V\right)\right) }{\sqrt{-\zz}} d\zz  
\\
&= i \sqrt{2N} \int_{-V}^{\infty - i \e}
\sum_{n \in \Z} \lp n + \frac{j}{2N} \rp 
e^{- 2 \pi N \left(n+\frac{j}{2N}\right)^2 \left(\zz+V\right) + 2 \pi i N \left(n+\frac{j}{2N}\right)^2 \varrho } \frac{d\zz}{\sqrt{-\zz}}.
\end{align*}
First note that the integral is absolutely convergent, since the integrand is a cusp form and therefore exponentially decaying as $\zz\rightarrow -V$ and as $\zz\rightarrow\infty$. 
Also note that each summand is exponentially decaying as $\zz\rightarrow\infty$ but we lose this condition as $\zz\rightarrow -V$. Since we want to be able to interchange the sum and the integral, we rewrite
\begin{align*}
\CE_{j,N,\varrho} \left(\varrho + iV\right) = i \sqrt{2N} \lim\limits_{\delta\rightarrow0^+}\int_{-V+\delta}^{\infty - i \e}
\sum_{n \in \Z} \lp n + \frac{j}{2N} \rp 
e^{- 2 \pi N \left(n+\frac{j}{2N}\right)^2 \left(\zz+V\right) + 2 \pi i N \left(n+\frac{j}{2N}\right)^2 \varrho } \frac{d\zz}{\sqrt{-\zz}}.
\end{align*}
We can exchange the sum and the integral now and get
\begin{align*}
	\CE_{j,N,\varrho}& \left(\varrho + iV\right) \notag\\
	=& i \sqrt{2N}  \lim\limits_{\delta\rightarrow0^+}\sum_{n \in \Z} \lp n + \frac{j}{2N} \rp 	e^{- 2 \pi N \left(n+\frac{j}{2N}\right)^2 V + 2 \pi i N \left(n+\frac{j}{2N}\right)^2 \varrho } \int_{-V+\delta}^{\infty - i \e}
	e^{- 2 \pi N \left(n+\frac{j}{2N}\right)^2 \zz} \frac{d\zz}{\sqrt{-\zz}}.
\end{align*}
Using the identity
\begin{small}
	\begin{align*}
	\int_{-V+\d}^{\infty - i \e} \frac{e^{- 2 \pi N (n+\frac{j}{2N})^2 \zz}}{\sqrt{-\zz}} d\zz = 
	-\frac{i}{\sqrt{2N} (n+\frac{j}{2N})}
	\lp \sgn \lp n+\frac{j}{2N} \rp + \erf \lp i \lp n+\frac{j}{2N} \rp \sqrt{2 \pi N (V-\d)}  \rp \rp , 
\end{align*}
\end{small}
\hspace{-0.25cm} yields
\begin{align} \label{epsilondelta final}
\CE_{j,N,\varrho} \left(\varrho + iV\right) 
=&   \lim\limits_{\delta\rightarrow0^+}\sum_{n \in \Z} e^{- 2 \pi N \left(n+\frac{j}{2N}\right)^2 V + 2 \pi i N \left(n+\frac{j}{2N}\right)^2 \varrho } \notag \\
&\times \left(\sgn\left(n+\frac{j}{2N}\right) +\erf\left(i \left(n+\frac{j}{2N}\right)\sqrt{2\pi N\left(V-\delta\right)} \right) \right) .
\end{align}
To check the convergence as $\delta\rightarrow0^+$ we start analogously to \cite[page 10]{bringmann2019framework}. First we notice that the definition of the error function yields the asymptotic behavior 
$$ \erf\left(iz\right) = \frac{ie^{z^2}}{\sqrt{\pi}z}\left(1+O\left(|z|^{-2}\right)\right)  ,$$
if $|\operatorname{Arg}(\pm z)| <\frac{\pi}{4}$ as $|z|\rightarrow\infty$.
Because of this we note that \eqref{epsilondelta final} does not converge absolutely at $\delta=0$ and we have to be careful by taking the limit $\delta\rightarrow0^+$. 
Seperating this main term of the error function as 
\begin{small}
	\begin{align} \label{equation: seperating main term error function}
		\left(	\erf\left(i \left(n+\frac{j}{2N}\right)\sqrt{2\pi N\left(V-\delta\right)}\right) -\frac{i e^{2\pi N\left(n+\frac{j}{2N}\right)^2 \left(V-\delta\right)}}{\pi\left(n+\frac{j}{2N}\right)\sqrt{2 N\left(V-\delta\right)}}\right)+ \frac{i e^{2\pi N\left(n+\frac{j}{2N}\right)^2 \left(V-\delta\right)}}{\pi\left(n+\frac{j}{2N}\right)\sqrt{2 N\left(V-\delta\right)}},
	\end{align}
\end{small}
\hspace{-0.25cm} we find that the term in the brackets is absolutely and uniformly convergent on compact subsets $\operatorname{Re}(V)>0$ and $0\leq\d\leq\d_0$ for sufficently small $\d_0$, so we can plug in $\d=0$ for these terms to take the limit. 

We go on by focussing on the last term of \eqref{equation: seperating main term error function} whose contibution to $\CE_{j,N,\varrho}(\varrho+iV)$ is given by
\begin{align*}
 \lim\limits_{\d\rightarrow0^+} \frac{i}{\pi\sqrt{2N(V-\d)}}	\sum_{n \in \Z} 
	\frac{ e^{2\pi i N\left(n+\frac{j}{2N}\right)^2 \varrho}}{n+\frac{j}{2N}}
	e^{-2 \pi  N \left(n+\frac{j}{2N}\right)^2 \d  }.
\end{align*}
Since $|e^{-2 \pi  N \left(n+\frac{j}{2N}\right)^2 \d} |<1 $ for all $\d>0$ this series is absolutely convergent for any $\d>0$. If the corresponding series is also convergent for $\d=0$ the limit as $\d\rightarrow 0^+$ is simply the value at $\d=0$, by Abel's Theorem (viewing it as a power series in $e^{-2\pi N\d}$). 

To prove convergence at $\d=0$, recall that we assume $j\neq0$. 
Let $\varrho=\frac hk$ with $\gcd(h,k)=1$ and $k>0$ and consider for $\nu\in\N$ the following sum
\begin{align}\label{sum we want to have convergent}
	\sum_{-\nu\leq n\leq \nu} \frac{ e^{2\pi i N\left(n+\frac{j}{2N}\right)^2 \frac hk}}{n+\frac{j}{2N}} .
\end{align}
We immediately see that
$$ e^{2\pi i N\left(n+k+\frac{j}{2N}\right)^2 \frac hk}  = e^{2\pi i N\left(n+\frac{j}{2N}\right)^2 \frac hk} , $$
which means that the phase is periodic in $n$ with period $k$. Denoting the average as
$$ \mathfrak{a} \coloneqq  \frac 1k \sum_{n \pmod{k}} e^{2\pi i N\left(n+\frac{j}{2N}\right)^2 \frac hk},$$
which is convergent by definition, we can rewrite \eqref{sum we want to have convergent} as
\begin{align}\label{splitting in two sums}
 \sum_{-\nu\leq n\leq \nu} \frac{ e^{2\pi i N\left(n+\frac{j}{2N}\right)^2 \frac hk} -\mathfrak{a} }{n+\frac{j}{2N}} + \sum_{-\nu\leq n\leq \nu} \frac{\mathfrak{a}}{n+\frac{j}{2N}}.
\end{align}
We first look at the second sum in \eqref{splitting in two sums}. We have that
$$ \sum_{-\nu\leq n\leq \nu} \frac{\mathfrak{a}}{n+\frac{j}{2N}} =\frac{\mathfrak{a}}{\frac{j}{2N}} + \mathfrak{a}\sum_{1 \leq n\leq \nu} \left(\frac{1}{n+\frac{j}{2N}}+\frac{1}{-n+\frac{j}{2N}}\right)  =\frac{\mathfrak{a}}{\frac{j}{2N}} + \mathfrak{a}\sum_{1 \leq n\leq \nu} \frac{\frac jN}{\left(\frac{j}{2N}\right)^2-n^2},$$
where the summand is $O(n^{-2})$, which gives us that the sum converges absolutely. Looking at the first sum in \eqref{splitting in two sums} and writing $n=km+r$ we obtain
\begin{align}\label{first sum splitting}
	\sum_{-\nu\leq n\leq \nu} \frac{ e^{2\pi i N\left(n+\frac{j}{2N}\right)^2 \frac hk} -\mathfrak{a} }{n+\frac{j}{2N}} =& \sum_{-\frac{\nu}{k}\leq m \leq \frac{\nu}{k}} \sum_{ r=0}^{  k-1} \frac{ e^{2\pi i N\left(km+r+\frac{j}{2N}\right)^2 \frac hk} -\mathfrak{a} }{km+r+\frac{j}{2N}} +O\left(\frac{k}{\nu}\right)\notag\\
	=& \sum_{-\frac{\nu}{k}\leq m \leq \frac{\nu}{k}} \sum_{ r=0}^{  k-1} \frac{ e^{2\pi i N\left(r+\frac{j}{2N}\right)^2 \frac hk} -\mathfrak{a} }{km+r+\frac{j}{2N}} +O\left(\frac{k}{\nu}\right),
\end{align}
using the periodicity of the exponential. For simplicity we denote $d_r\coloneqq e^{2\pi i N\left(r+\frac{j}{2N}\right)^2 \frac hk}-\mathfrak{a}$.
Since we have that
$$ \sum_{ r=0}^{  k-1}  d_r= -k\mathfrak{a} + \sum_{ r=0}^{  k-1}  e^{2\pi i N\left(r+\frac{j}{2N}\right)^2 \frac hk}  =0$$
by definition of $\mathfrak{a}$, we can write 
$ d_{k-1}= -d_0-d_1-\dots-d_{k-2}.$
With this we can rewrite \eqref{first sum splitting} as
\begin{align*}
	\sum_{-\frac{\nu}{k}\leq m \leq \frac{\nu}{k}}  &\left(d_0\left(\frac{1}{km+\frac{j}{2N}}-\frac{1}{km+k-1+\frac{j}{2N}}\right)+ d_1\left(\frac{1}{km+1+\frac{j}{2N}}-\frac{1}{km+k-1+\frac{j}{2N}}\right)\right.\\
	 &\;\;\;\;\left.+\dots+d_{k-2}\left(\frac{1}{km+k-2+\frac{j}{2N}}-\frac{1}{km+k-1+\frac{j}{2N}}\right)\right)+O\left(\frac{k}{\nu}\right),
\end{align*}
where each term in the brackets is $O(m^{-2})$, which gives us that \eqref{first sum splitting} and thus \eqref{sum we want to have convergent} is absolutely convergent, by taking the limit $\nu\rightarrow\infty$.

Therefore the last term of \eqref{equation: seperating main term error function} is convergent for $\d=0$ and with this we see that \eqref{epsilondelta final} is convergent.
Thus we are allowed to set $\d=0$ in \eqref{epsilondelta final} to obtain
\begin{equation} \label{obstruction as sum}
\CE_{j,N,\varrho} (\varrho + iV) = 
\sum_{n \in \Z} 
 \lp \sgn \lp n+\frac{j}{2N} \rp + \erf \lp i \lp n+\frac{j}{2N} \rp \sqrt{2 \pi N V}  \rp \rp 
e^{2 \pi i N \left(n+\frac{j}{2N}\right)^2 (\varrho+iV)  } . 
\end{equation}

Hence, using \eqref{modular transformation A_jN} and the definition of $\CA_{j,N}$, we get 

\makeatletter
\renewcommand{\maketag@@@}[1]{\hbox{\m@th\normalsize\normalfont#1}}
\makeatother

\begin{small}
	\begin{align} \label{connection C and I}
		\CA_{j,N} (\tau) =&  \sum\limits_{r=1}^{N-1} \chi_{j,r} (N,M) \left( \CA_{r,N}  \lp \varrho+iV \rp 
		-  \eta \lp \varrho+iV\rp^{-1} \,
		\CE_{r,N,\varrho} \lp \varrho+iV \rp \right)\notag\\
		 =&  \sum\limits_{r=1}^{N-1} \chi_{j,r} (N,M) \left( \CA_{r,N}  \lp \varrho+iV \rp -  \eta \lp \varrho+iV\rp^{-1} \, \sum\limits_{n \in \Z} \sgn \lp n+\frac{r}{2N} \rp e^{2 \pi i N 	\left(n+\frac{r}{2N}\right)^2 (\varrho+iV)  } \right.\notag\\
		& \left.- \eta \lp \varrho+iV\rp^{-1} \, \sum\limits_{n \in \Z} \erf \lp i \lp n+\frac{r}{2N} \rp \sqrt{2 \pi N V}  \rp 
		e^{2 \pi i N \left(n+\frac{r}{2N}\right)^2 (\varrho+iV)  } \right) \notag \\
		 =& \sum\limits_{r=1}^{N-1} \chi_{j,r} (N,M) \; \eta \lp \varrho+iV\rp^{-1} \left( - \sum\limits_{n \in \Z} \erf \lp i \lp n+\frac{r}{2N} \rp \sqrt{2 \pi N V}  \rp 
		e^{2 \pi i N \left(n+\frac{r}{2N}\right)^2 (\varrho+iV)  } \right).
	\end{align}
\end{small}
\hspace{-0.25cm} We see that the first term of \eqref{obstruction as sum} cancels against the contribution of $\CA_{r,N}  ( \frac{a \tau + b}{c \tau + d} ) $, so we focus on the second term of \eqref{obstruction as sum} and define
\begin{equation*}
\CI_{j,N,\varrho} (\varrho + iV) \coloneqq  
- \sum_{n \in \Z}  \erf \lp i \lp n+\frac{j}{2N} \rp \sqrt{2 \pi N V}  \rp 
e^{2 \pi i N \left(n+\frac{j}{2N}\right)^2 (\varrho+iV)  } ,
\end{equation*}
which is basically our error of modularity plus the holomorphic part of our function.
Using the identity (for $s \in \R \setminus \{ 0\}$ and $\operatorname{Re}(V) > 0$)
\begin{small}
	\begin{align*}
e^{- \pi s^2 V} \erf \lp i s \sqrt{\pi V} \rp =& 
- \frac{i}{\pi} \, \mathrm{P.V.} \int_{-\infty}^\infty \frac{e^{-\pi V x^2}}{x-s} \, dx 
 \coloneqq -\frac{i}{\pi} \lim\limits_{\varepsilon\rightarrow0^+} \left(\int_{-\infty}^{s-\varepsilon} \frac{e^{-\pi V x^2}}{x-s} \, dx + \int_{s+\varepsilon}^\infty \frac{e^{-\pi V x^2}}{x-s} \, dx\right),
\end{align*}
\end{small}
\hspace{-0.25cm} we obtain
\begin{equation}\label{representation of I}
\CI_{j,N,\varrho} (\varrho + iV) = \frac{i}{\pi} \sum_{n \in \Z} e^{2 \pi i N \left(n+\frac{j}{2N}\right)^2\varrho }
 \, \mathrm{P.V.}
 \int_{-\infty}^\infty \frac{e^{-2 \pi N V x^2}}{x-\lp n + \frac{j}{2N} \rp } \,dx.
\end{equation}

\subsection{Splitting of the Mordell-type integral}\label{subsection: Splitting of the Mordell-type integral}
Let $\varrho = \frac{h'}{k}$ with $h',k \in \Z$, $\mathrm{gcd}(h',k)=1$, and $k>0$. For a real number $d$ with $0\leq d<N$ such that $2\sqrt{dN}\notin\Z\backslash\{0\}$ we split $\CI_{j,N,\frac{h'}{k}}$ as follows
\begin{equation*}
e^{2 \pi d V} \CI_{j,N,\frac{h'}{k}} \lp \frac{h'}{k} + iV \rp = \CI_{j,N,\frac{h'}{k},d}^* \lp \frac{h'}{k} + iV \rp+\CI_{j,N,\frac{h'}{k},d}^e \lp \frac{h'}{k} + iV \rp,
\end{equation*}
where
\begin{align}\label{representation of I^*}
\CI_{j,N,\frac{h'}{k},d}^* \lp \frac{h'}{k} + iV \rp &= \frac{i}{\pi}  e^{2 \pi d V} \sum_{n \in \Z} e^{2 \pi i N \left(n+\frac{j}{2N}\right)^2\frac{h'}{k} }
\, \mathrm{P.V.} 
\int_{-\sqrt{\frac dN}}^{\sqrt{\frac dN}} \frac{e^{-2 \pi N V x^2}}{x-\lp n + \frac{j}{2N} \rp } \, dx, \\ 
\label{representation of I^e}
\CI_{j,N,\frac{h'}{k},d}^e \lp \frac{h'}{k} + iV \rp &= \frac{i}{\pi}  e^{2 \pi d V} \sum_{n \in \Z} e^{2 \pi i N \left(n+\frac{j}{2N}\right)^2\frac{h'}{k} }
\, \mathrm{P.V.} 
\int_{|x|\geq\sqrt{\frac dN}} \frac{e^{-2 \pi N V x^2}}{x-\lp n + \frac{j}{2N} \rp } \, dx  . 
\end{align}
Note that the assumption $2\sqrt{dN}\notin\Z\backslash\{0\}$ ensures the well-definedness of the principal value integral, since we avoid having poles on the boundary.

\section{Proof of Theorem \ref{Thm Kloosterman bound}}\label{section: Bound Kloosterman sum}
In this section we prove Theorem \ref{Thm Kloosterman bound} by using a bound of Malishev, which we stated as Lemma \ref{Lemma: Malishev} in the introduction of this paper.

We note that $K_{k,j,N}(n,r,\kappa)$ from \eqref{Kloostermansum} is well-defined and a Kloosterman sum of modulus $k$, which follows from a lengthy but straightforward calculation using the Chinese Remainder Theorem, quadratic reciprocity, and some formulas on the Kronecker symbol. Thus we can rewrite it as 
\begin{small}
	\begin{align*}
		K_{k,j,N}(n,r,\kappa)=& \sum_{\substack{h\pmod{k} \\ \gcd(h,k)=1 }}  \chi_{j,r} (N,M_{h,k})  \exp\left(-\frac{2\pi i}{k}\left(\left(n+\frac{j^2}{4N}-\frac{1}{24}\right)h - \left(N\left(\kappa+\frac{r}{2N}\right)^2-\frac{1}{24}\right)h'\right) \right).
	\end{align*}
\end{small}

	Note that for even $k$ we have
	\begin{small}
		\begin{align*}
			\chi_{j,r}(N,M_{h,k})
			=& \left(\frac{k}{-h}\right) \exp\left(\frac{\pi i}{12}\left(h'k\left(1-(-h)^2\right)+(-h)\left(-\frac{hh'+1}{k}-k+3\right)-3\right)\right) \\
			&\times \exp\left( \frac{3\pi i}{4}
			\right) \sqrt{\frac{2}{Nk}} \sum_{s=0}^{k-1} \exp\left(-\frac{\pi i}{2Nk} \left(-h(2Ns+j)^2+h'r^2\right)\right) \sin\left(\frac{\pi r(2Ns +j)}{Nk}\right),
		\end{align*}
	\end{small}
	\hspace{-0.25cm} while for odd $k$ we have
	\begin{small}
		\begin{align*}
			\chi_{j,r} (N,M_{h,k})  
			=& \left(\frac{-h}{k}\right) \exp\left(\frac{\pi i}{12} \left((h'-h)k-\frac{hh'+1}{k}h\left(k^2-1\right)-3k\right)\right) \\
			&\times \exp\left( \frac{3\pi i}{4}
			\right) \sqrt{\frac{2}{Nk}}\sum_{s=0}^{k-1} \exp\left(-\frac{\pi i}{2Nk} \left(-h(2Ns+j)^2+h'r^2\right)\right) \sin\left(\frac{\pi r(2Ns +j)}{Nk}\right) .
		\end{align*}
	\end{small}
	
	The strategy of the proof is to rewrite our Kloosterman sum into a sort of \textit{Sali\'e sum} 
	$$K_{k,j,N}(n,r,\kappa) = \epsilon(k,j,N,r)   \sum_{\substack{h\pmod{Gk} \\ \gcd(h,Gk)=1}} \left(\frac{h}{\rho}\right)\exp\left(\frac{2\pi i}{Gk}\left(\mu_* h-\nu_* [h]_{Gk}'\right)\right) ,$$ 
	where $\mu_*,\nu_*\in\Z$, $G\in\N$, $\rho\in\N$ odd such that all his prime divisors divide $Gk$, $[h]_{Gk}'$ the negative modular inverse of $h$ modulo $Gk$, and some $\epsilon(k,j,N,r)=O_N(1)$. Then we bound it using \cite[equation (12)]{knopp1964kloosterman}. Note that we use the $[\cdot]_{\cdot}$ notation from now on to denote the negative modular inverse of given modulus.

	We write
	\begin{align*}
		\sin\left(\frac{\pi r(2Ns +j)}{Nk}\right) = \frac{1}{2i}\left(\exp\left( \frac{\pi ir(2Ns +j)}{Nk}\right)-\exp\left(-\frac{\pi ir(2Ns +j)}{Nk}\right)\right),
	\end{align*}
	which yields that 
	\begin{align*}
			&\sum_{s=0}^{k-1} \exp\left(-\frac{\pi i}{2Nk} \left(-h(2Ns+j)^2+h'r^2\right)\right) \sin\left(\frac{\pi r(2Ns +j)}{Nk}\right) \notag \\ 
			=& \frac{1}{2i} \sum_{s=0}^{k-1} \left(\exp\left(\frac{2\pi i}{k} \left(hNs^2+\left(hj+r\right)s\right)\right)\exp\left(\frac{2\pi i}{4Nk}\left(hj^2-h'r^2+2rj\right)\right) \right. \notag \\
			&\left.\hspace{10mm}- \exp\left(\frac{2\pi i}{k} \left(hNs^2+\left(hj-r\right)s\right)\right)\exp\left(\frac{2\pi i}{4Nk}\left(hj^2-h'r^2-2rj\right)\right)\right).
		\end{align*}
	We additionally see that this equals 
	\begin{footnotesize}
		\begin{align}\label{equation: rewriting sum in character 2}
			\frac{1}{2i} \left( \exp\left(\frac{2\pi i}{4Nk}\left(hj^2-h'r^2+2rj\right)\right)G\left(hN,hj+r,k\right) - \exp\left(\frac{2\pi i}{4Nk}\left(hj^2-h'r^2-2rj\right)\right)G\left(hN,hj-r,k\right)  \right) ,
		\end{align}
	\end{footnotesize}
	\hspace{-0.25cm} where 
	$$G\left(a,b,c\right)\coloneqq \sum_{s=0}^{c-1} \exp\left(2\pi i\frac{as^2+bs}{c}\right)$$ 
	denotes the \textit{generalized quadratic Gauss sum}\footnote{Note that this sum is well-defined for any $a,c\in\N$ and $b\pmod{c}$.}.
	From this point on we have to look at odd, respectively even, $k$ seperately.

\subsection{Odd $k$}

We have that
\begin{align*}
	\chi_{j,r} (N,M_{h,k})  
	=& \left(\frac{-h}{k}\right) \sqrt{\frac{2}{Nk}} \exp\left(2\pi i \left(\frac{1}{24}\left((h'-h)k-\frac{hh'+1}{k}h\left(k^2-1\right)-3k\right)+\frac{3}{8}\right)\right) \notag \\
	&\times \sum_{s=0}^{k-1} \exp\left(-\frac{\pi i}{2Nk} \left(-h(2Ns+j)^2+h'r^2\right)\right) \sin\left(\frac{\pi r(2Ns +j)}{Nk}\right).
\end{align*}
Using \eqref{equation: rewriting sum in character 2} we can thus rewrite this as 
\begin{tiny}
	\begin{align}\label{equation: character to exponential odd step 1}
		\chi_{j,r} (N,M_{h,k})  
		=& -i \left(\frac{-h}{k}\right) \sqrt{\frac{1}{2Nk}} \exp\left(2\pi i \left(\frac{1}{24}\left((h'-h)k-\frac{hh'+1}{k}h\left(k^2-1\right)-3k\right)+\frac{3}{8}+ \frac{1}{4Nk}\left(hj^2-h'r^2+2rj\right)\right)\right) \notag \\
		& \times G\left(hN,hj+r,k\right)\notag \\
		& +i \left(\frac{-h}{k}\right) \sqrt{\frac{1}{2Nk}} \exp\left(2\pi i \left(\frac{1}{24}\left((h'-h)k-\frac{hh'+1}{k}h\left(k^2-1\right)-3k\right)+\frac{3}{8}+\frac{1}{4Nk}\left(hj^2-h'r^2-2rj\right)\right)\right) \notag\\
		& \times G\left(hN,hj-r,k\right).
	\end{align}
\end{tiny}

Note that $\gcd(Nh,k)=\gcd(N,k)$, since $\gcd(h,k)=1$. For odd $k$ we obtain that
\begin{scriptsize}
	\begin{align}\label{Gauss sum odd k}
		& G(hN,hj\pm r,k) \notag \\
		=& \begin{cases}
			0 & \text{ if } \gcd(N,k)>1, \text{ and } \gcd(N,k)\nmid (hj \pm r),\\
			\gcd(N,k) G\left(\frac{Nh}{\gcd(N,k)}, \frac{hj \pm r}{\gcd(N,k)},\frac{k}{\gcd(N,k)}\right) & \text{ if } \gcd(N,k)>1, \text{ and } \gcd(N,k)\mid (hj \pm r),\\
			\varepsilon_{k} \sqrt{k} \left(\frac{Nh}{k}\right) \exp\left(-2\pi i \frac{\psi\left(Nh\right)\left(hj\pm r\right)^2}{k}\right) & \text{ if } \gcd(N,k)=1,
		\end{cases} \notag \\
		=& \begin{cases}
			0 & \text{ if }  \gcd(N,k)\nmid (hj \pm r),\\
			\gcd(N,k)\varepsilon_{\frac{k}{\gcd(N,k)}} \sqrt{\frac{k}{\gcd(N,k)}} \left(\frac{\frac{Nh}{\gcd(N,k)}}{\frac{k}{\gcd(N,k)}}\right) \exp\left(-2\pi i \frac{\psi^*\left(\frac{Nh}{\gcd(N,k)}\right)\left(\frac{hj\pm r}{\gcd(N,k)}\right)^2}{\frac{k}{\gcd(N,k)}}\right) & \text{ otherwise},
		\end{cases}
	\end{align}
\end{scriptsize}
\hspace{-0.25cm} where $\psi(a)$, respectively $\psi^*(a)$, is some number satisfying\footnote{Note that $\psi(a)$, respectively $\psi^*(a)$, exists, since we assumed that $k$, and thus $\frac{k}{\gcd(N,k)}$, is odd and that $\gcd(Nh,k)=1$ by assumption of the first case, respectively $\gcd(Nh,\frac{k}{\gcd(N,k)})=1$.}  $4\psi(a)a\equiv 1 \pmod{k}$, respectively  $4\psi^*(a)a\equiv 1 \quad\!\!(\!\!\!\mod{\frac{k}{\gcd(N,k)}})$, and 
\begin{align} \label{definition epsilon_m}
	\varepsilon_m \coloneqq \begin{cases}
		1 & \text{ if } m \equiv 1 \pmod{4},\\
		i & \text{ if } m \equiv 3 \pmod{4},
	\end{cases}
\end{align}
for every odd integer $m$. 
We can thus rewrite \eqref{equation: character to exponential odd step 1}  as
\begin{tiny}
	\begin{align}\label{equation: rewrite exponential odd step 0}
		\chi_{j,r} (N,M_{h,k})  
		=& -i \left(\frac{-h}{k}\right) \sqrt{\frac{1}{2Nk}} \exp\left(2\pi i \left(\frac{1}{24}\left((h'-h)k-\frac{hh'+1}{k}h\left(k^2-1\right)-3k\right)+\frac{3}{8}+ \frac{1}{4Nk}\left(hj^2-h'r^2+2rj\right)\right)\right) \notag \\
		& \times \gcd(N,k) \varepsilon_{\frac{k}{\gcd(N,k)}} \sqrt{\frac{k}{\gcd(N,k)}} \left(\frac{\frac{Nh}{\gcd(N,k)}}{\frac{k}{\gcd(N,k)}}\right) \exp\left(-2\pi i \frac{\psi^*\left(\frac{Nh}{\gcd(N,k)}\right)\left(\frac{hj+r}{\gcd(N,k)}\right)^2}{\frac{k}{\gcd(N,k)}}\right) \delta_{\gcd(N,k)\mid (hj+r)} \notag \\
		& +i \left(\frac{-h}{k}\right) \sqrt{\frac{1}{2Nk}} \exp\left(2\pi i \left(\frac{1}{24}\left((h'-h)k-\frac{hh'+1}{k}h\left(k^2-1\right)-3k\right)+\frac{3}{8}+\frac{1}{4Nk}\left(hj^2-h'r^2-2rj\right)\right)\right) \notag \\
		& \times \gcd(N,k) \varepsilon_{\frac{k}{\gcd(N,k)}} \sqrt{\frac{k}{\gcd(N,k)}} \left(\frac{\frac{Nh}{\gcd(N,k)}}{\frac{k}{\gcd(N,k)}}\right) \exp\left(-2\pi i \frac{\psi^*\left(\frac{Nh}{\gcd(N,k)}\right)\left(\frac{hj-r}{\gcd(N,k)}\right)^2}{\frac{k}{\gcd(N,k)}}\right) \delta_{\gcd(N,k)\mid (hj-r)} ,
	\end{align}
\end{tiny}
\hspace{-0.25cm} using
 \begin{align*}
 	\delta_{\text{condition}} \coloneqq \begin{cases}
1 & \text{ if this condition is true},\\
0 & \text{ otherwise,}
\end{cases}
 \end{align*}
here and throughout the rest of the paper.
By definition we have that 
$$-4a\psi^*(a)\equiv -1\pmod{\frac{k}{\gcd(N,k)}},$$ 
which gives us that 
$$\psi^*(a)=[-4a]_{\frac{k}{\gcd(N,k)}}'.$$ 
Using that
$[ab]_x' = -[a]_x'[b]_x',$
for any modulus $x\in\N$ and arbitrary $a,b\in\N$, 
we obtain that $\psi^*(a) =[4]_{\frac{k}{\gcd(N,k)}}'[a]_{\frac{k}{\gcd(N,k)}}'$ and thus
$$ \psi^*\left(\frac{Nh}{\gcd(N,k)}\right) =[4]_{\frac{k}{\gcd(N,k)}}' \left[\frac{Nh}{\gcd(N,k)}\right]_{\frac{k}{\gcd(N,k)}}' = -[4]_{\frac{k}{\gcd(N,k)}}' \left[\frac{N}{\gcd(N,k)}\right]_{\frac{k}{\gcd(N,k)}}' [h]_{\frac{k}{\gcd(N,k)}}'.$$
Note that \eqref{equation: rewrite exponential odd step 0} is well-defined for $\left[a\right]'_{\frac{k}{\gcd(N,k)}}$ a solution of $a\left[a\right]'_{\frac{k}{\gcd(N,k)}} \equiv -1 ~(\!\!\!\mod{\frac{k}{\gcd(N,k)}})$, since we have that 
\begin{small}
	\begin{align*}
		\exp\left(-2\pi i \frac{\psi^*\left(\frac{Nh}{\gcd(N,k)}\right)\left(\frac{hj\pm r}{\gcd(N,k)}\right)^2}{\frac{k}{\gcd(N,k)}}\right) =  \exp\left(-2\pi i \frac{-[4]_{\frac{k}{\gcd(N,k)}}' \left[\frac{N}{\gcd(N,k)}\right]_{\frac{k}{\gcd(N,k)}}' [h]_{\frac{k}{\gcd(N,k)}}'\left(\frac{hj\pm r}{\gcd(N,k)}\right)^2}{\frac{k}{\gcd(N,k)}}\right) 
	\end{align*} 
\end{small}
\hspace{-0.25cm} is invariant under any shifts by $\frac{k}{\gcd(N,k)}$, since $\frac{hj\pm r}{\gcd(N,k)}\in\Z$.

For simplicity we stick to the notation $[h]_{k}'=h'$. We obtain 
\begin{tiny}
	\begin{align*}
		\chi_{j,r}& (N,M_{h,k}) \\
		=& -i  \varepsilon_{\frac{k}{\gcd(N,k)}}  \left(\frac{-h}{k}\right) \left(\frac{\frac{Nh}{\gcd(N,k)}}{\frac{k}{\gcd(N,k)}}\right) \sqrt{\frac{\gcd(N,k)}{2N}} \exp\left(-2\pi i \frac{-[4]_{\frac{k}{\gcd(N,k)}}' \left[\frac{N}{\gcd(N,k)}\right]_{\frac{k}{\gcd(N,k)}}' [h]_{\frac{k}{\gcd(N,k)}}'\left(\frac{hj+ r}{\gcd(N,k)}\right)^2}{\frac{k}{\gcd(N,k)}}\right)   \notag \\
		& \times  \exp\left(2\pi i \left(\frac{1}{24}\left((h'-h)k-\frac{hh'+1}{k}h\left(k^2-1\right)-3k\right)+\frac{3}{8}+ \frac{1}{4Nk}\left(hj^2-h'r^2+2rj\right)\right)\right) \delta_{\gcd(N,k)\mid (hj+r)}  \notag \\
		& +i \varepsilon_{\frac{k}{\gcd(N,k)}}  \left(\frac{-h}{k}\right) \left(\frac{\frac{Nh}{\gcd(N,k)}}{\frac{k}{\gcd(N,k)}}\right)  \sqrt{\frac{ \gcd(N,k)}{2N}} \exp\left(-2\pi i \frac{-[4]_{\frac{k}{\gcd(N,k)}}' \left[\frac{N}{\gcd(N,k)}\right]_{\frac{k}{\gcd(N,k)}}' [h]_{\frac{k}{\gcd(N,k)}}'\left(\frac{hj- r}{\gcd(N,k)}\right)^2}{\frac{k}{\gcd(N,k)}}\right)   \notag \\
		& \times   \exp\left(2\pi i \left(\frac{1}{24}\left((h'-h)k-\frac{hh'+1}{k}h\left(k^2-1\right)-3k\right)+\frac{3}{8}+\frac{1}{4Nk}\left(hj^2-h'r^2-2rj\right)\right)\right) \delta_{\gcd(N,k)\mid (hj-r)} 
	\end{align*}
\end{tiny}
\hspace{-0.25cm} and see that
\begin{align*}
	\left(\frac{-h}{k}\right) \left(\frac{\frac{Nh}{\gcd(N,k)}}{\frac{k}{\gcd(N,k)}}\right) = \left(\frac{-h}{k}\right) \left(\frac{-h}{\frac{k}{\gcd(N,k)}}\right) \left(\frac{\frac{-N}{\gcd(N,k)}}{\frac{k}{\gcd(N,k)}}\right) =  \left(\frac{-h}{\gcd(N,k)}\right)\left(\frac{\frac{-N}{\gcd(N,k)}}{\frac{k}{\gcd(N,k)}}\right).
\end{align*}
Therefore our Kloosterman sum equals 
\begin{tiny}
	\begin{align*}
		&K_{k,j,N}(n,r,\kappa) \\
		=&  i  \varepsilon_{\frac{k}{\gcd(N,k)}}  \left(\frac{\frac{-N}{\gcd(N,k)}}{\frac{k}{\gcd(N,k)}}\right) \sqrt{\frac{\gcd(N,k)}{2N}} \exp\left(\frac{2\pi i}{24k}\left(-3k^2 +9k \right)\right)  \sum_{\substack{h\pmod{k} \\ \gcd(h,k)=1 }} \left(\frac{-h}{\gcd(N,k)}\right) \\
		& \times  \exp\left(\frac{2\pi i}{24k}\left(\left(-24n+2-2k^2\right)h - \left(-24N\kappa^2-24\kappa r +1- k^2\right)h'\right) \right) \\
		& \times \left( \delta_{\gcd(N,k)\mid (hj-r)} \exp\left(2\pi i \frac{[4]_{\frac{k}{\gcd(N,k)}}' \left[\frac{N}{\gcd(N,k)}\right]_{\frac{k}{\gcd(N,k)}}' [h]_{\frac{k}{\gcd(N,k)}}'\left(\frac{hj- r}{\gcd(N,k)}\right)^2}{\frac{k}{\gcd(N,k)}}\right) \exp\left(\frac{2\pi i}{24k}\left(-h^2h'k^2+h^2h'-\frac{12rj}{N}\right)\right) \right. \\
		& \left. - \delta_{\gcd(N,k)\mid (hj+r)} \exp\left(2\pi i \frac{[4]_{\frac{k}{\gcd(N,k)}}' \left[\frac{N}{\gcd(N,k)}\right]_{\frac{k}{\gcd(N,k)}}' [h]_{\frac{k}{\gcd(N,k)}}'\left(\frac{hj+ r}{\gcd(N,k)}\right)^2}{\frac{k}{\gcd(N,k)}}\right) \exp\left(\frac{2\pi i}{24k}\left(-h^2h'k^2+h^2h'+\frac{12rj}{N}\right)\right) \right).
	\end{align*}
\end{tiny}

We already saw that the following is well-defined and now observe that
\begin{align*}
	&\exp\left(2\pi i \frac{[4]_{\frac{k}{\gcd(N,k)}}' \left[\frac{N}{\gcd(N,k)}\right]_{\frac{k}{\gcd(N,k)}}' [h]_{\frac{k}{\gcd(N,k)}}'\left(\frac{hj\pm r}{\gcd(N,k)}\right)^2}{\frac{k}{\gcd(N,k)}}\right) \notag \\
	=& \exp\left(\frac{2\pi i}{k\gcd(N,k)} \left([4]_{\frac{k}{\gcd(N,k)}}' \left[\frac{N}{\gcd(N,k)}\right]_{\frac{k}{\gcd(N,k)}}' [h]_{\frac{k}{\gcd(N,k)}}'\left(h^2j^2\pm 2hjr +r^2\right)\right) \right).
\end{align*}

Choose $[h]_{\frac{k}{\gcd(N,k)}}'=h'$ from now on\footnote{Note that $hh'\equiv -1 \pmod{k}$ implies that $hh'\equiv -1 ~(\!\!\!\mod{\frac{k}{\gcd(N,k)}})$, since $\frac{k}{\gcd(N,k)}\mid k$. Thus $h'$ is a possible choice for $\left[h\right]_{\frac{k}{\gcd(N,k)}}'$.}. Let $x\in\N$ such that $\gcd(x,h)=1$ (note that this condition is necessary to make sure that the negative modular inverse is well-defined) and $[h]_{xk}'$ the negative modular inverse of $h$ modulo $xk$, i.e.\@, 
$$h[h]_{xk}'\equiv -1 \pmod{xk}.$$ 
Then we see that we also have $h[h]_{xk}'\equiv -1 \pmod{k}$, since $k\mid xk$. This yields that 
$$h' \equiv [h]_{xk}' \pmod{k} .$$
Thus we can choose $h'$ such that $hh'\equiv -1 \pmod{xk}.$
Taking $x=\gcd(N,k)$
we obtain
\begin{tiny}
	\begin{align*}
		K_{k,j,N}&(n,r,\kappa) \\
		=& i  \varepsilon_{\frac{k}{\gcd(N,k)}}  \left(\frac{\frac{-N}{\gcd(N,k)}}{\frac{k}{\gcd(N,k)}}\right) \left(\frac{-1}{\gcd(N,k)}\right) \sqrt{\frac{\gcd(N,k)}{2N}} \exp\left(\frac{2\pi i}{24k}\left(-3k^2 +9k \right)\right)  \sum_{\substack{h\pmod{k} \\ \gcd(h,k)=1 }} \left(\frac{h}{\gcd(N,k)}\right) \\
		& \times  \exp\left(\frac{2\pi i}{24\gcd(N,k)k}\left(\left(\left(-24n+2-2k^2\right)\gcd(N,k) -24j^2[4]_{\frac{k}{\gcd(N,k)}}' \left[\frac{N}{\gcd(N,k)}\right]_{\frac{k}{\gcd(N,k)}}'\right)h \right. \right. \\
		&\left.\left. \hspace{2.5cm} - \left(\left(-24N\kappa^2-24\kappa r +1- k^2\right)\gcd(N,k) -24r^2[4]_{\frac{k}{\gcd(N,k)}}' \left[\frac{N}{\gcd(N,k)}\right]_{\frac{k}{\gcd(N,k)}}' \right)[h]_{k\gcd(N,k)}'\right) \right) \\
		& \times \left( \delta_{\gcd(N,k)\mid (hj-r)} \exp\left(\frac{2\pi i}{k\gcd(N,k)} \left([4]_{\frac{k}{\gcd(N,k)}}' \left[\frac{N}{\gcd(N,k)}\right]_{\frac{k}{\gcd(N,k)}}'  2jr \right) \right) \right. \\
		& \left. \times \exp\left(\frac{2\pi i}{24k}\left(h^2[h]_{k\gcd(N,k)}'\left(1-k^2\right)-\frac{12rj}{N}\right)\right) \right. \\
		& \left. - \delta_{\gcd(N,k)\mid (hj+r)} \exp\left(\frac{2\pi i}{k\gcd(N,k)} \left(-[4]_{\frac{k}{\gcd(N,k)}}' \left[\frac{N}{\gcd(N,k)}\right]_{\frac{k}{\gcd(N,k)}}'  2jr \right) \right) \right. \\
		& \left. \times \exp\left(\frac{2\pi i}{24k}\left(h^2[h]_{k\gcd(N,k)}'\left(1-k^2\right)+\frac{12rj}{N}\right)\right) \rule{0pt}{0.5cm}\right).
	\end{align*}
\end{tiny}

We now need to split into two cases, $3\nmid k$ and $3\mid k$. In the first case we have $1-k^2\equiv 0\pmod{24}$. Thus we obtain\footnote{Using that $h[h]_{k\gcd(N,k)}'\equiv -1 \pmod{k}$ since $k \mid (k\gcd(N,k))$.}
	\begin{align*}
		K_{k,j,N}(n,r,\kappa) = K_{k,j,N,+}(n,r,\kappa) + K_{k,j,N,-}(n,r,\kappa),
	\end{align*}
with 
\begin{scriptsize}
	\begin{align*}
		K_{k,j,N,\pm}(n,r,\kappa) \coloneqq& \mp i  \varepsilon_{\frac{k}{\gcd(N,k)}}  \left(\frac{\frac{-N}{\gcd(N,k)}}{\frac{k}{\gcd(N,k)}}\right) \left(\frac{-1}{\gcd(N,k)}\right) \sqrt{\frac{\gcd(N,k)}{2N}} \exp\left(\frac{2\pi i}{24k}\left(-3k^2 +9k \right)\right)  \notag \\
		&   \times  \exp\left(\frac{2\pi i}{24k\gcd(N,k)} \left(\mp48jr[4]_{\frac{k}{\gcd(N,k)}}' \left[\frac{N}{\gcd(N,k)}\right]_{\frac{k}{\gcd(N,k)}}'   \pm \frac{12rj}{N}\gcd(N,k) \right) \right) \notag \\
		&\times \sum_{\substack{h\pmod{k} \\ \gcd(h,k)=1 }} \left(\frac{h}{\gcd(N,k)}\right) \delta_{\gcd(N,k)\mid (hj\pm r)} \\
		& \times  \exp\left(\frac{2\pi i}{24\gcd(N,k)k}\left(\left(\left(-24n+1-k^2\right)\gcd(N,k) -24j^2[4]_{\frac{k}{\gcd(N,k)}}' \left[\frac{N}{\gcd(N,k)}\right]_{\frac{k}{\gcd(N,k)}}'\right)h \right. \right. \notag \\
		&\left.\left. \hspace{0.8cm} - \left(\left(-24N\kappa^2-24\kappa r +1- k^2\right)\gcd(N,k) -24r^2[4]_{\frac{k}{\gcd(N,k)}}' \left[\frac{N}{\gcd(N,k)}\right]_{\frac{k}{\gcd(N,k)}}' \right)[h]_{k\gcd(N,k)}'\right) \right).
	\end{align*}
\end{scriptsize}

We set
\begin{tiny}
	\begin{align*}
		K_{k,j,N,\pm}&(n,r,\kappa)\\
		\eqqcolon& \epsilon_{o,\pm}(k,j,N,r) \frac{1}{\gcd(N,k)}  \sum_{\substack{h\pmod{\gcd(N,k)k} \\ \gcd(h,\gcd(N,k)k)=1 }} \left(\frac{h}{\gcd(N,k)}\right)  \delta_{\gcd(N,k)\mid (hj\pm r)} \notag \\
		& \times  \exp\left(\frac{2\pi i}{\gcd(N,k)k}\left(\left(\left(-n+\frac{1-k^2}{24}\right)\gcd(N,k) -j^2[4]_{\frac{k}{\gcd(N,k)}}' \left[\frac{N}{\gcd(N,k)}\right]_{\frac{k}{\gcd(N,k)}}'\right)h \right. \right. \notag \\
		&\left.\left. \hspace{1.5cm} - \left(\left(-N\kappa^2-\kappa r +\frac{1- k^2}{24}\right)\gcd(N,k) -r^2[4]_{\frac{k}{\gcd(N,k)}}' \left[\frac{N}{\gcd(N,k)}\right]_{\frac{k}{\gcd(N,k)}}' \right)[h]_{k\gcd(N,k)}'\right) \right) \notag \\
		\eqqcolon& \epsilon_{o,\pm}(k,j,N,r) \frac{1}{\gcd(N,k)}  \sum_{\substack{h\pmod{\gcd(N,k)k} \\ \gcd(h,\gcd(N,k)k)=1 }} \left(\frac{h}{\gcd(N,k)}\right) \delta_{\gcd(N,k)\mid (hj\pm r)} \exp\left(\frac{2\pi i}{\gcd(N,k)k}\left( \mu_1h -\nu_1 [h]_{k\gcd(N,k)}'\right) \right)
	\end{align*}
\end{tiny}
\hspace{-0.25cm} and note that, by orthogonality of roots of unity, we have
\begin{align*}
	\delta_{\gcd(N,k)\mid (hj\pm r)} = \frac{1}{\gcd(N,k)} \sum_{s=0}^{\gcd(N,k)-1} \exp\left(2\pi i \frac{(hj\pm r)s}{\gcd(N,k)}\right),
\end{align*}
which finally gives us that
\begin{footnotesize}
	\begin{align}\label{equation: Kloosterman 3nmidk odd}
		K_{k,j,N,\pm}(n,r,\kappa) =& \epsilon_{o,\pm}(k,j,N,r) \frac{1}{\gcd(N,k)^2} \sum_{s=0}^{\gcd(N,k)-1} \exp\left(\pm 2\pi i \frac{rs}{\gcd(N,k)}\right) \notag \\
		& \times  \sum_{\substack{h\pmod{\gcd(N,k)k} \\ \gcd(h,\gcd(N,k)k)=1 }} \left(\frac{h}{\gcd(N,k)}\right) \exp\left(\frac{2\pi i}{\gcd(N,k)k}\left( \left(\mu_1+js k\right)h -\nu_1 [h]_{k\gcd(N,k)}'\right) \right).
	\end{align}
\end{footnotesize}

In the second case we have $1-k^2\equiv 0\pmod{8}$ and $3\nmid h$. Thus, choosing $[h]_{k\gcd(N,k)}'$ such that $h[h]_{k\gcd(N,k)}' \equiv -1 \pmod{3k\gcd(N,k)}$ analogously to above, we obtain\footnote{Using that $h[h]_{3k\gcd(N,k)}'\equiv -1 \pmod{3k}$ since $(3k) \mid (3k\gcd(N,k))$.}
\begin{tiny}
	\begin{align*}
		K_{k,j,N}(n,r,\kappa) =& i  \varepsilon_{\frac{k}{\gcd(N,k)}}  \left(\frac{\frac{-N}{\gcd(N,k)}}{\frac{k}{\gcd(N,k)}}\right) \left(\frac{-1}{\gcd(N,k)}\right) \sqrt{\frac{\gcd(N,k)}{2N}} \exp\left(\frac{2\pi i}{24k}\left(-3k^2 +9k \right)\right) \sum_{\substack{h\pmod{k} \\ \gcd(h,k)=1 }} \left(\frac{h}{\gcd(N,k)}\right) \notag \\
		& \times \left( \delta_{\gcd(N,k)\mid (hj-r)} \exp\left(\frac{2\pi i}{24k\gcd(N,k)} \left(48jr[4]_{\frac{k}{\gcd(N,k)}}' \left[\frac{N}{\gcd(N,k)}\right]_{\frac{k}{\gcd(N,k)}}'   -\frac{12rj}{N}\gcd(N,k) \right) \right)\right. \notag \\
		& \left. - \delta_{\gcd(N,k)\mid (hj+r)} \exp\left(\frac{2\pi i}{24k\gcd(N,k)} \left(-48jr[4]_{\frac{k}{\gcd(N,k)}}' \left[\frac{N}{\gcd(N,k)}\right]_{\frac{k}{\gcd(N,k)}}'  +\frac{12rj}{N}\gcd(N,k) \right) \right)  \right) \notag \\
		& \times  \exp\left(\frac{2\pi i}{24\gcd(N,k)k}\left(\left(\left(-24n+1-k^2\right)\gcd(N,k) -24j^2[4]_{\frac{k}{\gcd(N,k)}}' \left[\frac{N}{\gcd(N,k)}\right]_{\frac{k}{\gcd(N,k)}}'\right)h \right. \right. \notag \\
		&\left.\left. \hspace{1.5cm} - \left(\left(-24N\kappa^2-24\kappa r +1- k^2\right)\gcd(N,k) -24r^2[4]_{\frac{k}{\gcd(N,k)}}' \left[\frac{N}{\gcd(N,k)}\right]_{\frac{k}{\gcd(N,k)}}' \right)[h]_{3k\gcd(N,k)}'\right) \right) \notag \\
		\eqqcolon&  K^*_{k,j,N,-}(n,r,\kappa) + K^*_{k,j,N,+}(n,r,\kappa).
	\end{align*}
\end{tiny}	
\hspace{-0.25cm} Here we set
\begin{tiny}
	\begin{align*}
		K^*_{k,j,N,\pm}(n,r,\kappa)=&  \epsilon_{o,\pm}(k,j,N,r) \frac{1}{3\gcd(N,k)}  \sum_{\substack{h\pmod{3\gcd(N,k)k} \\ \gcd(h,3\gcd(N,k)k)=1 }} \left(\frac{h}{\gcd(N,k)}\right) \delta_{\gcd(N,k)\mid (hj+r)} \notag \\
		& \times  \exp\left(\frac{2\pi i}{3\gcd(N,k)k}\left(\left(\left(-3n+\frac{1-k^2}{8}\right)\gcd(N,k) -3j^2[4]_{\frac{k}{\gcd(N,k)}}' \left[\frac{N}{\gcd(N,k)}\right]_{\frac{k}{\gcd(N,k)}}'\right)h \right. \right. \notag \\
		&\left.\left. \hspace{1.5cm} - \left(\left(-3N\kappa^2-3\kappa r +\frac{1- k^2}{8}\right)\gcd(N,k) -3r^2[4]_{\frac{k}{\gcd(N,k)}}' \left[\frac{N}{\gcd(N,k)}\right]_{\frac{k}{\gcd(N,k)}}' \right)[h]_{3k\gcd(N,k)}'\right) \right) \notag \\
		\eqqcolon& \epsilon_{o,\pm}(k,j,N,r) \frac{1}{3\gcd(N,k)} \\
		&\times  \sum_{\substack{h\pmod{3\gcd(N,k)k} \\ \gcd(h,3\gcd(N,k)k)=1 }} \left(\frac{h}{\gcd(N,k)}\right) \delta_{\gcd(N,k)\mid (hj+r)} \exp\left(\frac{2\pi i}{3\gcd(N,k)k}\left( \mu_2h -\nu_2 [h]_{3k\gcd(N,k)}'\right) \right)
	\end{align*}
\end{tiny}
\hspace{-0.25cm} and, by orthogonality of roots of unity, we finally have
\begin{footnotesize}
	\begin{align}\label{equation: Kloosterman 3midk odd}
		K^*_{k,j,N,\pm}(n,r,\kappa) =& \epsilon_{o,\pm}(k,j,N,r) \frac{1}{3\gcd(N,k)^2} \sum_{s=0}^{\gcd(N,k)-1} \exp\left(\pm 2\pi i \frac{rs}{\gcd(N,k)}\right) \notag \\
		& \times  \sum_{\substack{h\pmod{3\gcd(N,k)k} \\ \gcd(h,3\gcd(N,k)k)=1 }} \left(\frac{h}{\gcd(N,k)}\right) \exp\left(\frac{2\pi i}{3\gcd(N,k)k}\left( \left(\mu_2+3js k\right)h -\nu_2 [h]_{3k\gcd(N,k)}'\right) \right).
	\end{align}
\end{footnotesize}

Since in \eqref{equation: Kloosterman 3nmidk odd} and \eqref{equation: Kloosterman 3midk odd} both sums over $h$ are of the required shape we can bound them using Malishev's result (see  Lemma \ref{Lemma: Malishev}) and obtain that they are
\begin{small}
	\begin{align*}
		\begin{cases}
			O\left((\gcd(N,k)k)^{\frac{1}{2}+\varepsilon} \min\left(\gcd\left(\mu_1+js k,\gcd(N,k)k\right)^\frac{1}{2}, \gcd\left(\nu_1,\gcd(N,k)k\right)^\frac{1}{2}\right)\right) & \text{ if } 3\nmid k, \\
			O\left((3\gcd(N,k)k)^{\frac{1}{2}+\varepsilon} \min\left(\gcd\left(\mu_2+ 3j s k,3\gcd(N,k)k\right)^\frac{1}{2}, \gcd\left(\nu_2,3\gcd(N,k)k\right)^\frac{1}{2}\right)\right) & \text{ if } 3\mid k,
		\end{cases} 
	\end{align*}
\end{small}
\hspace{-0.25cm} for $\varepsilon>0$.
We see that $\gcd(N,k)\leq N=O_N(1)$,
\begin{align*}
	&\min\left(\gcd\left(\mu_1+jsk,\gcd(N,k)k\right)^\frac{1}{2}, \gcd\left(\nu_1,\gcd(N,k)k\right)^\frac{1}{2}\right) = O_N\left(n^\frac{1}{2}\right),
\end{align*}
and 
\begin{align*}
	\min\left(\gcd\left(\mu_2+3jsk,3\gcd(N,k)k\right)^\frac{1}{2}, \gcd\left(\nu_2,3\gcd(N,k)k\right)^\frac{1}{2}\right) =  O_N\left(n^\frac{1}{2}\right).
\end{align*}
Thus we showed that
\begin{small}
	\begin{align*}
		K_{k,j,N,\pm}(n,r,\kappa) =&
		O_N\left( \left|\epsilon_{o,\pm}(k,j,N,r) \frac{1}{\gcd(N,k)^2} \right|\sum\limits_{s=0}^{\gcd(N,k)-1} \left|\exp\left(\pm 2\pi i \frac{rs}{\gcd(N,k)}\right)\right| n^\frac{1}{2}k^{\frac{1}{2}+\varepsilon} \right) \\
		=& O_N\left(n^\frac{1}{2}k^{\frac{1}{2}+\varepsilon} \right)
	\end{align*}
\end{small}
\hspace{-0.25cm} and analogously $K_{k,j,N,\pm}^*(n,r,\kappa) = O_N(n^\frac{1}{2}k^{\frac{1}{2}+\varepsilon} )$, which yields 
\begin{align*}
	K_{k,j,N}(n,r,\kappa) =& O_N\left(n^\frac{1}{2}k^{\frac{1}{2}+\varepsilon} \right)
\end{align*}
and finishes the proof for odd $k$.

\subsection{Even $k$}

	We go on with the case of even $k$ and have that
	\begin{small}
		\begin{align*}
			\chi_{j,r}&(N,M_{h,k}) \\
			=& -i \left(\frac{k}{-h}\right) \sqrt{\frac{1}{2Nk}} \exp\left(2 \pi i\left(\frac{1}{24}\left(h'k\left(1-(-h)^2\right)+(-h)\left(-\frac{hh'+1}{k}-k+3\right)-3\right)+\frac{3}{8}\right)\right) \\
			& \times \exp\left( \frac{2\pi i }{4Nk}\left(hj^2-h'r^2+2rj\right)\right)  G\left(hN,hj+r,k\right)\\
			& +i \left(\frac{k}{-h}\right) \sqrt{\frac{1}{2Nk}}\exp\left(2 \pi i\left(\frac{1}{24}\left(h'k\left(1-(-h)^2\right)+(-h)\left(-\frac{hh'+1}{k}-k+3\right)-3\right)+\frac{3}{8}\right)\right) \\
			& \times \exp\left( \frac{2\pi i }{4Nk}\left(hj^2-h'r^2-2rj\right)\right) G\left(hN,hj-r,k\right),
		\end{align*}
	\end{small}
	\hspace{-0.25cm} using \eqref{equation: rewriting sum in character 2}.  
	For even $k$ we can write $k=2^\nu \mu$ with $\nu\geq 1$ and $\mu$ odd. Using the multiplicativity of the generalized quadratic Gauss sum\footnote{For given $a,c,d \in\N$, $b\pmod{c}$ and $\gcd(c,d)=1$ we have that $ G(a,b,cd) = G(ac,b,d) G(ad,b,c).$}, we thus have that
	\begin{align*}
		& G(hN,hj\pm r,k) = G(hN,hj\pm r,2^\nu \mu) = G(hN2^\nu,hj\pm r, \mu) G(hN \mu,hj\pm r,2^\nu)  .
	\end{align*}
	Defining $\alpha \coloneqq \max(x:2^x\mid (hN\mu))=\max(x:2^x\mid N)$ we obtain
	\begin{tiny}
		\begin{align*}
			G(hN \mu,hj\pm r,2^\nu)& \\
			&\hspace{-0.8cm}= \begin{cases}
				2^\nu &\text{ if }\nu-\alpha=1 \text{ and } hj\pm r\not\equiv 0 \pmod{2}, \\
				2^{\frac{\nu+\alpha}{2}} (i+1) \left(\frac{-2^{\nu+\alpha}}{\frac{hN \mu}{2^\alpha}}\right) \varepsilon_{\frac{hN \mu}{2^\alpha}} \exp\left(-2\pi i \frac{\left[\frac{hN \mu}{2^\alpha}\right]_{2^{\nu+\alpha+2}}' \frac{(hj\pm r)^2}{4}}{2^{\nu+\alpha}}\right)  & \text{ if }\nu-\alpha>1 \text{ and } hj\pm r\equiv 0 \pmod{2^{\alpha+1}}, \\
				0 & \text{ otherwise}.
			\end{cases}
		\end{align*}
	\end{tiny}
	\hspace{-0.25cm} Noting that $\gcd(hN2^\nu,\mu)= \gcd(hN,\mu)=\gcd(N,\mu)$, $hj\pm r \equiv hj+ r \pmod{2}$, and combining this with \eqref{Gauss sum odd k} yields 
	\begin{tiny}
		\begin{align*}
			G(hN ,hj\pm r,2^\nu \mu)& \\
			&\hspace{-0.8cm} = \begin{cases}
				2^\nu A_{\pm} &\text{ if }\nu-\alpha=1,~ hj + r\not\equiv 0 \pmod{2}, \\
				&~\text{ and } \gcd(N,\mu) \mid (hj\pm r), \\
				2^{\frac{\nu+\alpha}{2}} A_\pm (i+1) \left(\frac{-2^{\nu+\alpha}}{\frac{hN \mu}{2^\alpha}}\right) \varepsilon_{\frac{hN \mu}{2^\alpha}}  \exp\left(-2\pi i \frac{\left[\frac{hN \mu}{2^\alpha}\right]_{2^{\nu+\alpha+2}}' \frac{(hj\pm r)^2}{4}}{2^{\nu+\alpha}}\right)   & \text{ if }\nu-\alpha>1,~ hj\pm r\equiv 0 \pmod{2^{\alpha+1}}, \\[-8pt]
				&~\text{ and } \gcd(N,\mu) \mid (hj\pm r), \\[7pt]
				0 & \text{ otherwise},
			\end{cases}
		\end{align*}
	\end{tiny}
	\hspace{-0.25cm} with 
	\begin{small}
		\begin{align*}
			A_{\pm}\coloneqq \gcd(N,\mu)\varepsilon_{\frac{\mu}{\gcd(N,\mu)}} \sqrt{\frac{\mu}{\gcd(N,\mu)}} \left(\frac{\frac{Nh2^\nu}{\gcd(N,\mu)}}{\frac{\mu}{\gcd(N,\mu)}}\right) \exp\left(-2\pi i \frac{\widetilde{\psi}\left(\frac{Nh2^\nu}{\gcd(N,\mu)}\right)\left(\frac{hj\pm r}{\gcd(N,\mu)}\right)^2}{\frac{\mu}{\gcd(N,\mu)}}\right),
		\end{align*}
	\end{small}
	\hspace{-0.25cm} $\varepsilon_{\frac{hN \mu}{2^\alpha}}$ and $\varepsilon_{\frac{\mu}{\gcd(N,\mu)}}$ as in \eqref{definition epsilon_m},
	and where $\widetilde{\psi}(a)$ is some number satisfying $4\widetilde{\psi}(a)a\equiv 1 ~(\!\!\!\mod{\frac{\mu}{\gcd(N,\mu)}})$.
	
	Note that in the first case we have that $\nu=\alpha+1$ which gives us that $2^\nu =2^{\alpha+1}\leq 2N$, and allows us to say that $2^\nu=O_N(1)$.
	
	Using that for even $k$ the $h$ we are summing over have to be odd we split our Kloosterman sum as follows
	\begin{footnotesize}
		\begin{align*}
			 K_{k,j,N}&(n,r,\kappa)\\
			=& \left(\delta_{\substack{\nu-\alpha=1 \\ j\not\equiv r \pmod{2}}}\left(\sum_{\substack{0\leq h < k \\ \gcd(h,k)=1 \\ \gcd(N,\mu)\mid (hj+r)}}+  \sum_{\substack{0\leq h < k \\ \gcd(h,k)=1 \\ \gcd(N,\mu)\mid (hj-r )}}\right) + \delta_{\nu-\alpha>1} \left(\sum_{\substack{0\leq h < k \\ \gcd(h,k)=1 \\ \gcd(N,\mu)2^{\alpha +1}\mid (hj+r) }}  +  \sum_{\substack{0\leq h < k \\ \gcd(h,k)=1 \\ \gcd(N,\mu)2^{\alpha +1}\mid (hj-r) }} \right) \right) \\
			& \times \chi_{j,r} (N,M_{h,k})  \zeta_{24k}^{\left(24N\left(\kappa+\frac{r}{2N}\right)^2-1\right)h'-24\left(n+\frac{j^2}{4N}-\frac{1}{24}\right)h  }\\
			\eqqcolon&  K_{k,j,N,1,+}(n,r,\kappa) + K_{k,j,N,1,-}(n,r,\kappa)  + K_{k,j,N,2,+}(n,r,\kappa)  + K_{k,j,N,2,-}(n,r,\kappa) \\
			\eqqcolon&  K_{k,j,N,1}(n,r,\kappa)  + K_{k,j,N,2}(n,r,\kappa).
		\end{align*}
	\end{footnotesize}

	For $K_{k,j,N,1}(n,r,\kappa)$ we can run a similar calculation as in the odd $k$ case.
	By definition we have that $-4a\widetilde{\psi}(a)\equiv -1~(\!\!\!\mod{\frac{\mu}{\gcd(N,\mu)}})$, which gives us that 
	$$\widetilde{\psi}(a)=[-4a]_{\frac{\mu}{\gcd(N,\mu)}}'=[4]_{\frac{\mu}{\gcd(N,\mu)}}'[a]_{\frac{\mu}{\gcd(N,\mu)}}'$$ 
	and thus
	$$ \widetilde{\psi}\left(\frac{Nh2^\nu}{\gcd(N,\mu)}\right) =[4]_{\frac{\mu}{\gcd(N,\mu)}}' \left[\frac{Nh2^\nu}{\gcd(N,\mu)}\right]_{\frac{\mu}{\gcd(N,\mu)}}' = -[4]_{\frac{\mu}{\gcd(N,\mu)}}' \left[\frac{N2^\nu}{\gcd(N,\mu)}\right]_{\frac{\mu}{\gcd(N,\mu)}}' [h]_{\frac{\mu}{\gcd(N,\mu)}}'.$$
	Note that 
	\begin{small}
		\begin{align*}
			\exp\left(-2\pi i \frac{\widetilde{\psi}\left(\frac{Nh2^\nu}{\gcd(N,\mu)}\right)\left(\frac{hj \pm r}{\gcd(N,\mu)}\right)^2}{\frac{\mu}{\gcd(N,\mu)}}\right) =  \exp\left(-2\pi i \frac{-[4]_{\frac{\mu}{\gcd(N,\mu)}}' \left[\frac{N2^\nu}{\gcd(N,\mu)}\right]_{\frac{\mu}{\gcd(N,\mu)}}' [h]_{\frac{\mu}{\gcd(N,\mu)}}'\left(\frac{hj \pm r}{\gcd(N,\mu)}\right)^2}{\frac{\mu}{\gcd(N,\mu)}}\right) 
		\end{align*} 
	\end{small}
	\hspace{-0.25cm} is well-defined for $\left[a\right]'_{\frac{\mu}{\gcd(N,\mu)}}$ a solution of $a\left[a\right]'_{\frac{\mu}{\gcd(N,\mu)}} \equiv -1 ~(\!\!\!\mod{\frac{\mu}{\gcd(N,\mu)}})$, since it is invariant under any shifts by $\frac{\mu}{\gcd(N,\mu)}$, because $\frac{hj\pm r}{\gcd(N,\mu)}\in\Z$ by assumption.
	
	For simplicity we stick to the notation $[h]_{k}'=h'$. For $K_{k,j,N,1}(n,r,\kappa)$ we obtain that
	\begin{scriptsize}
		\begin{align*}
			\chi_{j,r} &(N,M_{h,k}) \\
			=&  -i \varepsilon_{\frac{\mu}{\gcd(N,\mu)}} \left(\frac{k}{-h}\right) \left(\frac{\frac{Nh2^\nu}{\gcd(N,\mu)}}{\frac{\mu}{\gcd(N,\mu)}}\right) \sqrt{\frac{\gcd(N,\mu)2^\nu}{2N}} \exp\left(-2\pi i \frac{-[4]_{\frac{\mu}{\gcd(N,\mu)}}' \left[\frac{N2^\nu}{\gcd(N,\mu)}\right]_{\frac{\mu}{\gcd(N,\mu)}}' [h]_{\frac{\mu}{\gcd(N,\mu)}}'\left(\frac{hj + r}{\gcd(N,\mu)}\right)^2}{\frac{\mu}{\gcd(N,\mu)}}\right)  \\
			& \times   \exp\left(2 \pi i\left(\frac{1}{24}\left(h'k\left(1-h^2\right)-h\left(-\frac{hh'+1}{k}-k+3\right)-3\right)+\frac{3}{8}  + \frac{hj^2-h'r^2+2rj}{4Nk}\right)\right) \delta_{\gcd(N,\mu)\mid (hj+r)} \notag \\
			& +i \varepsilon_{\frac{\mu}{\gcd(N,\mu)}} \left(\frac{k}{-h}\right) \left(\frac{\frac{Nh2^\nu}{\gcd(N,\mu)}}{\frac{\mu}{\gcd(N,\mu)}}\right) \sqrt{\frac{\gcd(N,\mu)2^\nu}{2N}} \exp\left(-2\pi i \frac{-[4]_{\frac{\mu}{\gcd(N,\mu)}}' \left[\frac{N2^\nu}{\gcd(N,\mu)}\right]_{\frac{\mu}{\gcd(N,\mu)}}' [h]_{\frac{\mu}{\gcd(N,\mu)}}'\left(\frac{hj - r}{\gcd(N,\mu)}\right)^2}{\frac{\mu}{\gcd(N,\mu)}}\right) \\
			& \times   \exp\left(2 \pi i\left(\frac{1}{24}\left(h'k\left(1-h^2\right)-h\left(-\frac{hh'+1}{k}-k+3\right)-3\right)+\frac{3}{8}+ \frac{hj^2-h'r^2-2rj}{4Nk}\right)\right) \delta_{\gcd(N,\mu)\mid (hj-r)}.
		\end{align*}
	\end{scriptsize}
	\hspace{-0.25cm} Using quadratic reciprocity together with $\left(\frac{k}{-h}\right) = \sgn(k) \left(\frac{k}{h}\right) = \left(\frac{k}{h}\right)$ we have
	\begin{align*}
		\left(\frac{k}{-h}\right) \left(\frac{\frac{Nh2^\nu}{\gcd(N,\mu)}}{\frac{\mu}{\gcd(N,\mu)}}\right) =&  \left(\frac{2}{h}\right)^\nu  \left(\frac{\mu}{h}\right) \left(\frac{h}{\frac{\mu}{\gcd(N,\mu)}}\right) \left(\frac{\frac{N2^\nu}{\gcd(N,\mu)}}{\frac{\mu}{\gcd(N,\mu)}}\right) \\
		=& \left((-1)^{\frac{h^2-1}{8}}\right)^\nu  (-1)^{\frac{(\mu-1)(h-1)}{4}}  \left(\frac{h}{\gcd(N,\mu)}\right) \left(\frac{\frac{N2^\nu}{\gcd(N,\mu)}}{\frac{\mu}{\gcd(N,\mu)}}\right).
	\end{align*}
	Therefore 
	\begin{scriptsize}
		\begin{align*}
			& \hspace{-0.6cm}K_{k,j,N,1}(n,r,\kappa)\\
			=& i \varepsilon_{\frac{\mu}{\gcd(N,\mu)}} \left(\frac{\frac{N2^\nu}{\gcd(N,\mu)}}{\frac{\mu}{\gcd(N,\mu)}}\right) \sqrt{\frac{\gcd(N,\mu)2^\nu}{2N}} \exp\left(\frac{2 \pi i}{4}\right) \delta_{\substack{\nu-\alpha=1 \\ j\not\equiv r \pmod{2}}} \sum_{\substack{h\pmod{k} \\ \gcd(h,k)=1  }}  \left((-1)^{\frac{h^2-1}{8}}\right)^\nu  (-1)^{\frac{(\mu-1)(h-1)}{4}}  \left(\frac{h}{\gcd(N,\mu)}\right) \\
			& \times  \exp\left(\frac{2\pi i}{24k}\left(\left(-24n+2+k^2-3k\right)h - \left(-24N\kappa^2-24\kappa r+1-k^2\right)h'\right) \right) \\
			& \times \left( -  \delta_{\gcd(N,\mu)\mid (hj+r)} \exp\left(2\pi i \frac{[4]_{\frac{\mu}{\gcd(N,\mu)}}' \left[\frac{N2^\nu}{\gcd(N,\mu)}\right]_{\frac{\mu}{\gcd(N,\mu)}}' [h]_{\frac{\mu}{\gcd(N,\mu)}}'\left(\frac{hj + r}{\gcd(N,\mu)}\right)^2}{\frac{\mu}{\gcd(N,\mu)}}\right) \right. \notag \\
			& \left. \times   \exp\left(\frac{2 \pi i}{24k}\left(-h^2h'k^2+h^2h'+\frac{12rj}{N} \right)\right)  \right.  \notag \\
			& \left. + \delta_{\gcd(N,\mu)\mid (hj-r)} \exp\left(2\pi i \frac{[4]_{\frac{\mu}{\gcd(N,\mu)}}' \left[\frac{N2^\nu}{\gcd(N,\mu)}\right]_{\frac{\mu}{\gcd(N,\mu)}}' [h]_{\frac{\mu}{\gcd(N,\mu)}}'\left(\frac{hj - r}{\gcd(N,\mu)}\right)^2}{\frac{\mu}{\gcd(N,\mu)}}\right) \right. \notag \\
			& \left. \times \exp\left(\frac{2 \pi i}{24k}\left(-h^2h'k^2+h^2h' -\frac{12rj}{N} \right)\right)  \rule{0pt}{0.75cm}\right).
		\end{align*}
	\end{scriptsize}
	
	We already saw that the following is well-defined and now observe that
	\begin{align*}
		&\exp\left(2\pi i \frac{[4]_{\frac{\mu}{\gcd(N,\mu)}}' \left[\frac{N2^\nu}{\gcd(N,\mu)}\right]_{\frac{\mu}{\gcd(N,\mu)}}' [h]_{\frac{\mu}{\gcd(N,\mu)}}'\left(\frac{hj\pm r}{\gcd(N,\mu)}\right)^2}{\frac{\mu}{\gcd(N,\mu)}}\right) \\
		=& \exp\left(\frac{2\pi i}{k\gcd(N,\mu)} \left(2^\nu [4]_{\frac{\mu}{\gcd(N,\mu)}}' \left[\frac{N2^\nu}{\gcd(N,\mu)}\right]_{\frac{\mu}{\gcd(N,\mu)}}' [h]_{\frac{\mu}{\gcd(N,\mu)}}'\left(h^2j^2\pm 2hjr +r^2\right)\right) \right).
	\end{align*}
	
	Choose $[h]_{\frac{\mu}{\gcd(N,\mu)}}'=h'$ here\footnote{Note that $hh'\equiv -1 \pmod{k}$ implies that $hh'\equiv -1 ~(\!\!\!\mod{\frac{\mu}{\gcd(N,\mu)}})$, since $\frac{\mu}{\gcd(N,\mu)}\mid k$. Thus $h'$ is a possible choice for $\left[h\right]_{\frac{\mu}{\gcd(N,\mu)}}'$.}.
	Analogously to above we can choose $h'$ such that the equivalence $hh'\equiv -1 \pmod{xk}$ holds for some $x\in\N$ such that $\gcd(x,h)=1$.
	Taking $x=\gcd(N,\mu)$ we obtain
	\begin{scriptsize}
		\begin{align*}
			& \hspace{-0.6cm} K_{k,j,N,1}(n,r,\kappa) \\
			=& i \varepsilon_{\frac{\mu}{\gcd(N,\mu)}} \left(\frac{\frac{N2^\nu}{\gcd(N,\mu)}}{\frac{\mu}{\gcd(N,\mu)}}\right) \sqrt{\frac{\gcd(N,\mu)2^\nu}{2N}} \exp\left(\frac{2 \pi i}{4}\right) \delta_{\substack{\nu-\alpha=1 \\ j\not\equiv r \pmod{2}}} \sum_{\substack{h\pmod{k} \\ \gcd(h,k)=1 }} \left((-1)^{\frac{h^2-1}{8}}\right)^\nu  (-1)^{\frac{(\mu-1)(h-1)}{4}}  \left(\frac{h}{\gcd(N,\mu)}\right) \\
			& \times  \exp\left(\frac{2\pi i}{24k\gcd(N,\mu)}\left(\left(\left(-24n+2+k^2-3k\right)\gcd(N,\mu)- 24j^2 2^\nu [4]_{\frac{\mu}{\gcd(N,\mu)}}' \left[\frac{N2^\nu}{\gcd(N,\mu)}\right]_{\frac{\mu}{\gcd(N,\mu)}}'\right)h \right.\right. \\
			&\left. \left. \hspace{1.5cm} - \left(\left(-24N\kappa^2-24\kappa r+1-k^2\right)\gcd(N,\mu)-24r^2 2^\nu [4]_{\frac{\mu}{\gcd(N,\mu)}}' \left[\frac{N2^\nu}{\gcd(N,\mu)}\right]_{\frac{\mu}{\gcd(N,\mu)}}'\right)[h]_{k\gcd(N,\mu)}'\right) \right) \\
			& \times \left( - \delta_{\gcd(N,\mu)\mid (hj+r)}   \exp\left(\frac{2\pi i}{k\gcd(N,\mu)} \left(-2^{\nu+1}jr [4]_{\frac{\mu}{\gcd(N,\mu)}}' \left[\frac{N2^\nu}{\gcd(N,\mu)}\right]_{\frac{\mu}{\gcd(N,\mu)}}' \right) \right) \right. \\
			&\left. \times   \exp\left(\frac{2 \pi i}{24k}\left(h^2[h]_{k\gcd(N,\mu)}'(1-k^2)+\frac{12rj}{N} \right)\right)  \right.  \notag \\
			& \left. + \delta_{\gcd(N,\mu)\mid (hj-r)} \exp\left(\frac{2\pi i}{k\gcd(N,\mu)} \left(2^{\nu+1}jr [4]_{\frac{\mu}{\gcd(N,\mu)}}' \left[\frac{N2^\nu}{\gcd(N,\mu)}\right]_{\frac{\mu}{\gcd(N,\mu)}}' \right) \right)  \right. \\
			& \left. \times \exp\left(\frac{2 \pi i}{24k}\left(h^2[h]_{k\gcd(N,\mu)}'(1-k^2) -\frac{12rj}{N} \right)\right)  \right).
		\end{align*}
	\end{scriptsize}
	
	We now need to split into two cases, namely $3\nmid k$ and $3\mid k$. In the first case we obtain that $3\mid (k^2-1)$. Choosing $[h]_{k\gcd(N,\mu)}'$ such that $h[h]_{k\gcd(N,\mu)}' \equiv -1 \pmod{8k\gcd(N,\mu)}$, analogously to above, yields\footnote{We are allowed to do this since $\gcd(8,h)=1$, this is because we know that $h$ is odd. Additionally we use that $h[h]_{8k\gcd(N,\mu)}'\equiv -1 \pmod{8k}$ since $(8k) \mid (8k\gcd(N,\mu))$.}
	\begin{footnotesize}
		\begin{align*}
			K_{k,j,N,1}&(n,r,\kappa)\\
			=&  i \varepsilon_{\frac{\mu}{\gcd(N,\mu)}} \left(\frac{\frac{N2^\nu}{\gcd(N,\mu)}}{\frac{\mu}{\gcd(N,\mu)}}\right) \sqrt{\frac{\gcd(N,\mu)2^\nu}{2N}} \exp\left(\frac{2 \pi i}{4}\right) \delta_{\substack{\nu-\alpha=1 \\ j\not\equiv r \pmod{2}}} \notag \\
			&\times \sum_{\substack{h\pmod{k} \\ \gcd(h,k)=1 }} \left((-1)^{\frac{h^2-1}{8}}\right)^\nu  (-1)^{\frac{(\mu-1)(h-1)}{4}}  \left(\frac{h}{\gcd(N,\mu)}\right) \notag  \\
			& \times \left( -  \delta_{\gcd(N,\mu)\mid (hj+r)}  \exp\left(\frac{2\pi i}{24k\gcd(N,\mu)} \left(-48\cdot2^{\nu}jr [4]_{\frac{\mu}{\gcd(N,\mu)}}' \left[\frac{N2^\nu}{\gcd(N,\mu)}\right]_{\frac{\mu}{\gcd(N,\mu)}}' +\frac{12rj}{N}\gcd(N,\mu) \right)\right)  \right.  \notag \\
			& \left. + \delta_{\gcd(N,\mu)\mid (hj-r)} \exp\left(\frac{2\pi i}{24k\gcd(N,\mu)} \left(48\cdot 2^{\nu}jr [4]_{\frac{\mu}{\gcd(N,\mu)}}' \left[\frac{N2^\nu}{\gcd(N,\mu)}\right]_{\frac{\mu}{\gcd(N,\mu)}}'  -\frac{12rj}{N}\gcd(N,\mu) \right)\right)  \right) \notag \\
			& \times  \exp\left(\frac{2\pi i}{24k\gcd(N,\mu)}\left(\left(\left(-24n+1+2k^2-3k\right)\gcd(N,\mu)- 24j^2 2^\nu [4]_{\frac{\mu}{\gcd(N,\mu)}}' \left[\frac{N2^\nu}{\gcd(N,\mu)}\right]_{\frac{\mu}{\gcd(N,\mu)}}'\right)h \right.\right. \notag  \\
			&\left. \left. \hspace{1cm} - \left(\left(-24N\kappa^2-24\kappa r+1-k^2\right)\gcd(N,\mu)-24r^2 2^\nu [4]_{\frac{\mu}{\gcd(N,\mu)}}' \left[\frac{N2^\nu}{\gcd(N,\mu)}\right]_{\frac{\mu}{\gcd(N,\mu)}}'\right)[h]_{8k\gcd(N,\mu)}'\right) \right)\\
			\eqqcolon& K_{k,j,N,1,+}(n,r,\kappa) + K_{k,j,N,1,-}(n,r,\kappa) 
		\end{align*}
	\end{footnotesize}
	\hspace{-0.25cm} where
	\begin{footnotesize}
		\begin{align*}
			&K_{k,j,N,1,\pm}(n,r,\kappa)\\
			=&\epsilon_{e,\pm}^*(k,j,N,r)   \frac{1}{8\gcd(N,\mu)^2}\sum_{s=0}^{\gcd(N,\mu)-1} \exp\left(\pm 2\pi i \frac{rs}{\gcd(N,\mu)}\right) \sum_{\substack{h\pmod{8\gcd(N,\mu)k} \\ \gcd(h,8\gcd(N,\mu)k)=1 }}  (-1)^{\frac{h^2-1}{8}\nu+\frac{(\mu-1)(h-1)}{4}}  \left(\frac{h}{\gcd(N,\mu)}\right)  \notag \\
			& \times \exp\left(\frac{2\pi i}{8k\gcd(N,\mu)}\left(\left(\mu_3+8jsk\right)h - \nu_3 [h]_{8k\gcd(N,\mu)}'\right) \right),
		\end{align*}
	\end{footnotesize}
	\hspace{-0.25cm} with 
	\begin{small}
		\begin{align*}
		\epsilon_{e,\pm}^*(k,j,N,r) \coloneqq&  \mp i \varepsilon_{\frac{\mu}{\gcd(N,\mu)}} \left(\frac{\frac{N2^\nu}{\gcd(N,\mu)}}{\frac{\mu}{\gcd(N,\mu)}}\right) \sqrt{\frac{\gcd(N,\mu)2^\nu}{2N}} \exp\left(\frac{2 \pi i}{4}\right) \delta_{\substack{\nu-\alpha=1 \\ j\not\equiv r \pmod{2}}} \\
		& \times \exp\left(\frac{2\pi i}{24k\gcd(N,\mu)} \left(\mp 48\cdot 2^{\nu}jr [4]_{\frac{\mu}{\gcd(N,\mu)}}' \left[\frac{N2^\nu}{\gcd(N,\mu)}\right]_{\frac{\mu}{\gcd(N,\mu)}}'  \pm \frac{12rj}{N}\gcd(N,\mu) \right)\right), \\
		\mu_3 \coloneqq& \left(-8n+\frac{1+2k^2}{3}-k\right)\gcd(N,\mu)- 8j^2 2^\nu [4]_{\frac{\mu}{\gcd(N,\mu)}}' \left[\frac{N2^\nu}{\gcd(N,\mu)}\right]_{\frac{\mu}{\gcd(N,\mu)}}', \\
		\nu_3 \coloneqq& \left(-8N\kappa^2-8\kappa r+\frac{1-k^2}{3}\right)\gcd(N,\mu)-8r^2 2^\nu [4]_{\frac{\mu}{\gcd(N,\mu)}}' \left[\frac{N2^\nu}{\gcd(N,\mu)}\right]_{\frac{\mu}{\gcd(N,\mu)}}'.
	\end{align*}
	\end{small}
	\hspace{-0.25cm} Note that $\mu_3,\nu_3 \in \Z$, since $k^2-1 \equiv 0 \pmod{3}$ is equivalent to $2k^2 +1 \equiv 0 \pmod{3}$.

	Lastly we use a small trick to rewrite our Kloosterman sum into the shape that we want. First we note that $16 \mid (8k\gcd(N,\mu))$ and that $ (-1)^{\frac{h^2-1}{8}\nu+\frac{(\mu-1)(h-1)}{4}} $ only depends on $h$ modulo $16$. Thus we obtain
	\begin{footnotesize}
		\begin{align}\label{equation: Kloosterman 3nmidk}
			&K_{k,j,N,1,\pm}(n,r,\kappa) \\
		=& \epsilon_{e,\pm}^*(k,j,N,r)   \frac{1}{8\gcd(N,\mu)^2} \sum_{s=0}^{\gcd(N,\mu)-1} \exp\left(\pm 2\pi i \frac{rs}{\gcd(N,\mu)}\right) \frac{1}{16} \sum_{j\pmod{16}}  (-1)^{\frac{j^2-1}{8}\nu+\frac{(\mu-1)(j-1)}{4}} \sum_{\ell \pmod{16}} e^{\frac{-2\pi i j \ell}{16}} \notag \\
			& \times  \sum_{\substack{h\pmod{8\gcd(N,\mu)k}  \\ \gcd(h,8\gcd(N,\mu)k)=1 }}    \left(\frac{h}{\gcd(N,\mu)}\right)  \exp\left(\frac{2\pi i}{8k\gcd(N,\mu)}\left(\left(\mu_3+8jsk+\frac{8\ell\gcd(N,\mu)k}{16}\right)h - \nu_3 [h]_{8k\gcd(N,\mu)}'\right) \right), \notag 
		\end{align}
	\end{footnotesize}
	\hspace{-0.25cm} using the orthogonality of roots of unity 
	$$ \frac{1}{16} \sum_{\ell \pmod{16}} e^{\frac{2\pi i a \ell}{16}} = \begin{cases}
		1 & \text{ if } 16\mid a,\\
		0 & \text{ otherwise.}
	\end{cases}$$

	In the second case, $3\mid k$, we have that $3\nmid h$ and thus $\gcd(24,h)=1$.
	Choosing $[h]_{k\gcd(N,\mu)}'$ such that $h[h]_{k\gcd(N,\mu)}' \equiv -1 \pmod{24k\gcd(N,\mu)}$, analogously to above, yields\footnote{Using that $h[h]_{24k\gcd(N,\mu)}'\equiv -1 \pmod{24k}$ since $(24k) \mid (24k\gcd(N,\mu))$.}
	\begin{align*}
			K_{k,j,N,1}(n,r,\kappa) \eqqcolon& K_{k,j,N,1,+}^*(n,r,\kappa) + K_{k,j,N,1,-}^*(n,r,\kappa),
		\end{align*}
	where analogously to the first case
	\begin{scriptsize}
		\begin{align}\label{equation: Kloosterman 3midk}
			&K_{k,j,N,1,\pm}^*(n,r,\kappa)  \\
			=& \epsilon_{e,\pm}^*(k,j,N,r)   \frac{1}{24\gcd(N,\mu)^2} \sum_{s=0}^{\gcd(N,\mu)-1} \exp\left(\pm 2\pi i \frac{rs}{\gcd(N,\mu)}\right) \frac{1}{16} \sum_{j\pmod{16}}  (-1)^{\frac{j^2-1}{8}\nu+\frac{(\mu-1)(j-1)}{4}} \sum_{\ell \pmod{16}} e^{\frac{-2\pi i j \ell}{16}} \notag \\
			& \times  \sum_{\substack{h\pmod{24\gcd(N,\mu)k}  \\ \gcd(h,24\gcd(N,\mu)k)=1 }}    \left(\frac{h}{\gcd(N,\mu)}\right)  \exp\left(\frac{2\pi i}{24k\gcd(N,\mu)}\left(\left(\mu_4+24jsk+\frac{24\ell\gcd(N,\mu)k}{16}\right)h - \nu_4 [h]_{24k\gcd(N,\mu)}'\right) \right), \notag 
		\end{align}
	\end{scriptsize}
	\hspace{-0.25cm} with 
	\begin{align*}
		\mu_4 \coloneqq& \left(-24n+1+2k^2-3k\right)\gcd(N,\mu)- 24j^2 2^\nu [4]_{\frac{\mu}{\gcd(N,\mu)}}' \left[\frac{N2^\nu}{\gcd(N,\mu)}\right]_{\frac{\mu}{\gcd(N,\mu)}}', \\
		\nu_4 \coloneqq& \left(-24N\kappa^2-24\kappa r+1-k^2\right)\gcd(N,\mu)-24r^2 2^\nu [4]_{\frac{\mu}{\gcd(N,\mu)}}' \left[\frac{N2^\nu}{\gcd(N,\mu)}\right]_{\frac{\mu}{\gcd(N,\mu)}}'. 
	\end{align*}

	We now note that we can bound \eqref{equation: Kloosterman 3nmidk}, respectively \eqref{equation: Kloosterman 3midk}, by 
	\begin{footnotesize}
		\begin{align*}
			&\left|K_{k,j,N,1,\pm}(n,r,\kappa)\right| \\
		\leq& \left|\epsilon_{e,\pm}^*(k,j,N,r) \frac{1}{8\gcd(N,\mu)^2} \frac{1}{16}\right| \sum_{s=0}^{\gcd(N,\mu)-1} \sum_{j\pmod{16}}   \sum_{\ell \pmod{16}} \notag \\
			&  \times \left|  \sum_{\substack{h\pmod{8\gcd(N,\mu)k}  \\ \gcd(h,8\gcd(N,\mu)k)=1 }}    \left(\frac{h}{\gcd(N,\mu)}\right)  \exp\left(\frac{2\pi i}{8k\gcd(N,\mu)}\left(\left(\mu_3+8jsk+\frac{8\ell\gcd(N,\mu)k}{16}\right)h - \nu_3 [h]_{8k\gcd(N,\mu)}'\right) \right)\right|,
		\end{align*}
	\end{footnotesize}
	\hspace{-0.25cm} respectively
	\begin{scriptsize}
		\begin{align*}
			&\left|K_{k,j,N,1,\pm}^*(n,r,\kappa) \right| \\
			\leq&  \left|\epsilon_{e,\pm}^*(k,j,N,r)   \frac{1}{24\gcd(N,\mu)^2} \frac{1}{16}\right| \sum_{s=0}^{\gcd(N,\mu)-1} \sum_{j\pmod{16}}   \sum_{\ell \pmod{16}}  \notag \\
			&  \times  \left|\sum_{\substack{h\pmod{24\gcd(N,\mu)k}  \\ \gcd(h,24\gcd(N,\mu)k)=1 }}    \left(\frac{h}{\gcd(N,\mu)}\right)  \exp\left(\frac{2\pi i}{24k\gcd(N,\mu)}\left(\left(\mu_4+24jsk+\frac{24\ell\gcd(N,\mu)k}{16}\right)h - \nu_4 [h]_{24k\gcd(N,\mu)}'\right) \right) \right|.
		\end{align*}
	\end{scriptsize}
	\hspace{-0.25cm} Both sums over $h$ are of the required shape, so we can bound them using Malishev's result (see  Lemma \ref{Lemma: Malishev}) and obtain that they are 
	\begin{small}
		\begin{align*}
			O\left((8\gcd(N,\mu)k)^{\frac{1}{2}+\varepsilon} \min\left(\gcd\left(\mu_3+8jsk+\frac{8\ell\gcd(N,\mu)k}{16},8\gcd(N,\mu)k\right)^\frac{1}{2}, \gcd\left(\nu_3,8\gcd(N,\mu)k\right)^\frac{1}{2}\right)\right),
		\end{align*}
	\end{small}
	\hspace{-0.25cm} respectively
	\begin{footnotesize}
		\begin{align*}
			O\left((24\gcd(N,\mu)k)^{\frac{1}{2}+\varepsilon} \min\left(\gcd\left(\mu_4+24jsk+\frac{24\ell\gcd(N,\mu)k}{16},24\gcd(N,\mu)k\right)^\frac{1}{2}, \gcd\left(\nu_4,24\gcd(N,\mu)k\right)^\frac{1}{2}\right)\right),
		\end{align*}
	\end{footnotesize}
	\hspace{-0.25cm} for $\varepsilon>0$.

	We see that $8\gcd(N,\mu)\leq24\gcd(N,\mu)\leq24N=O_N(1)$ and 
		\begin{align*}
			&\min\left(\gcd\left(\mu_3+8jsk+\frac{8\ell\gcd(N,\mu)k}{16},8\gcd(N,\mu)k\right)^\frac{1}{2}, \gcd\left(\nu_3,8\gcd(N,\mu)k\right)^\frac{1}{2}\right) = O_N\left(n^\frac{1}{2}\right),
		\end{align*}
		and 
		\begin{align*}
			\min\left(\gcd\left(\mu_4+24jsk+\frac{24\ell\gcd(N,\mu)k}{16},24\gcd(N,\mu)k\right)^\frac{1}{2}, \gcd\left(\nu_4,24\gcd(N,\mu)k\right)^\frac{1}{2}\right) =  O_N\left(n^\frac{1}{2}\right).
		\end{align*}

	This yields
	\begin{small}
		\begin{align*}
			K_{k,j,N,1,\pm}(n,r,\kappa) 
			=& 	O_N\left(\left|\epsilon_{e,\pm}^*(k,j,N,r)   \frac{1}{8\gcd(N,\mu)^2} \frac{1}{16}\right| \sum\limits_{s=0}^{\gcd(N,\mu)-1} \sum\limits_{j\pmod{16}}   \sum\limits_{\ell \pmod{16}} k^{\frac{1}{2}+\varepsilon}n^\frac{1}{2}\right) \\
			=& 	O_N\left(\left|\epsilon_{e,\pm}^*(k,j,N,r)   \frac{1}{8\gcd(N,\mu)^2} \frac{1}{16}\right| 16^2 \gcd(N,\mu) k^{\frac{1}{2}+\varepsilon}n^\frac{1}{2}\right)  \\
			=& O_N\left(n^\frac{1}{2}k^{\frac{1}{2}+\varepsilon} \right),
		\end{align*}
	\end{small}
	\hspace{-0.25cm} and analogously $K_{k,j,N,1,\pm}^*(n,r,\kappa) =O_N(n^\frac{1}{2}k^{\frac{1}{2}+\varepsilon}),$ since $\epsilon_{e,\pm}^*(k,j,N,r)=O_N(1)$. We thus showed that
	\begin{align*}
		K_{k,j,N,1}(n,r,\kappa) =& O_N\left(n^\frac{1}{2}k^{\frac{1}{2}+\varepsilon} \right).
	\end{align*}

	The only thing left to do now is to look at $K_{k,j,N,2}(n,r,\kappa) $, where
	$$ G(hN ,hj\pm r,2^\nu \mu) = 2^{\frac{\nu+\alpha}{2}} A_{\pm} (i+1) \left(\frac{-2^{\nu+\alpha}}{\frac{hN \mu}{2^\alpha}}\right) \varepsilon_{\frac{hN \mu}{2^\alpha}}  \exp\left(-2\pi i \frac{\left[\frac{hN \mu}{2^\alpha}\right]_{2^{\nu+\alpha+2}}' \frac{(hj\pm r)^2}{4}}{2^{\nu+\alpha}}\right).$$

	Analogously to the calculation of $K_{k,j,N,1}$ we obtain that
	\begin{small}
		\begin{align*}
			\chi_{j,r} &(N,M_{h,k}) \notag \\ 
			=&  \delta_{\left(\gcd(N,\mu)2^{\alpha+1}\right)\mid (hj+ r)} (1-i) \varepsilon_{\frac{hN \mu}{2^\alpha}} \varepsilon_{\frac{\mu}{\gcd(N,\mu)}} \left(\frac{k}{-h}\right) \left(\frac{-2^{\nu+\alpha}}{\frac{hN \mu}{2^\alpha}}\right) \left(\frac{\frac{Nh2^\nu}{\gcd(N,\mu)}}{\frac{\mu}{\gcd(N,\mu)}}\right)  \sqrt{\frac{2^{\alpha-1}\gcd(N,\mu)}{N}}   \\
			&\times \exp\left(2 \pi i\left(\frac{1}{24}\left(h'k\left(1-(-h)^2\right)+(-h)\left(-\frac{hh'+1}{k}-k+3\right)-3\right)+\frac{3}{8}\right)\right) \notag \\
			& \times \exp\left( \frac{2\pi i }{4Nk}\left(hj^2-h'r^2+2rj\right)\right)     \exp\left(-2\pi i \frac{\left[\frac{hN \mu}{2^\alpha}\right]_{2^{\nu+\alpha+2}}' \frac{(hj+ r)^2}{4}}{2^{\nu+\alpha}}\right) \notag \\
			& \times   \exp\left(-2\pi i \frac{-[4]_{\frac{\mu}{\gcd(N,\mu)}}' \left[\frac{N2^\nu}{\gcd(N,\mu)}\right]_{\frac{\mu}{\gcd(N,\mu)}}' [h]_{\frac{\mu}{\gcd(N,\mu)}}'\left(\frac{hj + r}{\gcd(N,\mu)}\right)^2}{\frac{\mu}{\gcd(N,\mu)}}\right)  \notag \\
			& + \delta_{(\gcd(N,\mu)2^{\alpha+1})\mid (hj- r)} (i-1)  \varepsilon_{\frac{hN \mu}{2^\alpha}} \varepsilon_{\frac{\mu}{\gcd(N,\mu)}} \left(\frac{k}{-h}\right) \left(\frac{-2^{\nu+\alpha}}{\frac{hN \mu}{2^\alpha}}\right) \left(\frac{\frac{Nh2^\nu}{\gcd(N,\mu)}}{\frac{\mu}{\gcd(N,\mu)}}\right) \sqrt{\frac{2^{\alpha-1}\gcd(N,\mu)}{N}}  \\
			& \times \exp\left(2 \pi i \left(\frac{1}{24}\left(h'k\left(1-(-h)^2\right)+(-h)\left(-\frac{hh'+1}{k}-k+3\right)-3\right)+\frac{3}{8}\right)\right)  \notag \\
			& \times \exp\left( \frac{2\pi i }{4Nk}\left(hj^2-h'r^2-2rj\right)\right)     \exp\left(-2\pi i \frac{\left[\frac{hN \mu}{2^\alpha}\right]_{2^{\nu+\alpha+2}}' \frac{(hj- r)^2}{4}}{2^{\nu+\alpha}}\right) \notag \\
			&\times  \exp\left(-2\pi i \frac{-[4]_{\frac{\mu}{\gcd(N,\mu)}}' \left[\frac{N2^\nu}{\gcd(N,\mu)}\right]_{\frac{\mu}{\gcd(N,\mu)}}' [h]_{\frac{\mu}{\gcd(N,\mu)}}'\left(\frac{hj - r}{\gcd(N,\mu)}\right)^2}{\frac{\mu}{\gcd(N,\mu)}}\right)  .
		\end{align*}
	\end{small}
	\hspace{-0.25cm} Using quadratic reciprocity we have
	\begin{align*}
		&\left(\frac{k}{-h}\right) \left(\frac{-2^{\nu+\alpha}}{\frac{hN \mu}{2^\alpha}}\right) \left(\frac{\frac{Nh2^\nu}{\gcd(N,\mu)}}{\frac{\mu}{\gcd(N,\mu)}}\right)\\
		=& \left(\frac{2}{h}\right)^\nu  \left(\frac{\mu}{h}\right) \left(\frac{-1}{h}\right) \left(\frac{2}{h}\right)^{\nu+\alpha} \left(\frac{-2^{\nu+\alpha}}{\frac{N \mu}{2^\alpha}}\right) \left(\frac{h}{\frac{\mu}{\gcd(N,\mu)}}\right) \left(\frac{\frac{N2^\nu}{\gcd(N,\mu)}}{\frac{\mu}{\gcd(N,\mu)}}\right) \\
		=& \left((-1)^{\frac{h^2-1}{8}}\right)^\nu  (-1)^{\frac{(\mu-1)(h-1)}{4}} (-1)^{\frac{h-1}{2}+\frac{h^2-1}{8}(\nu+\alpha)}  
		\left(\frac{h}{\gcd(N,\mu)}\right) \left(\frac{-2^{\nu+\alpha}}{\frac{N \mu}{2^\alpha}}\right) \left(\frac{\frac{N2^\nu}{\gcd(N,\mu)}}{\frac{\mu}{\gcd(N,\mu)}}\right) .
	\end{align*}
	Therefore our Kloosterman sum equals 
	\begin{small}
		\begin{align*}
			& K_{k,j,N,2}(n,r,\kappa) \\
			=& (1-i)  \varepsilon_{\frac{\mu}{\gcd(N,\mu)}}  \sqrt{\frac{2^{\alpha-1}\gcd(N,\mu)}{N}} \left(\frac{-2^{\nu+\alpha}}{\frac{N \mu}{2^\alpha}}\right) \left(\frac{\frac{N2^\nu}{\gcd(N,\mu)}}{\frac{\mu}{\gcd(N,\mu)}}\right) \exp\left(\frac{2\pi i}{4}\right) \delta_{\nu-\alpha>1} \notag \\
			& \times  \sum_{\substack{h\pmod{k} \\ \gcd(h,k)=1 }} (-1)^{\frac{h^2-1}{8}\nu+\frac{(\mu-1)(h-1)}{4}+\frac{h-1}{2}+\frac{h^2-1}{8}(\nu+\alpha)} \varepsilon_{\frac{hN \mu}{2^\alpha}} \left(\frac{h}{\gcd(N,\mu)}\right)   \notag \\
			& \times   \exp\left(\frac{2\pi i}{24k}\left(\left(-24n+2+k^2-3k\right)h - \left(-24N\kappa^2 -24\kappa r +1-k^2\right)h'\right) \right)  \\
			&\times \left(\rule{0pt}{0.7cm}  \delta_{(\gcd(N,\mu)2^{\alpha+1})\mid (hj+ r)} \exp\left(\frac{2 \pi i}{24k} \left( -h^2h'k^2 +h^2h'+\frac{12rj}{N}\right)\right) \exp\left(-2\pi i \frac{\left[\frac{hN \mu}{2^\alpha}\right]_{2^{\nu+\alpha+2}}' \frac{(hj+ r)^2}{4}}{2^{\nu+\alpha}}\right)  \right. \notag \\
			& \left. \times  \exp\left(2\pi i \frac{[4]_{\frac{\mu}{\gcd(N,\mu)}}' \left[\frac{N2^\nu}{\gcd(N,\mu)}\right]_{\frac{\mu}{\gcd(N,\mu)}}' [h]_{\frac{\mu}{\gcd(N,\mu)}}'\left(\frac{hj + r}{\gcd(N,\mu)}\right)^2}{\frac{\mu}{\gcd(N,\mu)}}\right)\right.  \notag \\
			& \left.-  \delta_{(\gcd(N,\mu)2^{\alpha+1})\mid (hj- r)}  \exp\left(\frac{2 \pi i}{24k} \left( -h^2h'k^2 +h^2h'-\frac{12rj}{N}\right)\right) \exp\left(-2\pi i \frac{\left[\frac{hN \mu}{2^\alpha}\right]_{2^{\nu+\alpha+2}}' \frac{(hj- r)^2}{4}}{2^{\nu+\alpha}}\right) \right.  \notag \\
			& \left. \times  \exp\left(2\pi i \frac{[4]_{\frac{\mu}{\gcd(N,\mu)}}' \left[\frac{N2^\nu}{\gcd(N,\mu)}\right]_{\frac{\mu}{\gcd(N,\mu)}}' [h]_{\frac{\mu}{\gcd(N,\mu)}}'\left(\frac{hj - r}{\gcd(N,\mu)}\right)^2}{\frac{\mu}{\gcd(N,\mu)}}\right)\rule{0pt}{0.7cm}\right).
		\end{align*}
	\end{small}
	
	We observe that
	\begin{align*}
		\exp\left(-2\pi i \frac{\left[\frac{hN \mu}{2^\alpha}\right]_{2^{\nu+\alpha+2}}' \frac{(hj\pm r)^2}{4}}{2^{\nu+\alpha}}\right) 
		=& \exp\left(\frac{2\pi i}{2^{\nu+\alpha+2}} \left(\left[\frac{N \mu}{2^\alpha}\right]_{2^{\nu+\alpha+2}}' \left(-hj^2\mp 2jr+ \left[h\right]_{2^{\nu+\alpha+2}}'r^2\right)\right)\right).
	\end{align*}
	
	Choose $[h]_{\frac{\mu}{\gcd(N,\mu)}}'=h'$ from now on\footnote{Note that $hh'\equiv -1 \pmod{k}$ implies that $hh'\equiv -1 ~(\!\!\!\mod{\frac{\mu}{\gcd(N,\mu)}})$, since $\frac{\mu}{\gcd(N,\mu)}\mid k$. Thus $h'$ is a possible choice for $\left[h\right]_{\frac{\mu}{\gcd(N,\mu)}}'$.}. Analogously to the odd $k$ case or the calculations of $K_{k,j,N,1}$ we are able to choose $h'$ such that $hh'\equiv -1 \pmod{2^{\alpha+2}\gcd(N,\mu)k}.$
	Choosing $\left[h\right]_{2^{\nu+\alpha+2}}'=[h]_{2^{\alpha+2}k\gcd(N,\mu)}'$ in addition\footnote{Note that $h[h]_{2^{\alpha+2}k\gcd(N,\mu)}'\equiv -1 \pmod{2^{\alpha+2}k\gcd(N,\mu)}$ implies that $h[h]_{2^{\alpha+2}k\gcd(N,\mu)}'\equiv -1 \pmod{2^{\nu+\alpha+2}}$, since $2^{\nu+\alpha+2}\mid (2^{\alpha+2}k\gcd(N,\mu))$. Thus $[h]_{2^{\alpha+2}k\gcd(N,\mu)}'$ is a possible choice for $\left[h\right]_{2^{\nu+\alpha+2}}'$.}, we obtain that 
	\begin{scriptsize}
		\begin{align*}
			& K_{k,j,N,2}(n,r,\kappa) \\
			=& (1-i) \varepsilon_{\frac{\mu}{\gcd(N,\mu)}}  \sqrt{\frac{2^{\alpha-1}\gcd(N,\mu)}{N}} \left(\frac{-2^{\nu+\alpha}}{\frac{N \mu}{2^\alpha}}\right) \left(\frac{\frac{N2^\nu}{\gcd(N,\mu)}}{\frac{\mu}{\gcd(N,\mu)}}\right) \exp\left(\frac{2\pi i}{4}\right) \delta_{\nu-\alpha>1} \notag \\
			& \times  \sum_{\substack{h\pmod{k} \\ \gcd(h,k)=1 }} (-1)^{\frac{h^2-1}{8}\nu+\frac{(\mu-1)(h-1)}{4}+\frac{h-1}{2}+\frac{h^2-1}{8}(\nu+\alpha)} \varepsilon_{\frac{hN \mu}{2^\alpha}} \left(\frac{h}{\gcd(N,\mu)}\right)   \notag \\
			& \times   \exp\left(\frac{2\pi i}{24\cdot 2^{\alpha+2}k\gcd(N,\mu)} \left( \left(\left(-24n+2+k^2-3k\right)2^{\alpha+2}\gcd(N,\mu)- 24j^2 2^{\nu+\alpha+2} [4]_{\frac{\mu}{\gcd(N,\mu)}}' \left[\frac{N2^\nu}{\gcd(N,\mu)}\right]_{\frac{\mu}{\gcd(N,\mu)}}' \right. \right. \right. \\
			& \left. \left. \left. \hspace{4.5cm} -24\mu j^2\left[\frac{N\mu}{2^\alpha}\right]_{2^{\nu+\alpha+2}}'\gcd(N,\mu)\rule{0pt}{0.5cm}\right)h \right.\right. \\
			&\left.\left. \hspace{3cm}- \left( \left(-24N\kappa^2 -24\kappa r +1-k^2\right)2^{\alpha+2}\gcd(N,\mu) - 24r^2 2^{\nu+\alpha+2} [4]_{\frac{\mu}{\gcd(N,\mu)}}' \left[\frac{N2^\nu}{\gcd(N,\mu)}\right]_{\frac{\mu}{\gcd(N,\mu)}}'\right. \right. \right. \\
			&\left. \left. \left. \hspace{4cm} -24\mu r^2\left[\frac{N\mu}{2^\alpha}\right]_{2^{\nu+\alpha+2}}'\gcd(N,\mu)\rule{0pt}{0.5cm}\right)[h]_{2^{\alpha+2}k\gcd(N,\mu)}'\rule{0pt}{0.5cm}\right)\rule{0pt}{0.5cm} \right)  \\
			&\times \left(\rule{0pt}{0.7cm} \delta_{(\gcd(N,\mu)2^{\alpha+1})\mid (hj+ r)}  \exp\left(\frac{2 \pi i}{24k} \left( h^2[h]_{2^{\alpha+2}k\gcd(N,\mu)}'(1-k^2)+\frac{12rj}{N}\right)\right) \right. \\
			& \left. \times \exp\left(\frac{2\pi i}{2^{\nu+\alpha+2}} \left(-2jr\left[\frac{N \mu}{2^\alpha}\right]_{2^{\nu+\alpha+2}}' \right)\right)   \exp\left(\frac{2\pi i}{k\gcd(N,\mu)} \left(-2^{\nu+1} jr [4]_{\frac{\mu}{\gcd(N,\mu)}}' \left[\frac{N2^\nu}{\gcd(N,\mu)}\right]_{\frac{\mu}{\gcd(N,\mu)}}' \right) \right)\right.  \notag \\
			& \left.- \delta_{(\gcd(N,\mu)2^{\alpha+1})\mid (hj- r)}   \exp\left(\frac{2 \pi i}{24k} \left( h^2[h]_{2^{\alpha+2}k\gcd(N,\mu)}'(1-k^2)-\frac{12rj}{N}\right)\right) \right. \\
			&\left. \times \exp\left(\frac{2\pi i}{2^{\nu+\alpha+2}} \left(2jr\left[\frac{N \mu}{2^\alpha}\right]_{2^{\nu+\alpha+2}}' \right)\right)  \exp\left(\frac{2\pi i}{k\gcd(N,\mu)} \left(2^{\nu+1}jr [4]_{\frac{\mu}{\gcd(N,\mu)}}' \left[\frac{N2^\nu}{\gcd(N,\mu)}\right]_{\frac{\mu}{\gcd(N,\mu)}}' \right) \right) \rule{0pt}{0.7cm}\right).
		\end{align*}
	\end{scriptsize}
	
	We now need to split into two cases, $3\nmid k$ and $3\mid k$. In the first case we obtain $3\mid (k^2-1)$.
	Choosing $[h]_{2^{\alpha+2}k\gcd(N,\mu)}'$ such that $h[h]_{2^{\alpha+2}k\gcd(N,\mu)}' \equiv -1 \pmod{2^{\alpha+5}k\gcd(N,\mu)}$, analogously to above, yields\footnote{We are allowed to do this since $\gcd(8,h)=1$, this is because we know that $h$ is odd. Additionally we use that $h[h]_{2^{\alpha+5}k\gcd(N,\mu)}'\equiv -1 \pmod{8k}$ since $(8k) \mid (2^{\alpha+5}k\gcd(N,\mu))$.}
	\begin{scriptsize}
		\begin{align*}
			&K_{k,j,N,2}(n,r,\kappa) \\
			=& (1-i)  \varepsilon_{\frac{\mu}{\gcd(N,\mu)}}  \sqrt{\frac{2^{\alpha-1}\gcd(N,\mu)}{N}} \left(\frac{-2^{\nu+\alpha}}{\frac{N \mu}{2^\alpha}}\right) \left(\frac{\frac{N2^\nu}{\gcd(N,\mu)}}{\frac{\mu}{\gcd(N,\mu)}}\right) \exp\left(\frac{2\pi i}{4}\right) \delta_{\nu-\alpha>1} \notag \\
			& \times  \sum_{\substack{h\pmod{k} \\ \gcd(h,k)=1 }} (-1)^{\frac{h^2-1}{8}\nu+\frac{(\mu-1)(h-1)}{4}+\frac{h-1}{2}+\frac{h^2-1}{8}(\nu+\alpha)} \varepsilon_{\frac{hN \mu}{2^\alpha}} \left(\frac{h}{\gcd(N,\mu)}\right)   \notag \\
			&\times \left(\rule{0pt}{0.7cm} \delta_{(\gcd(N,\mu)2^{\alpha+1})\mid (hj+ r)}  \exp\left(\frac{2 \pi i}{8k} \left(\frac{12rj}{3N}\right)\right) \right. \\
			& \left. \times \exp\left(\frac{2\pi i}{2^{\nu+\alpha+2}} \left(-2jr\left[\frac{N \mu}{2^\alpha}\right]_{2^{\nu+\alpha+2}}' \right)\right)   \exp\left(\frac{2\pi i}{k\gcd(N,\mu)} \left(-2^{\nu+1} jr [4]_{\frac{\mu}{\gcd(N,\mu)}}' \left[\frac{N2^\nu}{\gcd(N,\mu)}\right]_{\frac{\mu}{\gcd(N,\mu)}}' \right) \right)\right.  \notag \\
			& \left.- \delta_{(\gcd(N,\mu)2^{\alpha+1})\mid (hj- r)}   \exp\left(\frac{2 \pi i}{8k} \left( -\frac{12rj}{3N}\right)\right) \right. \\
			&\left. \times \exp\left(\frac{2\pi i}{2^{\nu+\alpha+2}} \left(2jr\left[\frac{N \mu}{2^\alpha}\right]_{2^{\nu+\alpha+2}}' \right)\right)  \exp\left(\frac{2\pi i}{k\gcd(N,\mu)} \left(2^{\nu+1}jr [4]_{\frac{\mu}{\gcd(N,\mu)}}' \left[\frac{N2^\nu}{\gcd(N,\mu)}\right]_{\frac{\mu}{\gcd(N,\mu)}}' \right) \right) \rule{0pt}{0.7cm}\right) \\
			& \times   \exp\left(\frac{2\pi i}{24\cdot 2^{\alpha+2}k\gcd(N,\mu)} \left( \left(\left(-24n+1+2k^2-3k\right)2^{\alpha+2}\gcd(N,\mu)- 24j^2 2^{\nu+\alpha+2} [4]_{\frac{\mu}{\gcd(N,\mu)}}' \left[\frac{N2^\nu}{\gcd(N,\mu)}\right]_{\frac{\mu}{\gcd(N,\mu)}}' \right. \right. \right. \\
			& \left. \left. \left. \hspace{4.5cm} -24\mu j^2\left[\frac{N\mu}{2^\alpha}\right]_{2^{\nu+\alpha+2}}'\gcd(N,\mu)\rule{0pt}{0.5cm}\right)h \right.\right. \\
			&\left.\left. \hspace{3cm}- \left( \left(-24N\kappa^2 -24\kappa r +1-k^2\right)2^{\alpha+2}\gcd(N,\mu) - 24r^2 2^{\nu+\alpha+2} [4]_{\frac{\mu}{\gcd(N,\mu)}}' \left[\frac{N2^\nu}{\gcd(N,\mu)}\right]_{\frac{\mu}{\gcd(N,\mu)}}'\right. \right. \right. \\
			&\left. \left. \left. \hspace{4cm} -24\mu r^2\left[\frac{N\mu}{2^\alpha}\right]_{2^{\nu+\alpha+2}}'\gcd(N,\mu)\rule{0pt}{0.5cm}\right)[h]_{2^{\alpha+5}k\gcd(N,\mu)}'\rule{0pt}{0.5cm}\right)\rule{0pt}{0.5cm} \right) \\
			\eqqcolon& K_{k,j,N,2,+}(n,r,\kappa)+ K_{k,j,N,2,-}(n,r,\kappa),
		\end{align*}
	\end{scriptsize}
	\hspace{-0.25cm} where
	\begin{small}
		\begin{align*}
			K_{k,j,N,2,\pm}(n,r,\kappa)
			=& \epsilon_{e,\pm}(k,j,N,r) \frac{1}{2^{2\alpha+6}\gcd(N,\mu)^2} \sum_{s=0}^{\gcd(N,\mu)2^{\alpha+1}-1} \exp\left(\pm 2 \pi i \frac{rs}{\gcd(N,\mu)2^{\alpha+1}}\right) \\
			& \times \sum_{\substack{h\pmod{2^{\alpha+5}k\gcd(N,\mu)} \\ \gcd(h,2^{\alpha+5}k\gcd(N,\mu))=1 }} (-1)^{\frac{h^2-1}{8}\nu+\frac{(\mu-1)(h-1)}{4}+\frac{h-1}{2}+\frac{h^2-1}{8}(\nu+\alpha)} \varepsilon_{\frac{hN \mu}{2^\alpha}} \left(\frac{h}{\gcd(N,\mu)}\right)   \notag \\
			& \times   \exp\left(\frac{2\pi i}{ 2^{\alpha+5}k\gcd(N,\mu)} \left( \left(\mu_5+16 jks\right)h - \nu_5[h]_{2^{\alpha+5}k\gcd(N,\mu)}'\right)\right),  
		\end{align*}
	\end{small}
	\hspace{-0.25cm} with
	\begin{scriptsize}
		\begin{align*}
		\epsilon_{e,\pm}&(k,j,N,r)\coloneqq \pm(1-i)  \varepsilon_{\frac{\mu}{\gcd(N,\mu)}}  \sqrt{\frac{2^{\alpha-1}\gcd(N,\mu)}{N}} \left(\frac{-2^{\nu+\alpha}}{\frac{N \mu}{2^\alpha}}\right) \left(\frac{\frac{N2^\nu}{\gcd(N,\mu)}}{\frac{\mu}{\gcd(N,\mu)}}\right) \exp\left(\frac{2\pi i}{4}\right) \delta_{\nu-\alpha>1}  \exp\left(\frac{2 \pi i}{8k} \left(\pm\frac{12rj}{3N}\right)\right) \\
		&\times \exp\left(\frac{2\pi i}{2^{\nu+\alpha+2}} \left(\mp 2jr\left[\frac{N \mu}{2^\alpha}\right]_{2^{\nu+\alpha+2}}' \right)\right)   \exp\left(\frac{2\pi i}{k\gcd(N,\mu)} \left(\mp 2^{\nu+1} jr [4]_{\frac{\mu}{\gcd(N,\mu)}}' \left[\frac{N2^\nu}{\gcd(N,\mu)}\right]_{\frac{\mu}{\gcd(N,\mu)}}' \right) \right),\\
		\mu_5\coloneqq& \left(-8n+\frac{1+2k^2}{3}-k\right)2^{\alpha+2}\gcd(N,\mu)- 8j^2 2^{\nu+\alpha+2} [4]_{\frac{\mu}{\gcd(N,\mu)}}' \left[\frac{N2^\nu}{\gcd(N,\mu)}\right]_{\frac{\mu}{\gcd(N,\mu)}}'  -8\mu j^2\left[\frac{N\mu}{2^\alpha}\right]_{2^{\nu+\alpha+2}}'\gcd(N,\mu),\\
		\nu_5\coloneqq&	\left(-8N\kappa^2 -8\kappa r +\frac{1-k^2}{3}\right)2^{\alpha+2}\gcd(N,\mu) - 8r^2 2^{\nu+\alpha+2} [4]_{\frac{\mu}{\gcd(N,\mu)}}' \left[\frac{N2^\nu}{\gcd(N,\mu)}\right]_{\frac{\mu}{\gcd(N,\mu)}}' -8\mu r^2\left[\frac{N\mu}{2^\alpha}\right]_{2^{\nu+\alpha+2}}'\gcd(N,\mu).
		\end{align*}
	\end{scriptsize}
	\hspace{-0.25cm} Note that $\mu_5, \nu_5 \in \Z$, since $3\mid (k^2-1) $ is equivalent to $3 \mid  (2k^2 +1)$.

	Analogously to above we note that $16 \mid (2^{\alpha+5}k\gcd(N,\mu))$, $(-1)^{\frac{h^2-1}{8}\nu+\frac{(\mu-1)(h-1)}{4}+\frac{h-1}{2}+\frac{h^2-1}{8}(\nu+\alpha)} $ only depends on $h$ modulo $16$, and $\varepsilon_{\frac{hN \mu}{2^\alpha}}$ only depends on $h$ modulo $4$, means we can also look at it modulo $16$, since $4\mid 16$. Thus we obtain
	\begin{footnotesize}
		\begin{align}\label{equation: Kloosterman 3nmidk even}
			K_{k,j,N,2,\pm}(n,r,\kappa) =& \epsilon_{e,\pm}(k,j,N,r) \frac{1}{2^{2\alpha+6}\gcd(N,\mu)^2} \frac{1}{16} \sum_{s=0}^{\gcd(N,\mu)2^{\alpha+1}-1} \exp\left(\pm 2 \pi i \frac{rs}{\gcd(N,\mu)2^{\alpha+1}}\right) \notag \\
			& \times \sum_{j\pmod{16}}   (-1)^{\frac{j^2-1}{8}\nu+\frac{(\mu-1)(j-1)}{4}+\frac{j-1}{2}+\frac{j^2-1}{8}(\nu+\alpha)} \varepsilon_{\frac{jN \mu}{2^\alpha}} \sum_{\ell\pmod{16}} e^{\frac{-2\pi i j\ell}{16}} \notag  \\
			& \times  \sum_{\substack{h\pmod{2^{\alpha+5}k\gcd(N,\mu)}  \\ \gcd(h,2^{\alpha+5}k\gcd(N,\mu))=1 }}  \ \left(\frac{h}{\gcd(N,\mu)}\right)   \\
			& \times  \exp\left(\frac{2\pi i}{ 2^{\alpha+5}k\gcd(N,\mu)} \left( \left(\mu_5 +16 jks+\frac{ 2^{\alpha+5}\ell k\gcd(N,\mu)}{16}\right)h - \nu_5[h]_{2^{\alpha+5}k\gcd(N,\mu)}'\right)\right). \notag 
		\end{align}
	\end{footnotesize}
	
	In the second case, $3\mid k$, we have that $3\nmid h$ and thus $\gcd(24,h)=1$.
	Choosing $[h]_{2^{\alpha+2}k\gcd(N,\mu)}'$ such that $h[h]_{2^{\alpha+2}k\gcd(N,\mu)}' \equiv -1 \pmod{24\cdot 2^{\alpha+2}k\gcd(N,\mu)}$, analogously to above, yields\footnote{Using that $h[h]_{24\cdot 2^{\alpha+2}k\gcd(N,\mu)}'\equiv -1 \pmod{24k}$ since $(24k) \mid (24\cdot 2^{\alpha+2}k\gcd(N,\mu))$.}
		\begin{align*}
			K_{k,j,N,2}(n,r,\kappa) \eqqcolon K^*_{k,j,N,2,+}(n,r,\kappa) + K^*_{k,j,N,2,-}(n,r,\kappa),
		\end{align*}
	where analogously to the first case
	\begin{footnotesize}
		\begin{align}\label{equation: Kloosterman 3midk even}
			K_{k,j,N,2,\pm}(n,r,\kappa) 	=& \epsilon_{e,\pm}(k,j,N,r) \frac{1}{3\cdot2^{2\alpha+6}\gcd(N,\mu)^2} \frac{1}{16} \sum_{s=0}^{\gcd(N,\mu)2^{\alpha+1}-1} \exp\left(\pm 2\pi i \frac{rs}{\gcd(N,\mu)2^{\alpha+1}}\right) \notag \\
			& \times \sum_{j\pmod{16}}   (-1)^{\frac{j^2-1}{8}\nu+\frac{(\mu-1)(j-1)}{4}+\frac{j-1}{2}+\frac{j^2-1}{8}(\nu+\alpha)} \varepsilon_{\frac{jN \mu}{2^\alpha}} \sum_{\ell\pmod{16}} e^{\frac{-2\pi i j\ell}{16}} \notag  \\
			& \times  \sum_{\substack{h\pmod{3\cdot2^{\alpha+5}k\gcd(N,\mu)}  \\ \gcd(h,3\cdot2^{\alpha+5}k\gcd(N,\mu))=1 }}  \ \left(\frac{h}{\gcd(N,\mu)}\right)    \\
			& \times  \exp\left(\frac{2\pi i}{ 3\cdot2^{\alpha+5}k\gcd(N,\mu)} \left( \left(\mu_6 + 48jks+\frac{ 3\cdot2^{\alpha+5}\ell k\gcd(N,\mu)}{16}\right)h - \nu_6[h]_{3\cdot2^{\alpha+5}k\gcd(N,\mu)}'\right)\right), \notag 
		\end{align}
	\end{footnotesize}
	\hspace{-0.25cm} with
	\begin{tiny}
		\begin{align*}
			\mu_6\coloneqq& \left(-24n+1+2k^2-3k\right)2^{\alpha+2}\gcd(N,\mu)- 24j^2 2^{\nu+\alpha+2} [4]_{\frac{\mu}{\gcd(N,\mu)}}' \left[\frac{N2^\nu}{\gcd(N,\mu)}\right]_{\frac{\mu}{\gcd(N,\mu)}}'  -24\mu j^2\left[\frac{N\mu}{2^\alpha}\right]_{2^{\nu+\alpha+2}}'\gcd(N,\mu),\\
			\nu_6\coloneqq&	\left(-24N\kappa^2 -24\kappa r +1-k^2\right)2^{\alpha+2}\gcd(N,\mu) - 24r^2 2^{\nu+\alpha+2} [4]_{\frac{\mu}{\gcd(N,\mu)}}' \left[\frac{N2^\nu}{\gcd(N,\mu)}\right]_{\frac{\mu}{\gcd(N,\mu)}}' -24\mu r^2\left[\frac{N\mu}{2^\alpha}\right]_{2^{\nu+\alpha+2}}'\gcd(N,\mu).
		\end{align*}
	\end{tiny}

	We now note that we can bound \eqref{equation: Kloosterman 3nmidk even} by 
	\begin{tiny}
		\begin{align*}
		&\left|K_{k,j,N,2,\pm}(n,r,\kappa)\right| \notag \\
		\leq& \left|\epsilon_{e,\pm}(k,j,N,r) \frac{1}{2^{2\alpha+6}\gcd(N,\mu)^2} \frac{1}{16} \right| \sum_{s=0}^{\gcd(N,\mu)2^{\alpha+1}-1} \sum_{j\pmod{16}}   \sum_{\ell\pmod{16}}  \notag  \\
		& \times \left|  \sum_{\substack{h\pmod{2^{\alpha+5}k\gcd(N,\mu)}  \\ \gcd(h,2^{\alpha+5}k\gcd(N,\mu))=1 }}  \ \left(\frac{h}{\gcd(N,\mu)}\right)  \exp\left(\frac{2\pi i}{ 2^{\alpha+5}k\gcd(N,\mu)} \left( \left(\mu_5 +16jsk +\frac{ 2^{\alpha+5}\ell k\gcd(N,\mu)}{16}\right)h - \nu_5[h]_{2^{\alpha+5}k\gcd(N,\mu)}'\right)\right)\rule{0pt}{1cm}\right|.
		\end{align*}
	\end{tiny}
	\hspace{-0.25cm} Moreover we obtain that \eqref{equation: Kloosterman 3midk even} is bounded by
	\begin{small}
		\begin{align*}
		&\left|K_{k,j,N,2,\pm}(n,r,\kappa) \right| \notag \\
		\leq& \left|\epsilon_{e,\pm}(k,j,N,r) \frac{1}{3\cdot2^{2\alpha+6}\gcd(N,\mu)^2} \frac{1}{16} \right| \sum_{s=0}^{\gcd(N,\mu)2^{\alpha+1}-1} \sum_{j\pmod{16}}    \sum_{\ell\pmod{16}}  \notag  \\
		& \times \left| \sum_{\substack{h\pmod{3\cdot2^{\alpha+5}k\gcd(N,\mu)}  \\ \gcd(h,3\cdot2^{\alpha+5}k\gcd(N,\mu))=1 }}  \ \left(\frac{h}{\gcd(N,\mu)}\right)  \right. \\
		& \left.\hspace{0.5cm}  \times  \exp\left(\frac{2\pi i}{ 3\cdot2^{\alpha+5}k\gcd(N,\mu)} \left( \left(\mu_6+ 48jsk +\frac{ 3\cdot2^{\alpha+5}\ell k\gcd(N,\mu)}{16}\right)h - \nu_6[h]_{3\cdot2^{\alpha+5}k\gcd(N,\mu)}'\right)\right) \rule{0pt}{1.2cm} \right|.
		\end{align*}
	\end{small} 
	
	Both last sums over $h$ are of the required shape, so we can bound them using Malishev's result (see  Lemma \ref{Lemma: Malishev}) and obtain that they are 
	\begin{scriptsize}
		\begin{align*}
			O\left((2^{\alpha+5}k\gcd(N,\mu))^{\frac{1}{2}+\varepsilon} \min\left(\gcd\left(\mu_5 +16jks+\frac{ 2^{\alpha+5}\ell k\gcd(N,\mu)}{16},2^{\alpha+5}k\gcd(N,\mu)\right)^\frac{1}{2}, \gcd\left(\nu_5,2^{\alpha+5}k\gcd(N,\mu)\right)^\frac{1}{2}\right)\right),
		\end{align*}
	\end{scriptsize}
	\hspace{-0.25cm} respectively
	\begin{footnotesize}
		\begin{align*}
			&O\left((3\cdot2^{\alpha+5}k\gcd(N,\mu))^{\frac{1}{2}+\varepsilon} \right. \\
			&\left.\times \min\left(\gcd\left(\mu_6 + 48jks +\frac{ 3\cdot2^{\alpha+5}\ell k\gcd(N,\mu)}{16},3\cdot2^{\alpha+5}k\gcd(N,\mu)\right)^\frac{1}{2}, \gcd\left(\nu_6,3\cdot2^{\alpha+5}k\gcd(N,\mu)\right)^\frac{1}{2}\right)\right),
		\end{align*}
	\end{footnotesize}
	\hspace{-0.25cm} for $\varepsilon>0$.

	We see that $2^{\alpha+5}\gcd(N,\mu)\leq 3\cdot2^{\alpha+5}\gcd(N,\mu)\leq 3\cdot 2^5 N^2=O_N(1)$,
	\begin{small}
		\begin{align*}
			\min\left(\gcd\left(\mu_5+ 16jks +\frac{ 2^{\alpha+5}\ell k\gcd(N,\mu)}{16},2^{\alpha+5}k\gcd(N,\mu)\right)^\frac{1}{2}, \gcd\left(\nu_5,2^{\alpha+5}k\gcd(N,\mu)\right)^\frac{1}{2}\right) = O_N\left(n\right),
		\end{align*}
	\end{small}
	\hspace{-0.25cm} and 
	\begin{footnotesize}
		\begin{align*}
			\min\left(\gcd\left(\mu_6+48jks +\frac{ 3\cdot2^{\alpha+5}\ell k\gcd(N,\mu)}{16},3\cdot2^{\alpha+5}k\gcd(N,\mu)\right)^\frac{1}{2}, \gcd\left(\nu_6,3\cdot2^{\alpha+5}k\gcd(N,\mu)\right)^\frac{1}{2}\right) =  O_N\left(n\right).
		\end{align*}
	\end{footnotesize}
	\hspace{-0.25cm} This yields
	\begin{small}
		\begin{align*}
			K_{k,j,N,2,\pm}(n,r,\kappa) =& 	O_N\left(\left|\epsilon_{e,\pm}(k,j,N,r)   \frac{1}{2^{2\alpha+6}\gcd(N,\mu)^2} \frac{1}{16}\right| \sum\limits_{s=0}^{\gcd(N,\mu)2^{\alpha+1}-1} \sum\limits_{j\pmod{16}}   \sum\limits_{\ell \pmod{16}} k^{\frac{1}{2}+\varepsilon}n\right) \\
			=& 	O_N\left(\left|\epsilon_{e,\pm}(k,j,N,r)   \frac{1}{2^{2\alpha+6}\gcd(N,\mu)^2} \frac{1}{16}\right| 16^2 \gcd(N,\mu)2^{\alpha+1}  k^{\frac{1}{2}+\varepsilon}n\right) =  O_N\left(n k^{\frac{1}{2}+\varepsilon} \right),
		\end{align*}
	\end{small}
	\hspace{-0.25cm} and analogously $K_{k,j,N,2,\pm}^*(n,r,\kappa) = O_N(n k^{\frac{1}{2}+\varepsilon})$, since $\epsilon_{e,\pm}(k,j,N,r)=O_N(1)$. We thus showed that
	\begin{align*}
		K_{k,j,N,2}(n,r,\kappa) =& O_N\left(n k^{\frac{1}{2}+\varepsilon} \right),
	\end{align*}
	which finally gives 
	\begin{align*}
		K_{k,j,N}(n,r,\kappa) =& O_N\left(n k^{\frac{1}{2}+\varepsilon} \right)
	\end{align*}
	and finishes the proof for $k$ even and therefore the proof of Theorem \ref{Thm Kloosterman bound}.

\section{Applying the Circle Method}\label{section: Circle Method}
In this section we use the Circle Method and ideas of Rademacher and Zuckerman \cite{rademacher1938partition,rademacher1938fourier,rademacherzuckerman1938fourier} to finally prove Theorem \ref{main Thm}. As we already mentioned in the introduction of this paper, the Kloosterman sum and transformation behavior of our family of functions is a little more complicated here than it is in \cite{rademacher1938fourier}, for example. Even though we now have a nice bound for for our Kloosterman sum this will cause extra work in bounding the error parts.

Let $0\leq h< k\leq J$ with $\gcd(h,k)=1$ and a parameter $J\in\N$ that later tends to infinity.
Furthermore let $\frac{h_1}{k_1} <\frac hk <\frac{h_2}{k_2}$ be consecutive fractions in the \textit{Farey sequence} of order $J$ (a series of
fractions $\frac{p_j}{q_j}$ with $p_j\leq q_j \leq J$, $\gcd(p_j , q_j) = 1$ and $\frac{p_j}{q_j}<\frac{p_\ell}{q_\ell}$ for all $j<\ell$). We denote the {\it Farey arc} $\xi_{h,k}$ to be the image of $(\frac{h_1+h}{k_1+k},\frac{h_2+h}{k_2+k})$ under the map (see e.g.\@, \cite{rademacher1938partition})
$$\phi \mapsto e^{-2\pi J^{-2}+2\pi i \phi }$$
and $\xi_{0,1}$ to be the image of $(-\frac{1}{J+1},\frac{1}{J+1})$ (see e.g.\@, \cite[equation (5.2.9)]{andrews1976partitions}).

Note that for the fraction $\frac{h}{k}$ and its neighbors we have (see \cite[page 503]{rademacher1938fourier})
$$ hk_1-h_1k =1 \quad \text{and} \quad h_2k-hk_2=1,$$
which is equivalent to
$$ hk_1\equiv 1 \pmod{k} \quad \text{and} \quad hk_2\equiv -1 \pmod{k},$$
or, using that $hh'\equiv -1 \pmod{k}$,
\begin{align}\label{equation: equivalences for k_1 and k_2}
	k_1 \equiv -h'\pmod{k} \quad \text{and} \quad k_2 \equiv h' \pmod{k}.
\end{align}
Since $\frac{h_1+h}{k_1+k}$ and $\frac{h_2+h}{k_2+k}$ do not belong to the Farey sequence of order $J$ we have $k_1+k >J$ and $k_2+k>J$, which, together with $k_1,k_2\leq J$, enclose $k_1$ and $k_2$ to the intervals
\begin{align}\label{equation: intervals k_1 and k_2}
J-k < k_1 \leq J, \qquad J-k < k_2 \leq J.
\end{align}
The formulae \eqref{equation: equivalences for k_1 and k_2} and \eqref{equation: intervals k_1 and k_2} thus determine $k_1$ and $k_2$ uniquely as functions of $h$ and $k$.

Using Cauchy's formula and \eqref{equation: series Aj,N} we write (see e.g.\@, \cite[equation (14.4)]{bringmann2017harmonic})
\begin{align}\label{equation: fourier coefficients integral}
	a_{j,N}(n) = \frac{1}{2\pi i} \int_{C_J} \frac{q^{\frac{1}{24}-\frac{j^2}{4N}}\CA_{j,N}(\tau)}{q^{n+1}} dq,
\end{align}
where $C_J$ is an arbitrary path inside the unit disk that loops around zero in the counterclockwise direction exactly once. Here we choose $C_J$ to be the circle of radius $e^{-2\pi J^{-2}}<1$ and note that we can split this circle into disjoint Farey arcs as done in Rademacher's original works \cite{rademacher1938fourier, rademacher1938partition} by
$$ \bigcup_{\substack{0\leq h < k \leq J \\ \gcd(h,k)=1 }} \xi_{h,k} = C_J,$$
which will allow us to focus on the most important cusps. Using this we are able to rewrite \eqref{equation: fourier coefficients integral} as
\begin{align*}
	a_{j,N}(n) = \sum_{\substack{0\leq h < k \leq J \\ \gcd(h,k)=1 }} \frac{1}{2\pi i} \int_{\xi_{h,k}} \frac{\CA_{j,N}(\tau)}{q^{g_{j,N}(n)+1}} dq,
\end{align*}
where we denoted $g_{j,N}(n)\coloneqq n+\frac{j^2}{4N}-\frac{1}{24}$ for simplicity.
Defining (see e.g.\@, \cite[equation (3.5)]{rademacher1938fourier}) 
\begin{align*}
\vartheta_{h,k}' \coloneqq \frac{1}{k(k_1+k)}, \quad \quad \vartheta_{h,k}'' \coloneqq \frac{1}{k(k_2+k)}
\end{align*}
and substituting $\tau = \frac hk +i(J^{-2}-i\phi)$ (arc length centered at $e^{2\pi i\frac hk}$) thus leads to
\begin{align}\label{equation: fourier coefficients integral over phi}
	a_{j,N}(n) 
	&= \sum_{\substack{0\leq h < k \leq J \\ \gcd(h,k)=1 }} e^{-\frac{2\pi ih}{k}g_{j,N}(n)} \int_{-\vartheta_{h,k}'}^{\vartheta_{h,k}''} e^{2\pi g_{j,N}(n)\left(J^{-2}-i\phi\right)} \CA_{j,N}\left(\frac hk +i\left(J^{-2}-i\phi\right)\right) d\phi.
\end{align}
Let $\omega\coloneqq J^{-2}-i\phi$. To better control the integrand's behavior near rational numbers we use the modular transformation $M_{h,k}$ from Theorem \ref{main Thm}. In addition equation \eqref{connection C and I} gives us
\begin{align} \label{equation transformation A}
\CA_{j,N}\left(\frac hk + i\omega \right) &= \sum_{r=1}^{N-1} \chi_{j,r} (N,M_{h,k}) \, \eta \lp \frac{h'}{k}+\frac{i}{k^2\omega}\rp^{-1} \,
\CI_{r,N,\frac{h'}{k}} \lp \frac{h'}{k}+\frac{i}{k^2\omega}\rp ,
\end{align}
where we used that
$$M_{h,k}\left(\frac hk + i\omega \right) = \frac{h'\left(\frac hk + i\omega \right)-\frac{hh'+1}{k}}{k\left(\frac hk + i\omega \right)-h}= \frac{ih'\omega-\frac 1k}{ik\omega} = \frac{h'}{k}+\frac{i}{k^2\omega}. $$

Taking a closer look at \eqref{equation transformation A} we obtain
\begin{small}
	\begin{align*}
	\CA_{j,N}\left(\frac hk + i\omega\right) 
	=& \sum_{r=1}^{N-1} \chi_{j,r} (N,M_{h,k}) \left(  \left(\eta \lp \frac{h'}{k}+\frac{i}{k^2\omega}\rp^{-1} -e^{-\frac{\pi i}{12}\left(\frac{h'}{k}+\frac{i}{k^2\omega}\right)}\right)
	\CI_{r,N,\frac{h'}{k}} \lp \frac{h'}{k}+\frac{i}{k^2\omega}\rp  \right. \\
	& \left. \quad \quad \quad \quad \quad \quad \quad\quad \quad \quad +\zeta_{24k}^{-h'}  \left(\CI_{r,N,\frac{h'}{k},\frac{1}{24}}^e \lp \frac{h'}{k}+\frac{i}{k^2\omega}\rp +\CI^*_{r,N,\frac{h'}{k},\frac{1}{24}} \lp \frac{h'}{k}+\frac{i}{k^2\omega}\rp\right) \rule{0pt}{0.6cm}\right).
	\end{align*} 
\end{small}
\hspace{-0.25cm} Note that in this calculation we set $d=\frac{1}{24}$ in the splitting of our Mordell-type integral such that our assumption from above simplifies to $\sqrt{\frac N6}\notin\Z$.

Plugging this into \eqref{equation: fourier coefficients integral over phi} gives us
\begin{align}\label{equation: splitting a_{j,N}}
a_{j,N}(n) = a_{\CI, j,N}(n) +a_{\CI^e, j,N}(n)  + a_{\CI^*, j,N}(n),
\end{align}
with 
\begin{small}
\begin{align*}
a_{\CI, j,N}(n) \coloneqq& \sum_{\substack{0\leq h < k \leq J \\ \gcd(h,k)=1 }} e^{-\frac{2\pi ih}{k}g_{j,N}(n)} \int_{-\vartheta_{h,k}'}^{\vartheta_{h,k}''} e^{2\pi g_{j,N}(n)\omega} \sum_{r=1}^{N-1} \chi_{j,r} (N,M_{h,k}) \\
& \quad \quad \quad \quad \quad \quad \quad \quad \quad \quad \quad \quad\times \left(\eta \lp \frac{h'}{k}+\frac{i}{k^2\omega}\rp^{-1} -e^{-\frac{\pi i}{12}\left(\frac{h'}{k}+\frac{i}{k^2\omega}\right)}\right)
\CI_{r,N,\frac{h'}{k}} \lp \frac{h'}{k}+\frac{i}{k^2\omega}\rp  d\phi,\\
 a_{\CI^e, j,N}(n) \coloneqq& \sum_{\substack{0\leq h < k \leq J \\ \gcd(h,k)=1 }} e^{-\frac{2\pi ih}{k}g_{j,N}(n)} \int_{-\vartheta_{h,k}'}^{\vartheta_{h,k}''} e^{2\pi g_{j,N}(n)\omega} \sum_{r=1}^{N-1} \chi_{j,r} (N,M_{h,k}) 
 \zeta_{24k}^{-h'}  \CI_{r,N,\frac{h'}{k},\frac{1}{24}}^e \lp \frac{h'}{k}+\frac{i}{k^2\omega}\rp d\phi, \\
 a_{\CI^*, j,N}(n) \coloneqq& \sum_{\substack{0\leq h < k \leq J \\ \gcd(h,k)=1 }} e^{-\frac{2\pi ih}{k}g_{j,N}(n)} \int_{-\vartheta_{h,k}'}^{\vartheta_{h,k}''} e^{2\pi g_{j,N}(n)\omega} \sum_{r=1}^{N-1} \chi_{j,r} (N,M_{h,k}) 
 \zeta_{24k}^{-h'}  \CI_{r,N,\frac{h'}{k},\frac{1}{24}}^* \lp \frac{h'}{k}+\frac{i}{k^2\omega}\rp d\phi.
 \end{align*}
\end{small}

\subsection{Principal part}
We now look at each of the terms in $a_{j,N}(n)$ seperately. We start with the part that contains the principal part, namely $a_{\CI^*, j,N}(n)$. Analogously to \cite{rademacher1938fourier} we split our integral as
$$ \int_{-\vartheta_{h,k}'}^{\vartheta_{h,k}''} = \int_{-\frac{1}{k(J+k)}}^{\frac{1}{k(J+k)}} + \int_{-\vartheta_{h,k}'}^{-\frac{1}{k(J+k)}} + \int_{\frac{1}{k(J+k)}}^{\vartheta_{h,k}''},$$
since we have 
$$ - \vartheta_{h,k}' =  -\frac{1}{k(k_1+k)}  \leq -\frac{1}{k(J+k)}\leq \frac{1}{k(J+k)} \leq \frac{1}{k(k_2+k)} = \vartheta_{h,k}'' .$$
Defining
\begin{small}
	\begin{align*}
	 a_{\CI^*, j,N,0}(n) \coloneqq& \sum_{\substack{0\leq h < k \leq J \\ \gcd(h,k)=1 }} e^{-\frac{2\pi ih}{k} g_{j,N}(n)} \int_{-\frac{1}{k(J+k)}}^{\frac{1}{k(J+k)}} e^{2\pi  g_{j,N}(n)\omega} \sum_{r=1}^{N-1} \chi_{j,r} (N,M_{h,k}) 
	\zeta_{24k}^{-h'}  \CI_{r,N,\frac{h'}{k},\frac{1}{24}}^* \lp \frac{h'}{k}+\frac{i}{k^2\omega}\rp d\phi, \\
	 a_{\CI^*, j,N,1}(n) \coloneqq& \sum_{\substack{0\leq h < k \leq J \\ \gcd(h,k)=1 }} e^{-\frac{2\pi ih}{k} g_{j,N}(n)} \int_{-\vartheta_{h,k}'}^{-\frac{1}{k(J+k)}} e^{2\pi  g_{j,N}(n)\omega} \sum_{r=1}^{N-1} \chi_{j,r} (N,M_{h,k}) 
	\zeta_{24k}^{-h'}  \CI_{r,N,\frac{h'}{k},\frac{1}{24}}^* \lp \frac{h'}{k}+\frac{i}{k^2\omega}\rp d\phi, \\
	 a_{\CI^*, j,N,2}(n) \coloneqq& \sum_{\substack{0\leq h < k \leq J \\ \gcd(h,k)=1 }} e^{-\frac{2\pi ih}{k} g_{j,N}(n)} \int_{\frac{1}{k(J+k)}}^{\vartheta_{h,k}''} e^{2\pi  g_{j,N}(n)\omega} \sum_{r=1}^{N-1} \chi_{j,r} (N,M_{h,k}) 
	\zeta_{24k}^{-h'}  \CI_{r,N,\frac{h'}{k},\frac{1}{24}}^* \lp \frac{h'}{k}+\frac{i}{k^2\omega}\rp d\phi,
\end{align*}
\end{small}
\hspace{-0.25cm} we thus obtain 
\begin{align}\label{equation: splitting a_{I^*,j,N}}
a_{\CI^*, j,N}(n) = a_{\CI^*, j,N,0}(n) + a_{ \CI^*, j,N,1}(n) + a_{\CI^*, j,N,2}(n). 
\end{align}
We go on by estimating $a_{\CI^*, j,N,0}(n)$. Using \eqref{representation of I^*} we see that
\begin{align*}
	 a_{\CI^*, j,N,0}(n) 
	 =& \frac{i}{\pi} \sum_{k=1}^J\int_{-\frac{1}{k(J+k)}}^{\frac{1}{k(J+k)}} e^{2\pi g_{j,N}(n)\omega} e^{\frac{2\pi}{24 k^2\omega}} \sum_{r=1}^{N-1} \sum_{\kappa\in\Z} \mathrm{P.V.} 
	\int_{-\sqrt{\frac{1}{24N}}}^{\sqrt{\frac{1}{24N}}} \frac{e^{-2 \pi N \frac{1}{k^2\omega} x^2}}{x-\lp \kappa + \frac{r}{2N} \rp } \, dx \\ 
	& \quad \quad \quad \quad \times  \sum_{\substack{0\leq h < k \\ \gcd(h,k)=1 }} e^{-\frac{2\pi ih}{k}g_{j,N}(n)} \chi_{j,r} (N,M_{h,k}) 
	\zeta_{24k}^{-h'} e^{2 \pi i N \left(\kappa+\frac{r}{2N}\right)^2\frac{h'}{k} } d\phi.
\end{align*}
Plugging in the definition of $K_{k,j,N}(n,r,\kappa)$ from \eqref{Kloostermansum}, which is well defined and a Kloosterman sum of modulus $k$, and taking the finite sum over $r$ out of the integral gives us
\begin{small}
	\begin{align*}
	a_{\CI^*, j,N,0}(n) &= \frac{i}{\pi} \sum_{k=1}^J \sum_{r=1}^{N-1} \int_{-\frac{1}{k(J+k)}}^{\frac{1}{k(J+k)}} e^{2\pi \left(g_{j,N}(n)\omega+\frac{1}{24 k^2\omega}\right)}  \sum_{\kappa\in\Z}  K_{k,j,N}(n,r,\kappa) \, \mathrm{P.V.} 
	\int_{-\sqrt{\frac{1}{24N}}}^{\sqrt{\frac{1}{24N}}} \frac{e^{-2 \pi N \frac{1}{k^2\omega} x^2}}{x-\lp \kappa + \frac{r}{2N} \rp } \, dx  \, d\phi.
\end{align*}
\end{small}
\hspace{-0.25cm} Note that, for arbitrary $\ell\in\Z$, we have 
\begin{align*}
\zeta_{24k}^{\left(24N\left((\kappa+\ell k)+\frac{r}{2N}\right)^2-1\right)h'} 
=& \zeta_{24k}^{\left(24N\left(\kappa+\frac{r}{2N}\right)^2-1\right)h'}
\end{align*}
and therefore $K_{k,j,N}(n,r,\kappa+\ell k)= K_{k,j,N}(n,r,\kappa)$.
Shifting $\kappa \mapsto \kappa +\ell k$ for $\ell\in\Z$ we thus obtain
\begin{small}
	\begin{align*}
	 a_{\CI^*, j,N,0}(n) =& \frac{i}{\pi} \sum_{k=1}^J \sum_{r=1}^{N-1} \sum_{\kappa=0}^{k-1}  K_{k,j,N}(n,r,\kappa) \\
	 &\times  \int_{-\frac{1}{k(J+k)}}^{\frac{1}{k(J+k)}} e^{2\pi \left(g_{j,N}(n)\omega+\frac{1}{24 k^2\omega}\right)} \lim\limits_{L\to\infty} \sum_{\ell=-L}^L  \mathrm{P.V.} 
	\int_{-\sqrt{\frac{1}{24N}}}^{\sqrt{\frac{1}{24N}}} \frac{e^{-2 \pi N \frac{1}{k^2\omega} x^2}}{x-\lp \kappa +\ell k + \frac{r}{2N} \rp } \, dx  \, d\phi.
	\end{align*}
\end{small}
\hspace{-0.25cm} Note that the convergence is uniform in our finite range so we are allowed to switch the order of the integral and the sum over $\ell$. 
Using the equality (see \cite[equation (3.10)]{bringmann2019framework})
\begin{align}\label{equation: cotangens}
\pi \cot (\pi x) = \lim\limits_{L\to\infty} \sum_{\ell=-L}^L \frac{1}{x+\ell},
\end{align} 
which holds for all $x\in\C\backslash\Z$, we thus obtain
	\begin{align*}
	a_{\CI^*, j,N,0}(n) =& -i \sum_{k=1}^J \sum_{r=1}^{N-1} \sum_{\kappa=0}^{k-1}  \frac{K_{k,j,N}(n,r,\kappa)}{k} \int_{-\frac{1}{k(J+k)}}^{\frac{1}{k(J+k)}} e^{2\pi \left(g_{j,N}(n)\omega+\frac{1}{24 k^2\omega}\right)}    \\
	&\times \mathrm{P.V.} 
	\int_{-\sqrt{\frac{1}{24N}}}^{\sqrt{\frac{1}{24N}}}   e^{-2 \pi N \frac{1}{k^2\omega} x^2} \cot\left(\pi \lp -\frac xk+ \frac{\kappa}{k} + \frac{r}{2Nk} \rp \right)  \, dx  \, d\phi.
	\end{align*}
Here the possible poles in $\Z$ have already been excluded by the principal value integral.

Note that we only have a simple pole in $x=\kappa+\ell k+\frac{r}{2N}$ if and only if $\kappa=\ell=0$ and $r< \sqrt{\frac{N}{6}}$. 
Since one can show that there exists a constant $C_{\varepsilon,k}$ ($C_{\varepsilon,k}\to0$ as $\varepsilon\to0$ and $k$ fix) such that
\begin{small}
	\begin{align*}
	\left| \left(\lim\limits_{\varepsilon\to 0}\!\!\!\!  \int\limits_{\substack{\left|x-\frac{r}{2N}\right|\geq \varepsilon \\ |x|\leq \sqrt{\frac{1}{24N}}}} \!\!\!\! e^{-2 \pi N \frac{1}{k^2\omega} x^2} \cot\left(\pi \lp -\frac xk + \frac{r}{2Nk} \rp \right) dx \right)- \!\!\!\!  \int\limits_{\substack{\left|x-\frac{r}{2N}\right|\geq \varepsilon \\ |x|\leq \sqrt{\frac{1}{24N}}}} \!\!\!\! e^{-2 \pi N \frac{1}{k^2\omega} x^2} \cot\left(\pi \lp -\frac xk + \frac{r}{2Nk} \rp \right) dx \right|\leq C_{\varepsilon,k}
\end{align*}
\end{small}
\hspace{-0.25cm} uniformly in $\phi$, using the Taylor expansion of the exponential together with \eqref{equation: cotangens}, we see that we only have integrals over compact subsets with continuous integrands and can additionally switch the integrals over $x$ and $\phi$ to get
\begin{align*}
a_{\CI^*, j,N,0}(n) = -i \sum_{k=1}^J \sum_{r=1}^{N-1} \sum_{\kappa=0}^{k-1} \frac{ K_{k,j,N}(n,r,\kappa)}{k} \, \mathrm{P.V.} 
\int_{-\sqrt{\frac{1}{24N}}}^{\sqrt{\frac{1}{24N}}} &\cot\left(\pi \lp -\frac xk+ \frac{\kappa}{k} + \frac{r}{2Nk} \rp \right)\\
&\times \int_{-\frac{1}{k(J+k)}}^{\frac{1}{k(J+k)}} e^{2\pi \left(g_{j,N}(n)\omega+\frac{1}{24 k^2\omega}-\frac{Nx^2}{k^2\omega}\right)} \, d\phi   \, dx .
\end{align*}
To evaluate the integral over $\phi$ we substitute $\omega = J^{-2}-i\phi$ to obtain
\begin{align*}
\int_{-\frac{1}{k(J+k)}}^{\frac{1}{k(J+k)}} e^{2\pi \left(g_{j,N}(n)\omega+\frac{1}{24 k^2\omega}-\frac{Nx^2}{k^2\omega}\right)} \, d\phi  = -i \int_{J^{-2}-\frac{i}{k(J+k)}}^{J^{-2}+\frac{i}{k(J+k)}} e^{2\pi \left(g_{j,N}(n)\omega+\frac{1}{k^2}\left(\frac{1}{24 }-Nx^2\right)\frac{1}{\omega}\right)} \, d\omega.
\end{align*}
Then we view it as an integral over the right vertical of a rectangle in the complex $\omega$-plane and denote the integrals over the other sides, $\gamma_1,\gamma_2$, and $\gamma_3$, by $R_1(x),R_2(x)$, and $R_3(x)$ respectively, where we dropped the dependence on the other parameters for simplicity (see Figure \ref{fig: integral rectangle}).

\begin{figure}[ht]
	\centering
	\begin{tikzpicture}  
	
	\begin{scope}[thick] 
	
	\draw[] (-2,-2) circle (1pt) node[left] {$-J^{-2}-i\frac{1}{k(J+k)}$};
	\draw[->] (-2,-2) -- (0,-2) node[below] {$\gamma_3$};  
	\draw[-] (0,-2) -- (2,-2) ;
	\draw[] (2,-2) circle (1pt) node[right] {$J^{-2}-i\frac{1}{k(J+k)}$};
	
	\draw[->] (2,-2) -- (2,-0.5)  ;  
	\draw[-] (2,-0.5) -- (2,1) ;
	\draw[] (2,1) circle (1pt) node[right] {$J^{-2}+i\frac{1}{k(J+k)}$};

	\draw[<-] (0,1) -- (2,1) ; 
	\draw[] (0,1)  node[above] {$\gamma_1$} ; 
	\draw[-] (0,1) -- (-2,1) ;
	\draw[] (-2,1) circle (1pt) node[left] {$-J^{-2}+i\frac{1}{k(J+k)}$};
	
	\draw[->] (-2,1) -- (-2,-0.5) node[left] {$\gamma_2$};  
	\draw[-] (-2,-0.5) -- (-2,-2) ;
		
	\end{scope}
	
	\end{tikzpicture}
	\caption{Rectangle in the complex $\omega$-plane.}
	\label{fig: integral rectangle}
\end{figure}
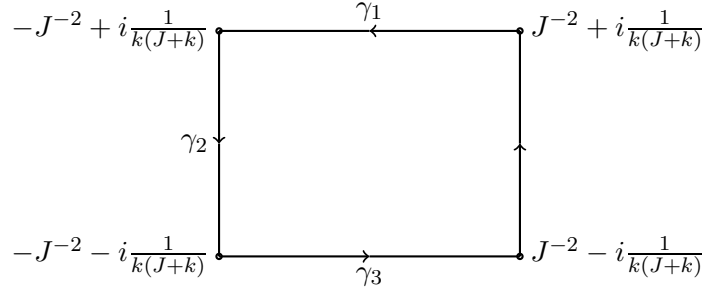

\noindent Let $R_{k,j,J,N}(n,x)$ denote the integral over the whole rectangle and let $R$ denote the rectangle itself with counterclockwise orientation such that
\begin{small}
\begin{align} \label{equation: a_0 rectangle splitting}
a_{\CI^*, j,N,0}(n) =& -i \sum_{k=1}^J \sum_{r=1}^{N-1} \sum_{\kappa=0}^{k-1}  \frac{K_{k,j,N}(n,r,\kappa)}{k} \mathrm{P.V.} 
\int_{-\sqrt{\frac{1}{24N}}}^{\sqrt{\frac{1}{24N}}} \cot\left(\pi \lp -\frac xk+ \frac{\kappa}{k} + \frac{r}{2Nk} \rp \right) \\
&\hspace{5.5cm} \times  (-i)\left(R_{k,j,J,N}(n,x) - R_1(x)-R_2(x)-R_3(x)\right)   \, dx . \notag 
\end{align}
\end{small}
\hspace{-0.25cm} We now have 
\begin{align*}
\frac{1}{2\pi i}R_{k,j,J,N}(n,x) = &   \frac{1}{2\pi i} \int_R e^{2\pi \left(g_{j,N}(n)\omega+\frac{1}{k^2}\left(\frac{1}{24 }-Nx^2\right)\frac{1}{\omega}\right)} \, d\omega\\
= &  \frac{1}{2\pi i} \int_R \sum_{\mu\geq0}  \frac{(2\pi g_{j,N}(n)\omega)^\mu}{\mu!} \sum_{\nu\geq0} \frac{\left(\frac{2\pi}{k^2\omega}\left(\frac{1}{24 }-Nx^2\right)\right)^\nu}{\nu!} \, d\omega,
\end{align*}
using the Taylor expansion of the exponential function. According to the residue theorem this integral equals zero unless there is a simple pole in $\omega=0$, which requires $\nu=\mu+1$.
Thus we obtain
\begin{align*}
	\frac{1}{2\pi i}R_{k,j,J,N}(n,x) 
	= & \frac{\sqrt{\frac{1}{24}-Nx^2}}{k\sqrt{g_{j,N}(n)}}\sum_{\mu\geq0} \frac{\left(\frac{2\pi\sqrt{g_{j,N}(n)}}{k} \sqrt{\frac{1}{24}-Nx^2}\right)^{2\mu +1} }{\mu!(\mu+1)!}.
\end{align*}
Using the representation (see e.g.\@, \cite[equation 10.25.2]{nist})
$$ I_\alpha(z) = \sum_{m\geq0} \frac{1}{m! \Gamma(m+\alpha+1)} \left(\frac{z}{2}\right)^{2m+\alpha}$$
of the $I$-Bessel function of first kind and order $\alpha$ we furthermore see that 
\begin{align*}
\frac{1}{2\pi i}R_{k,j,J,N}(n,x) 
= & \frac{\sqrt{\frac{1}{24}-Nx^2}}{k\sqrt{g_{j,N}(n)}} I_1\left(\frac{4\pi\sqrt{g_{j,N}(n)}}{k} \sqrt{\frac{1}{24}-Nx^2}\right).
\end{align*}

Plugging this into \eqref{equation: a_0 rectangle splitting} we obtain
\begin{small}
	\begin{align*}
	& \hspace{-4mm} a_{\CI^*, j,N,0}(n)\\
	=& - \sum_{k=1}^J \sum_{r=1}^{N-1} \sum_{\kappa=0}^{k-1}  \frac{K_{k,j,N}(n,r,\kappa)}{k} \mathrm{P.V.} 
	\int_{-\sqrt{\frac{1}{24N}}}^{\sqrt{\frac{1}{24N}}} \cot\left(\pi \lp -\frac xk+ \frac{\kappa}{k} + \frac{r}{2Nk} \rp \right)  \notag \\
	& \qquad \qquad \qquad \qquad \qquad \qquad \qquad \qquad \qquad  \times \frac{2\pi i \sqrt{\frac{1}{24}-Nx^2}}{k\sqrt{g_{j,N}(n)}} I_1\left(\frac{4\pi\sqrt{g_{j,N}(n)}}{k} \sqrt{\frac{1}{24}-Nx^2}\right)     \, dx \\
	& + \sum_{k=1}^J \sum_{r=1}^{N-1} \sum_{\kappa=0}^{k-1}  \frac{K_{k,j,N}(n,r,\kappa)}{k} \mathrm{P.V.} 
	\int_{-\sqrt{\frac{1}{24N}}}^{\sqrt{\frac{1}{24N}}} \cot\left(\pi \lp -\frac xk+ \frac{\kappa}{k} + \frac{r}{2Nk} \rp \right)  (R_1(x)+R_2(x)+R_3(x))\, dx\\
	 \eqqcolon& ~ M + E.
	\end{align*}
\end{small}
\hspace{-0.25cm} We are now left with estimating $E$. We can rewrite it as
\begin{small}
	\begin{align*}
	E_1 +E_2& + E_3\\
 \coloneqq&\sum_{k=1}^J \sum_{r=1}^{\left\lceil \sqrt{\frac{N}{6}}-1\right\rceil}   \frac{K_{k,j,N}(n,r,0)}{k} \mathrm{P.V.} 
\int_{-\sqrt{\frac{1}{24N}}}^{\sqrt{\frac{1}{24N}}} \cot\left(\pi \lp -\frac xk + \frac{r}{2Nk} \rp \right)  (R_1(x)+R_2(x)+R_3(x))\, dx  \\
  &+ \sum_{k=1}^J \sum_{r=\left\lceil \sqrt{\frac{N}{6}}\right\rceil}^{N-1}   \frac{K_{k,j,N}(n,r,0)}{k} 
\int_{-\sqrt{\frac{1}{24N}}}^{\sqrt{\frac{1}{24N}}} \cot\left(\pi \lp -\frac xk + \frac{r}{2Nk} \rp \right)  (R_1(x)+R_2(x)+R_3(x))\, dx  \\
 &+ \sum_{k=1}^J \sum_{r=1}^{N-1} \sum_{\kappa=1}^{k-1}  \frac{K_{k,j,N}(n,r,\kappa)}{k} 
\int_{-\sqrt{\frac{1}{24N}}}^{\sqrt{\frac{1}{24N}}} \cot\left(\pi \lp -\frac xk+ \frac{\kappa}{k} + \frac{r}{2Nk} \rp \right)  (R_1(x)+R_2(x)+R_3(x)) \, dx,
\end{align*}
\end{small}
\hspace{-0.25cm} since the integral over $x$ only has a simple pole in $x= \frac{r}{2N}$ for $\kappa=0$ and $r<\sqrt{\frac N6}$.

To bound the remaining sides of the rectangle we start with $R_1(x)$ and $R_3(x)$. On this paths of integration we have $\omega=u\pm i\frac{1}{k(J+k)}$, where $-J^{-2}\leq u \leq J^{-2}$, and 
\begin{align*}
\operatorname{Re}\left(\frac 1\omega\right) = \frac{u}{u^2+\frac{1}{k^2(J+k)^2}}= \frac{uk^2(J+k)^2}{u^2k^2(J+k)^2+1} <J^{-2} k^2 (J+k)^2 \leq 4k^2.
\end{align*}
Thus we see that the integrand is less than $e^{2\pi g_{j,N}(n)J^{-2} +8\pi\left(\frac{1}{24 }-Nx^2\right)} $ and obtain that
\begin{align}\label{equation: bounds on R_1 and R_3}
|R_1(x)| \text{ and } |R_3(x)| < 2J^{-2} e^{2\pi g_{j,N}(n)J^{-2} +8\pi\left(\frac{1}{24 }-Nx^2\right)}.
\end{align}

For $R_2(x)$ the path of integration is given by $\omega=-J^{-2}-iv$, where $-\frac{1}{k(J+k)}\leq v\leq \frac{1}{k(J+k)} $. Note that $g_{j,N}(n),~\frac{1}{24 }-Nx^2\geq 0$. Since the real part of $\omega$ is always $-J^{-2}<0$ and $\operatorname{Re}\left(\frac 1\omega\right) = \frac{-J^{-2}}{J^{-4}+v^2}<0$ we conclude that the integrand is $O(1)$ and therefore
\begin{align}\label{equation: bound on R_2}
|R_2(x)| <\frac{2}{k(J+k)} < 2 k^{-1}J^{-1}.
\end{align} 

From \eqref{equation: bounds on R_1 and R_3} and  \eqref{equation: bound on R_2} we conclude that we can bound $R_1(x)+R_2(x)+R_3(x)$ by 
\begin{align}\label{bound: other sides of the rectangle}
  O\left(k^{-1}J^{-1} e^{2\pi g_{j,N}(n)J^{-2} +8\pi\left(\frac{1}{24 }-Nx^2\right)} \right).
\end{align}

We start by evaluating $E_3$, which, using the calculations before together with \eqref{bound: Kloosterman sum}, equals
\begin{align*}
 O_N\left( \frac{n}{J} e^{2\pi g_{j,N}(n)J^{-2} } \sum_{k=1}^J \sum_{r=1}^{N-1} \sum_{\kappa=1}^{k-1} k^{-\frac 32+\varepsilon} 
\int_{-\sqrt{\frac{1}{24N}}}^{\sqrt{\frac{1}{24N}}} \left|\cot\left(\pi \lp -\frac xk+ \frac{\kappa}{k} + \frac{r}{2Nk} \rp \right)   \right|  \, dx\right).
\end{align*}
We note that
\begin{align*}
\left|\cot\left(\pi \lp -\frac xk+ \frac{\kappa}{k} + \frac{r}{2Nk} \rp \right)   \right| = \frac{\left|\cos\left(\pi \lp -\frac xk+ \frac{\kappa}{k} + \frac{r}{2Nk} \rp \right)\right| }{\left|\sin\left(\pi \lp -\frac xk+ \frac{\kappa}{k} + \frac{r}{2Nk} \rp \right)\right| }  \leq  \frac{1}{\left|\sin\left(\pi \lp -\frac xk+ \frac{\kappa}{k} + \frac{r}{2Nk} \rp \right)\right| }   .
\end{align*}
Similar to \cite[page 11]{bringmann2009rankpartition} we furthermore have
\begin{align*}
\left| \sin\left(\pi \lp -\frac xk+ \frac{\kappa}{k} + \frac{r}{2Nk} \rp \right)\right| =&  \left| \sin\left(\pi \left| -\frac xk+ \frac{\kappa}{k} + \frac{r}{2Nk} \right| \right)\right| \\
\gg& \min\left(\left\{\left|-\frac xk+ \frac{\kappa}{k} + \frac{r}{2Nk}\right|\right\},1-\left\{\left|-\frac xk+ \frac{\kappa}{k} + \frac{r}{2Nk}\right|\right\}\right)\\
=& \begin{cases}
\left\{\left|-\frac xk+ \frac{\kappa}{k} + \frac{r}{2Nk}\right|\right\} & \text{ if } 0\leq \left\{\left|-\frac xk+ \frac{\kappa}{k} + \frac{r}{2Nk}\right|\right\}\leq \frac 12,\\
1-\left\{\left|-\frac xk+ \frac{\kappa}{k} + \frac{r}{2Nk}\right|\right\} & \text{ else},
\end{cases}
\end{align*}
where $\{x\}\coloneqq x-\lfloor x\rfloor$ is the \textit{fractional part} of a real number $x$.

Taking a closer look at $-\frac xk+ \frac{\kappa}{k} + \frac{r}{2Nk}$ we observe that
\begin{align*}
-\frac xk+ \frac{\kappa}{k} + \frac{r}{2Nk} \leq&   1+\frac{1}{k\sqrt{24N}} -\frac{1}{2k} -\frac{1}{2Nk} =1-\frac{\sqrt{6N}+\sqrt{\frac{6}{N}}-1}{k\sqrt{24N}} < 1
\end{align*}
and 
\begin{align*}
-\frac xk+ \frac{\kappa}{k} + \frac{r}{2Nk}  \geq \frac{1}{k} \left(\frac{1}{2N}+\kappa-\frac{1}{\sqrt{24N}}\right) \geq 
\frac{47}{48 k}, 
\end{align*}
since we have $\kappa\geq1$.
In particular this gives us that 
\begin{align*}
&\hspace{-2cm}\min\left(\left\{\left|-\frac xk+ \frac{\kappa}{k} + \frac{r}{2Nk}\right|\right\},1-\left\{\left|-\frac xk+ \frac{\kappa}{k} + \frac{r}{2Nk}\right|\right\}\right)\\
=& \min\left(\frac{1}{k} \left( -x+\kappa +\frac{r}{2N}\right) ,1-\frac{1}{k} \left( -x+\kappa +\frac{r}{2N}\right) \right) \\
=& \begin{cases}
\frac{1}{k} \left( -x+\kappa +\frac{r}{2N}\right)  & \text{ if }  0\leq \frac{1}{k} \left( -x+\kappa +\frac{r}{2N}\right)  \leq \frac 12,  \\
1-\frac{1}{k} \left( -x+\kappa +\frac{r}{2N}\right)  & \text{ if }  \frac 12 < \frac{1}{k} \left( -x+\kappa +\frac{r}{2N}\right)  <1,
\end{cases} \\
=& \begin{cases}
\frac{1}{k} \left( -x+\kappa +\frac{r}{2N}\right) & \text{ if }   x\geq \frac{r}{2N}+\kappa-\frac k2,  \\
1-\frac{1}{k} \left( -x+\kappa +\frac{r}{2N}\right) & \text{ if }  x< \frac{r}{2N}+\kappa-\frac k2 .
\end{cases}
\end{align*}

Using this our $O$-term contributes to
\begin{scriptsize}
	\begin{align} \label{equation: O-term contribution E_3}
	& O_N\left( \frac{n}{J} e^{2\pi g_{j,N}(n)J^{-2} }  \sum_{k=1}^J k^{-\frac 32+\varepsilon} \sum_{r=1}^{N-1}  \sum_{\kappa=1}^{k-1} 
	\int_{-\sqrt{\frac{1}{24N}}}^{\sqrt{\frac{1}{24N}}}    \frac{1}{\min\left(\frac{1}{k} \left( -x+\kappa +\frac{r}{2N}\right),1-\frac{1}{k} \left( -x+\kappa +\frac{r}{2N}\right)\right)} \, dx\right) \notag \\
	=& O_N\left( \frac{n}{J} e^{2\pi g_{j,N}(n)J^{-2} }  \sum_{k=1}^J k^{-\frac 32+\varepsilon} \sum_{r=1}^{N-1} \sum_{\kappa=1}^{k-1}\left(
	\int\limits_{\substack{|x|\leq \sqrt{\frac{1}{24N}}\\x\geq \frac{r}{2N}+\kappa-\frac k2 }}    \frac{k}{-x+\kappa+\frac{r}{2N}} \, dx + \int\limits_{\substack{|x|\leq \sqrt{\frac{1}{24N}}\\x< \frac{r}{2N}+\kappa-\frac k2 }}    \frac{1}{1-\frac{1}{k} \left( -x+\kappa +\frac{r}{2N}\right)} \, dx\right)\right) \notag \\
	=& O_N\left( \frac{n}{J} e^{2\pi g_{j,N}(n)J^{-2} }  \sum_{k=1}^J k^{-\frac 12+\varepsilon} \sum_{r=1}^{N-1} \sum_{\kappa=1}^{k-1}\left(
	\delta_{\max\left(-\sqrt{\frac{1}{24N}}, \frac{r}{2N}+\kappa-\frac k2\right)<\sqrt{\frac{1}{24N}}} 
	\int_{\max\left(-\sqrt{\frac{1}{24N}}, \frac{r}{2N}+\kappa-\frac k2\right)}^{\sqrt{\frac{1}{24N}}}   \frac{1}{-x+\kappa+\frac{r}{2N}} \, dx \right.\right. \notag \\
	&\left.\left.\qquad+ \delta_{\min\left(\sqrt{\frac{1}{24N}}, \frac{r}{2N}+\kappa-\frac k2\right)>-\sqrt{\frac{1}{24N}}} 
	\int_{-\sqrt{\frac{1}{24N}}}^{\min\left(\sqrt{\frac{1}{24N}}, \frac{r}{2N}+\kappa-\frac k2\right)}    \frac{1}{k+x- \kappa - \frac{r}{2N}} \, dx\right)\right) \notag \\
	=& O_N\left( \frac{n}{J} e^{2\pi g_{j,N}(n)J^{-2} }  \sum_{k=1}^J k^{-\frac 12+\varepsilon} \sum_{r=1}^{N-1} \left(
	\sum_{\kappa=1}^{k-1}\delta_{\max\left(-\sqrt{\frac{1}{24N}}, \frac{r}{2N}+\kappa-\frac k2\right)<\sqrt{\frac{1}{24N}}} 
	\int_{\max\left(-\sqrt{\frac{1}{24N}}, \frac{r}{2N}+\kappa-\frac k2\right)}^{\sqrt{\frac{1}{24N}}}   \frac{1}{-x+\kappa+\frac{r}{2N}} \, dx \right.\right. \notag \\
	&\left.\left.\qquad+\sum_{\kappa'=1}^{k-1} \delta_{\min\left(\sqrt{\frac{1}{24N}}, \frac{r}{2N}+\frac k2-\kappa'\right)>-\sqrt{\frac{1}{24N}}} 
	\int_{-\sqrt{\frac{1}{24N}}}^{\min\left(\sqrt{\frac{1}{24N}}, \frac{r}{2N}+\frac k2-\kappa'\right)}    \frac{1}{\kappa'+x - \frac{r}{2N}} \, dx\right)\right),
	\end{align}
\end{scriptsize}
\hspace{-0.25cm} by substituting $\kappa'=k-\kappa$.

We note that in the first integral we have $\frac{r}{2N} -x  > -\sqrt{\frac{1}{48}}\eqqcolon C_1,$
while in the second integral we have $ x-\frac{r}{2N} >-\frac{25}{48}\eqqcolon C_2$.
Thus we obtain that \eqref{equation: O-term contribution E_3} equals
\begin{small}
	\begin{align*}
		& O_N\left( \frac{n}{J} e^{2\pi g_{j,N}(n)J^{-2} }  \sum_{k=1}^J k^{-\frac 12+\varepsilon} \sum_{r=1}^{N-1} \left(
		\sum_{\kappa=1}^{k-1}  \frac{1}{\kappa+C_1}\delta_{\max\left(-\sqrt{\frac{1}{24N}}, \frac{r}{2N}+\kappa-\frac k2\right)<\sqrt{\frac{1}{24N}}} 
		\int_{\max\left(-\sqrt{\frac{1}{24N}}, \frac{r}{2N}+\kappa-\frac k2\right)}^{\sqrt{\frac{1}{24N}}}   \, dx \right.\right. \notag \\
		&\left.\left.\qquad+\sum_{\kappa'=1}^{k-1}  \frac{1}{\kappa'+C_2} \delta_{\min\left(\sqrt{\frac{1}{24N}}, \frac{r}{2N}+\frac k2-\kappa'\right)>-\sqrt{\frac{1}{24N}}} 
		\int_{-\sqrt{\frac{1}{24N}}}^{\min\left(\sqrt{\frac{1}{24N}}, \frac{r}{2N}+\frac k2-\kappa'\right)}    \, dx\right)\right).
	\end{align*}
\end{small}
\hspace{-0.25cm} Using that (combining \cite[equations 5.7.6 and 5.11.2]{nist})
$$ \sum_{\kappa=1}^{k-1} \frac{1}{\kappa +C} =O\left(\log(k)\right),$$
as $k\to\infty$ and for a $C>-1$, this yields\footnote{We used that if we had a function $f(x)$ on $[1,\infty)$ with $f(x)=O(\log(x))$ as $x\to\infty$ we know that we have $f(x)\leq C_1 \log(x+1)$ for all $x\geq x_0$ and that $f$ is bounded by $\frac{f(x)}{\log(x+1)}\leq C_2$ on $[1,x_0)$. In total this would give us $\frac{f(x)}{\log(x+1)}\leq C \coloneqq C_1+C_2$ everywhere.}
	\begin{small}
		\begin{align}\label{bound: E_3}
		E_3 =  O_N\left( \frac{n}{J} e^{2\pi g_{j,N}(n)J^{-2} }  \sum_{k=1}^J k^{-\frac 12+\varepsilon} \sum_{r=1}^{N-1} N^{-\frac{1}{2}} \log(k) \right) = O_N\left(  n   e^{2\pi g_{j,N}(n)J^{-2} }   J^{-\frac 12+\varepsilon}  \log(J)\right),
		\end{align}
	\end{small}
\hspace{-0.25cm} which tends to $0$ as $J\to\infty$.

We go on by evaluating $E_2$. Using the fact that $\cot(z) = O\left(\frac{1}{z}\right)$ as $z \to 0$ we see that
\begin{align*}
\cot\left(\pi \lp -\frac xk + \frac{r}{2Nk} \rp\right) = O\left(\frac{k}{\pi \lp -x + \frac{r}{2N} \rp}\right) =O_N(k).
\end{align*}
Together with \eqref{bound: other sides of the rectangle} and \eqref{bound: Kloosterman sum} this gives us that 
\begin{align}\label{bound: E_2}
 E_2=O_N\left(  ne^{2\pi g_{j,N}(n)J^{-2} }   J^{-\frac 12+\varepsilon} \right) .
\end{align}

Lastly we evaluate the part with poles, namely $E_1$. Extracting the pole in $\frac{r}{2N}$ yields 
\begin{footnotesize}
	\begin{align*}
	E_1 =& \sum_{k=1}^J \sum_{r=1}^{\left\lceil \sqrt{\frac{N}{6}}-1\right\rceil}   \frac{K_{k,j,N}(n,r,0)}{k} 
	\int_{-\sqrt{\frac{1}{24N}}}^{\sqrt{\frac{1}{24N}}} \cot\left(\pi \lp -\frac xk + \frac{r}{2Nk} \rp \right)  \left(R_1(x)-R_1\left(\frac{r}{2N}\right)\right) \, dx \notag \\
	&+\sum_{k=1}^J \sum_{r=1}^{\left\lceil \sqrt{\frac{N}{6}}-1\right\rceil}   \frac{K_{k,j,N}(n,r,0)}{k} 
	\int_{-\sqrt{\frac{1}{24N}}}^{\sqrt{\frac{1}{24N}}} \cot\left(\pi \lp -\frac xk + \frac{r}{2Nk} \rp \right)  \left(R_2(x)-R_2\left(\frac{r}{2N}\right)\right) \, dx \notag \\
	& +\sum_{k=1}^J \sum_{r=1}^{\left\lceil \sqrt{\frac{N}{6}}-1\right\rceil}   \frac{K_{k,j,N}(n,r,0)}{k} 
	\int_{-\sqrt{\frac{1}{24N}}}^{\sqrt{\frac{1}{24N}}} \cot\left(\pi \lp -\frac xk + \frac{r}{2Nk} \rp \right)  \left(R_3(x)-R_3\left(\frac{r}{2N}\right)\right) \, dx \notag \\
	&+ \sum_{k=1}^J \sum_{r=1}^{\left\lceil \sqrt{\frac{N}{6}}-1\right\rceil}   \frac{K_{k,j,N}(n,r,0)}{k}\left( R_1\left(\frac{r}{2N}\right) +  R_2\left(\frac{r}{2N}\right)+  R_3\left(\frac{r}{2N}\right)\right) \mathrm{P.V.} 
	\int_{-\sqrt{\frac{1}{24N}}}^{\sqrt{\frac{1}{24N}}} \cot\left(\pi \lp -\frac xk + \frac{r}{2Nk} \rp \right)   \, dx.
\end{align*}
\end{footnotesize}

We first concentrate on the parts
$$\sum_{k=1}^J \sum_{r=1}^{\left\lceil \sqrt{\frac{N}{6}}-1\right\rceil}   \frac{K_{k,j,N}(n,r,0)}{k} 
\int_{-\sqrt{\frac{1}{24N}}}^{\sqrt{\frac{1}{24N}}} \cot\left(\pi \lp -\frac xk + \frac{r}{2Nk} \rp \right)  \left(R_m(x)-R_m\left(\frac{r}{2N}\right)\right) \, dx,$$
with $m\in\{1,2,3\}$.
Defining $f_m(x)\coloneqq R_m(x)-R_m(\frac{r}{2N})$ and using the Taylor expansion we obtain that
\begin{align*}
	f_m(x) = f_m\left(\frac{r}{2N}\right) +f_m'\left(\frac{r}{2N}\right)\left(x-\frac{r}{2N}\right)+ \dots + \frac{f_m^{(\ell)}\left(\frac{r}{2N}\right)}{\ell!}\left(x-\frac{r}{2N}\right)^\ell +\widetilde{R}_\ell(x),
\end{align*}
where $\widetilde{R}_\ell(x)$ is the \textit{remainder term} defined as
$$ \widetilde{R}_\ell(x)\coloneqq \frac{f_m^{(\ell+1)}\left(\xi_x\right)}{(\ell+1)!}\left(x-\frac{r}{2N}\right)^{\ell+1},$$
for some real $\xi_x$ betwen $\frac{r}{2N}$ and $x$.
Choosing $\ell=0$ we obtain
\begin{align*}
	f_m(x) = f_m\left(\frac{r}{2N}\right) +f_m'\left(\xi_x\right)\left(x-\frac{r}{2N}\right) = f_m'\left(\xi_x\right)\left(x-\frac{r}{2N}\right).
\end{align*}

Next we focus on $m\in\{1,3\}$ and thus have
\begin{align}\label{equation: first three parts of E_1}
	&\sum_{k=1}^J \sum_{r=1}^{\left\lceil \sqrt{\frac{N}{6}}-1\right\rceil}   \frac{K_{k,j,N}(n,r,0)}{k} 
	\int_{-\sqrt{\frac{1}{24N}}}^{\sqrt{\frac{1}{24N}}} \cot\left(\pi \lp -\frac xk + \frac{r}{2Nk} \rp \right)  \left(R_m(x)-R_m\left(\frac{r}{2N}\right)\right) \, dx \notag\\
	=& \sum_{k=1}^J \sum_{r=1}^{\left\lceil \sqrt{\frac{N}{6}}-1\right\rceil}   \frac{K_{k,j,N}(n,r,0)}{k} 
	\int_{-\sqrt{\frac{1}{24N}}}^{\sqrt{\frac{1}{24N}}} \left(x-\frac{r}{2N}\right)\cot\left(\pi \lp -\frac xk + \frac{r}{2Nk} \rp \right)  f_m'\left(\xi_x\right)  \, dx.
\end{align}
Now we want to bound $\left|\left(x-\frac{r}{2N}\right)\cot\left(\pi \lp -\frac xk + \frac{r}{2Nk} \rp \right)\right|$ and $\left|f_m'\left(\xi_x\right)\right|$ seperately.
For the first one we use the Taylor series expansion around $\frac{r}{2N}$ and see that
\begin{align*}
	\left(x-\frac{r}{2N}\right)\cot\left(\pi \lp -\frac xk + \frac{r}{2Nk} \rp \right) = -\frac{k}{\pi} +\frac{1}{3} \frac{\pi}{k} \left(\frac{r}{2N}-x\right)^2 +O\left(\frac{ \left(\frac{r}{2N}-x\right)^4}{k^3}\right) = O_N(k)
\end{align*}
as $k\to\infty$.
For the second one we see that
\begin{align*}
	\left|f_m'\left(\xi_x\right)\right| = \left|\left(R_m(x)-R_m\left(\frac{r}{2N}\right)\right)' \Big\vert_{x=\xi_x}\right| = \left|R_m'\left(\xi_x\right)\right|,
\end{align*}
since $R_m\left(\frac{r}{2N}\right)$ is independent of $x$. We note that
$$ R_m(x) = \int_{\gamma_m}  e^{2\pi \left(g_{j,N}(n)\omega+\frac{1}{k^2}\left(\frac{1}{24 }-Nx^2\right)\frac{1}{\omega}\right)} \, d\omega = \int_{\gamma_m}  e^{2\pi g_{j,N}(n)\omega} e^{\frac{\pi}{12k^2\omega}} e^{-\frac{2\pi Nx^2}{k^2\omega}} \, d\omega,$$
where we have an integral over a compact set and continuously differentiable integrand and thus are allowed to switch the integral with a derivative. This yields
\begin{align*}
	\left|R_m'\left(\xi_x\right)\right| = \left|\int_{\gamma_m}  e^{2\pi g_{j,N}(n)\omega} e^{\frac{\pi}{12k^2\omega}} \left(-\frac{4\pi N\xi_x}{k^2\omega}\right)e^{-\frac{2\pi N\xi_x^2}{k^2\omega}} \, d\omega \right|.
\end{align*}
Remember that we have $\omega=u\pm i\frac{1}{k(J+k)}$ with $-J^{-2} \leq u \leq J^{-2}$ and $\operatorname{Re}\left(\frac{1}{\omega}\right)\leq 4k^2$. Additionally we see that
\begin{align*}
	\left|\frac{1}{\omega}\right| =& \left(\frac{u^2+\left(\mp\frac{1}{k(J+k)}\right)^2}{\left(u^2+\frac{1}{k^2(J+k)^2}\right)^2}\right)^\frac{1}{2}  
	= \left(\frac{1}{u^2+\frac{1}{k^2(J+k)^2}}\right)^\frac{1}{2}  = k(J+k) \left(\frac{1}{k^2(J+k)^2u^2+1}\right)^\frac{1}{2} 
	<k(J+k)
\end{align*}
and therefore $\left|\frac{1}{k^2\omega}\right| < \frac{J+k}{k}$. Thus our integrand in this cases is less than
$$ \frac{4\pi N|\xi_x|(J+k)}{k} e^{2\pi g_{j,N}(n)J^{-2}} e^{\frac{\pi}{3}} e^{-8\pi N\xi_x^2}  $$ 
which finally yields that
\begin{align*}
	|R_1'(\xi_x)| \text{ and } |R_3'(\xi_x)| <&  8\pi N|\xi_x| \frac{J+k}{J^2k} e^{2\pi g_{j,N}(n)J^{-2}} e^{\frac{\pi}{3}} e^{-8\pi N\xi_x^2} \\
	\leq&  \frac{16\pi N|\xi_x|}{Jk} e^{2\pi g_{j,N}(n)J^{-2}} e^{\frac{\pi}{3}} e^{-8\pi N\xi_x^2}.
\end{align*}
This gives us that \eqref{equation: first three parts of E_1} equals
\begin{small}
	\begin{align*}
	&O\left(\sum_{k=1}^J \sum_{r=1}^{\left\lceil \sqrt{\frac{N}{6}}-1\right\rceil}   \frac{nk^{\frac{1}{2}+\varepsilon}}{k} 
	\int_{-\sqrt{\frac{1}{24N}}}^{\sqrt{\frac{1}{24N}}} k  \frac{16\pi N|\xi_x|}{Jk} e^{2\pi g_{j,N}(n)J^{-2}} \, dx\right) 
	=O_N\left(n   e^{2\pi g_{j,N}(n)J^{-2}} J^{-\frac{1}{2}+\varepsilon} \right).
\end{align*}
\end{small}

For $m=2$ we define $\delta\coloneqq \min(kJ^{-\frac{3}{4}}, |\sqrt{\frac{1}{24N}}-\frac{r}{2N}|)$ and split our integral over $x$ as follows 
\begin{footnotesize}
	\begin{align*}
&\sum_{k=1}^J \sum_{r=1}^{\left\lceil \sqrt{\frac{N}{6}}-1\right\rceil}   \frac{K_{k,j,N}(n,r,0)}{k} 
\int_{-\sqrt{\frac{1}{24N}}}^{\sqrt{\frac{1}{24N}}} \cot\left(\pi \lp -\frac xk + \frac{r}{2Nk} \rp \right)  \left(R_2(x)-R_2\left(\frac{r}{2N}\right)\right) \, dx \\
=& \sum_{k=1}^J \sum_{r=1}^{\left\lceil \sqrt{\frac{N}{6}}-1\right\rceil}   \frac{K_{k,j,N}(n,r,0)}{k} 
\left(\left(\int_{-\sqrt{\frac{1}{24N}}}^{\frac{r}{2N}-\delta}+ \int_{\frac{r}{2N}+\delta}^{\sqrt{\frac{1}{24N}}}+ \int_{\frac{r}{2N}-\delta}^{\frac{r}{2N}+\delta}\right) \cot\left(\pi \lp -\frac xk + \frac{r}{2Nk} \rp \right)  \left(R_2(x)-R_2\left(\frac{r}{2N}\right)\right) \, dx\right).
\end{align*}
\end{footnotesize}
Using \eqref{bound: Kloosterman sum} we see that
\begin{small}
	\begin{align*}
&\sum_{k=1}^J \sum_{r=1}^{\left\lceil \sqrt{\frac{N}{6}}-1\right\rceil}   \frac{K_{k,j,N}(n,r,0)}{k} 
\left(\left(\int_{-\sqrt{\frac{1}{24N}}}^{\frac{r}{2N}-\delta}+ \int_{\frac{r}{2N}+\delta}^{\sqrt{\frac{1}{24N}}}\right) \cot\left(\pi \lp -\frac xk + \frac{r}{2Nk} \rp \right)  \left(R_2(x)-R_2\left(\frac{r}{2N}\right)\right) \, dx\right) \\
=& O_N\left(\sum_{k=1}^J \sum_{r=1}^{\left\lceil \sqrt{\frac{N}{6}}-1\right\rceil}   \frac{k^{\frac{1}{2}+\varepsilon}n}{k} 
\left(\left(\int_{-\sqrt{\frac{1}{24N}}}^{\frac{r}{2N}-\delta}+ \int_{\frac{r}{2N}+\delta}^{\sqrt{\frac{1}{24N}}}\right)\left| \cot\left(\pi \lp -\frac xk + \frac{r}{2Nk} \rp \right)\right|  \left|R_2(x)-R_2\left(\frac{r}{2N}\right)\right| \, dx\right)\right).
\end{align*}
\end{small}
\hspace{-0.25cm} Since we are away from $x=\frac{r}{2N}$ we can bound 
$$ \left| \cot\left(\pi \lp -\frac xk + \frac{r}{2Nk} \rp \right)\right| = O\left(\frac{k}{\pi\left(-x+\frac{r}{2N}\right)}\right) = O\left(\frac{k}{\delta}\right)$$
and
$$ \left|R_2(x)-R_2\left(\frac{r}{2N}\right)\right| =O\left(\left|R_2(x)\right|\right) = O\left(k^{-1}J^{-1}\right), $$
using \eqref{equation: bound on R_2}. Thus we can simplify our $O$-term to
\begin{align*}
	& O_N\left(\sum_{k=1}^J \sum_{r=1}^{\left\lceil \sqrt{\frac{N}{6}}-1\right\rceil}   \frac{k^{\frac{1}{2}+\varepsilon}n}{k} \frac{k}{\delta} k^{-1}J^{-1}
	\left(\left(\int_{-\sqrt{\frac{1}{24N}}}^{\frac{r}{2N}-\delta}+ \int_{\frac{r}{2N}+\delta}^{\sqrt{\frac{1}{24N}}}\right) \, dx\right)\right) = O_N\left( n J^{-1} \sum_{k=1}^J  k^{-\frac{1}{2}+\varepsilon} \delta^{-1} \right) .
\end{align*}
For $\delta= kJ^{-\frac{3}{4}}$ this equals
$$ O_N\left( n J^{-\frac{1}{4}} \sum_{k=1}^J  k^{-\frac{3}{2}+\varepsilon}  \right) = O_N\left( n J^{-\frac{1}{4}} \right) ,$$
while for $\delta= |\sqrt{\frac{1}{24N}}-\frac{r}{2N}|$ it equals
$$ O_N\left( n J^{-1} \sum_{k=1}^J  k^{-\frac{1}{2}+\varepsilon}  \right) = O_N\left( n J^{-\frac{1}{2}+\varepsilon}  \right) .$$
For the last integral, the one close to $\frac{r}{2N}$, we use the Taylor expansion as seen in the cases $m\in\{1,3\}$ and have 
\begin{align}\label{equation: m=2 integral close to pole}
& \sum_{k=1}^J \sum_{r=1}^{\left\lceil \sqrt{\frac{N}{6}}-1\right\rceil}   \frac{K_{k,j,N}(n,r,0)}{k}  \int_{\frac{r}{2N}-\delta}^{\frac{r}{2N}+\delta} \cot\left(\pi \lp -\frac xk + \frac{r}{2Nk} \rp \right)  \left(R_2(x)-R_2\left(\frac{r}{2N}\right)\right) \, dx \notag \\
=& \sum_{k=1}^J \sum_{r=1}^{\left\lceil \sqrt{\frac{N}{6}}-1\right\rceil}   \frac{K_{k,j,N}(n,r,0)}{k}  \int_{\frac{r}{2N}-\delta}^{\frac{r}{2N}+\delta} \cot\left(\pi \lp -\frac xk + \frac{r}{2Nk} \rp \right)  \left(x-\frac{r}{2N}\right) R_2'(\xi_x) \, dx.
\end{align}
Since $|\cot\left(\pi \lp -\frac xk + \frac{r}{2Nk} \rp \right)  \left(x-\frac{r}{2N}\right)| =O_N(k)$, as seen before, we only need to look at $|R_2'(\xi_x)|$. 
Recall that on $\gamma_2$ we had $ \omega=-J^{-2}-iv$ with $-\frac{1}{k(J+k)}\leq v\leq \frac{1}{k(J+k)}$ and $\operatorname{Re}(\omega), \operatorname{Re}\left(\frac{1}{\omega}\right)<0$.
Using that
\begin{align*}
\left|\frac{1}{\omega}\right| &= \left(\frac{J^{-4}+v^2}{\left(J^{-4}+v^2\right)^2}\right)^\frac{1}{2} 
= \left(\frac{1}{J^{-4}+v^2}\right)^\frac{1}{2} = \left(\frac{J^4}{1+J^4v^2}\right)^\frac{1}{2}, 
\end{align*} 
we obatin
\begin{align*}
|R_2'(\xi_x)| =& \left|-i \int_{-\frac{1}{k(J+k)}}^{\frac{1}{k(J+k)}}  e^{2\pi g_{j,N}(n)\left(-J^{-2}-iv\right)} e^{\frac{2\pi}{k^2\left(-J^{-2}-iv\right)}\left(\frac{1}{24}-N\xi_x^2\right)} \left(-\frac{4\pi N\xi_x}{k^2\left(-J^{-2}-iv\right)}\right) \, dv \right| \\
\leq & \frac{4\pi N|\xi_x|}{k^2} \int_{-\frac{1}{k(J+k)}}^{\frac{1}{k(J+k)}}  \frac{1}{\left|-J^{-2}-iv\right|} \, dv = \frac{4\pi N|\xi_x|}{k^2} \int_{-\frac{1}{k(J+k)}}^{\frac{1}{k(J+k)}}  \left(\frac{J^4}{1+J^4v^2}\right)^\frac{1}{2} \, dv \\
\leq & \frac{4\pi N|\xi_x|}{k^2} \int_{-\frac{1}{k(J+k)}}^{\frac{1}{k(J+k)}}  \frac{J}{\sqrt{2v}} \, dv = \frac{4\pi N|\xi_x|J}{\sqrt{2}k^2} \int_{-\frac{1}{k(J+k)}}^{\frac{1}{k(J+k)}}  v^{-\frac{1}{2}} \, dv =O_N\left(\frac{J}{k^2\left(k(J+k)\right)^\frac{1}{2}}\right) \\
=& O_N\left(\frac{1}{k^2} \frac{J^\frac{1}{2}}{k^\frac{1}{2}}\right) =O_N\left(\frac{J^\frac{1}{2}}{k^\frac{5}{2}}\right).
\end{align*}
Thus we can simplify \eqref{equation: m=2 integral close to pole} to
\begin{align*}
O_N\left(\sum_{k=1}^J \sum_{r=1}^{\left\lceil \sqrt{\frac{N}{6}}-1\right\rceil}   \frac{nk^{\frac{1}{2}+\varepsilon}}{k} 
 \int_{\frac{r}{2N}-\delta}^{\frac{r}{2N}+\delta} k \frac{J^\frac{1}{2}}{k^\frac{5}{2}}\, dx\right) =& O_N\left(n J^\frac{1}{2} \sum_{k=1}^J    k^{-2+\varepsilon} \delta\right)  =  O_N\left(n J^{-\frac{1}{4}+\varepsilon} \right),
\end{align*}
where we used that $\delta\leq kJ^{-\frac{3}{4}}$.

Thus we overall see that 
\begin{footnotesize}
	\begin{align*}
	E_1 =& \sum_{k=1}^J \sum_{r=1}^{\left\lceil \sqrt{\frac{N}{6}}-1\right\rceil}   \frac{K_{k,j,N}(n,r,0)}{k}\left( R_1\left(\frac{r}{2N}\right) +  R_2\left(\frac{r}{2N}\right)+  R_3\left(\frac{r}{2N}\right)\right) \mathrm{P.V.} 
	\int_{-\sqrt{\frac{1}{24N}}}^{\sqrt{\frac{1}{24N}}} \cot\left(\pi \lp -\frac xk + \frac{r}{2Nk} \rp \right)   \, dx\\
	&+ O_N\left( n e^{2\pi g_{j,N}(n)J^{-2} }  J^{-\frac{1}{4}+\varepsilon}  \right).
\end{align*}
\end{footnotesize}

Substituting $y= \frac{r}{2N}-x$ yields that the principal value integral equals
\begin{small}
	\begin{align*}
	   \mathrm{P.V.} 
	 \int_{-\sqrt{\frac{1}{24N}}}^{\sqrt{\frac{1}{24N}}}& \cot\left(\frac{\pi}{k} \lp -x + \frac{r}{2N} \rp \right)   \, dx = \mathrm{P.V.} 
	 \int_{\frac{r}{2N}-\sqrt{\frac{1}{24N}}}^{\frac{r}{2N}+\sqrt{\frac{1}{24N}}} \cot\left(\frac{\pi}{k} y \right)   \, dy \\
	 =& \mathrm{P.V.} 
	 \int_{\frac{r}{2N}-\sqrt{\frac{1}{24N}}}^{\sqrt{\frac{1}{24N}}-\frac{r}{2N}} \cot\left(\frac{\pi}{k} y \right)   \, dy + 
	 \int_{\sqrt{\frac{1}{24N}}-\frac{r}{2N}}^{\frac{r}{2N}+\sqrt{\frac{1}{24N}}} \cot\left(\frac{\pi}{k} y \right)   \, dy .
\end{align*}
\end{small}
\hspace{-0.25cm} We obtain that the leftover principal value integral equals zero, since we have an odd function and a symmetric interval, and by using the Taylor expansion of $\cot(z)$ we see that
\begin{align*}
	& \int_{\sqrt{\frac{1}{24N}}-\frac{r}{2N}}^{\sqrt{\frac{1}{24N}}+\frac{r}{2N}} \cot\left(\frac{\pi}{k} y \right)  \, dy 
	= O_N\left(\int_{\sqrt{\frac{1}{24N}}-\frac{r}{2N}}^{\sqrt{\frac{1}{24N}}+\frac{r}{2N}} k  \, dy\right) 
	= O_N\left(k\right).
\end{align*}
Additionally we obtain, using \eqref{bound: other sides of the rectangle}, that
\begin{small}
	\begin{align*}
	R_1\left(\frac{r}{2N}\right) +  R_2\left(\frac{r}{2N}\right)+  R_3\left(\frac{r}{2N}\right) = O\left(k^{-1}J^{-1}e^{2\pi g_{j,N}(n)J^{-2}+8\pi\left(\frac{1}{24}-\frac{r^2}{4N}\right)}\right) = O_N\left(k^{-1}J^{-1}e^{2\pi g_{j,N}(n)J^{-2}}\right),
\end{align*}
\end{small}
\hspace{-0.25cm} which yields 
\begin{small}
	\begin{align*}
	&\sum_{k=1}^J \sum_{r=1}^{\left\lceil \sqrt{\frac{N}{6}}-1\right\rceil}   \frac{K_{k,j,N}(n,r,0)}{k}\left( R_1\left(\frac{r}{2N}\right) +  R_2\left(\frac{r}{2N}\right)+  R_3\left(\frac{r}{2N}\right)\right) \mathrm{P.V.} 
	\int_{-\sqrt{\frac{1}{24N}}}^{\sqrt{\frac{1}{24N}}} \cot\left(\pi \lp -\frac xk + \frac{r}{2Nk} \rp \right)   \, dx \\
	=& O_N\left(\sum_{k=1}^J \sum_{r=1}^{\left\lceil \sqrt{\frac{N}{6}}-1\right\rceil}   \frac{nk^{\frac{1}{2}+\varepsilon}}{k}  k^{-1}J^{-1}e^{2\pi g_{j,N}(n)J^{-2}}  k \right)  = O_N\left(n e^{2\pi g_{j,N}(n)J^{-2}} J^{-\frac{1}{2}+\varepsilon}    \right)
\end{align*}
\end{small}
\hspace{-0.25cm} using \eqref{bound: Kloosterman sum}.
Overall we therefore showed that
\begin{align}\label{bound: E_1}
E_1 = O_N\left(   n  e^{2\pi g_{j,N}(n)J^{-2} }   J^{-\frac{1}{4}+\varepsilon} \right).
\end{align}

Combining \eqref{bound: E_3}, \eqref{bound: E_2}, and \eqref{bound: E_1} finally gives 
\begin{small}
	\begin{align}\label{equation: final result a_{I^*,0}}
		a_{\CI^*, j,N,0}(n)
		=& - \frac{2\pi i }{\sqrt{g_{j,N}(n)}} \sum_{k=1}^J \sum_{r=1}^{N-1} \sum_{\kappa=0}^{k-1}  \frac{K_{k,j,N}(n,r,\kappa)}{k^2} \notag \\
		& \times  \mathrm{P.V.} 
		\int_{-\sqrt{\frac{1}{24N}}}^{\sqrt{\frac{1}{24N}}} \sqrt{\frac{1}{24}-Nx^2} \cot\left(\pi \lp -\frac xk+ \frac{\kappa}{k} + \frac{r}{2Nk} \rp \right)   I_1\left(\frac{4\pi\sqrt{g_{j,N}(n)}}{k} \sqrt{\frac{1}{24}-Nx^2}\right)    \, dx \notag\\
		&+O_N\left( n e^{2\pi g_{j,N}(n)J^{-2} }   \max\left(J^{-\frac 12+\varepsilon}  \log(J), J^{-\frac{1}{4}+\varepsilon} \right) \right).
	\end{align}
\end{small}

\subsection{Error part}

It is left to show that all the other parts of $a_{j,N}(n)$ are relatively small compared to $a_{\CI^*, j,N,0}(n)$.
Therefore we go on by analyzing $a_{\CI^*, j,N,1}(n)$ and $a_{\CI^*, j,N,2}(n)$, where we only discuss $a_{\CI^*, j,N,1}(n)$ in detail, since $a_{\CI^*, j,N,2}(n)$ can be treated accordingly. 

Remembering the definitions from above we have that
\begin{align*}
	a_{\CI^*, j,N,1}(n)=& \sum_{k=1}^J \sum_{\substack{0\leq h < k \\ \gcd(h,k)=1 }} e^{-\frac{2\pi ih}{k}g_{j,N}(n)} \int_{-\frac{1}{k(k_1+k)}}^{-\frac{1}{k(J+k)}} e^{2\pi g_{j,N}(n)\omega} \sum_{r=1}^{N-1} \chi_{j,r} (N,M_{h,k}) 
	\zeta_{24k}^{-h'}  	 
	\frac{i}{\pi}  e^{\frac{2 \pi}{24 k^2\omega}}  \\
	& \quad \quad \quad \quad \quad \quad \quad \quad \quad \quad \quad \quad\times \sum_{\kappa \in \Z} e^{2 \pi i N \left(\kappa+\frac{r}{2N}\right)^2\frac{h'}{k} } \, \mathrm{P.V.} 
	\int_{-\sqrt{\frac{1}{24N}}}^{\sqrt{\frac{1}{24N}}} \frac{e^{-2 \pi N \frac{1}{k^2\omega} x^2}}{x-\lp \kappa + \frac{r}{2N} \rp } \, dx \,   d\phi.
\end{align*}
Following \cite{rademacher1938fourier}, using \eqref{equation: intervals k_1 and k_2}, and splitting the integral over $\phi$ into integrals running over segments $[-\frac{1}{k\ell},-\frac{1}{k(\ell+1)}]$, for $k_1+k\leq\ell\leq J+k-1$, it follows that
\begin{align*}
	a_{\CI^*, j,N,1}(n)=& \sum_{k=1}^J \sum_{\substack{0\leq h < k \\ \gcd(h,k)=1 }} e^{-\frac{2\pi ih}{k}g_{j,N}(n)} \sum_{\ell=k_1+k}^{J+k-1} \int_{-\frac{1}{k\ell}}^{-\frac{1}{k(\ell+1)}} e^{2\pi g_{j,N}(n)\omega} \sum_{r=1}^{N-1} \chi_{j,r} (N,M_{h,k}) 
	\zeta_{24k}^{-h'}  	 
	\frac{i}{\pi}  e^{\frac{2 \pi}{24 k^2\omega}}  \\
	& \quad \quad \quad \quad \quad \quad \quad \quad \quad \quad \quad \quad\times \sum_{\kappa \in \Z} e^{2 \pi i N \left(\kappa+\frac{r}{2N}\right)^2\frac{h'}{k} } \, \mathrm{P.V.} 
	\int_{-\sqrt{\frac{1}{24N}}}^{\sqrt{\frac{1}{24N}}} \frac{e^{-2 \pi N \frac{1}{k^2\omega} x^2}}{x-\lp \kappa + \frac{r}{2N} \rp } \, dx \,   d\phi \\
	=& \frac{i}{\pi} \sum_{k=1}^J \sum_{\ell=J+1}^{J+k-1} \int_{-\frac{1}{k\ell}}^{-\frac{1}{k(\ell+1)}}   e^{2\pi g_{j,N}(n)\omega} e^{\frac{2 \pi}{24 k^2\omega}} \sum_{r=1}^{N-1}  	 
	   \sum_{\kappa \in \Z}  \, \mathrm{P.V.} 
	  \int_{-\sqrt{\frac{1}{24N}}}^{\sqrt{\frac{1}{24N}}} \frac{e^{-2 \pi N \frac{1}{k^2\omega} x^2}}{x-\lp \kappa + \frac{r}{2N} \rp } \, dx \\
	& \quad \quad \quad \quad \quad \quad \quad \quad\times  \sum_{\substack{0\leq h < k \\ \gcd(h,k)=1 \\ J<k_1+k\leq \ell }} e^{-\frac{2\pi ih}{k}g_{j,N}(n)} \chi_{j,r} (N,M_{h,k}) \zeta_{24k}^{-h'}  e^{2 \pi i N \left(\kappa+\frac{r}{2N}\right)^2\frac{h'}{k} }
	  \,   d\phi.
\end{align*}

The sum
\begin{align*}
K^*_{k,j,N}(n,r,\kappa,\ell)\coloneqq \sum_{\substack{0\leq h < k \\ \gcd(h,k)=1 \\ J<k_1+k\leq \ell }} \chi_{j,r} (N,M_{h,k})  \zeta_{24k}^{\left(24N \left(\kappa +\frac{r}{2N}\right)^2-1\right)h' -24g_{j,N}(n)h}
\end{align*}
is again well defined for $hh'\equiv -1 \pmod{k}$ and of modulus $k$. For $\mu \in \Z$ we observe that $K^*_{k,j,N}(n,r,\kappa+\mu k,\ell)=K^*_{k,j,N}(n,r,\kappa,\ell)$, which, shifting $\kappa \to \kappa+\mu k$, gives us 
\begin{align*}
a_{\CI^*, j,N,1}(n) =& \frac{i}{\pi} \sum_{k=1}^J \sum_{\ell=J+1}^{J+k-1}  \sum_{r=1}^{N-1} \sum_{\kappa=0}^{k-1} K^*_{k,j,N}(n,r,\kappa,\ell)   \\
& \times  \int_{-\frac{1}{k\ell}}^{-\frac{1}{k(\ell+1)}}   e^{2\pi g_{j,N}(n)\omega} e^{\frac{2 \pi}{24 k^2\omega}} 	 
\lim_{L\to\infty}\sum_{\mu =-L}^{L}  \, \mathrm{P.V.} 
\int_{-\sqrt{\frac{1}{24N}}}^{\sqrt{\frac{1}{24N}}} \frac{e^{-2 \pi N \frac{1}{k^2\omega} x^2}}{x-\lp \kappa+\mu k + \frac{r}{2N} \rp } \, dx \,   d\phi.
\end{align*}
Completely analogously to the calculations of $a_{\CI^*, j,N,0}(n)$ we obtain 
\begin{align*}
a_{\CI^*, j,N,1}(n) =& -i \sum_{k=1}^J \sum_{\ell=J+1}^{J+k-1} \sum_{r=1}^{N-1} \sum_{\kappa=0}^{k-1} \frac{K^*_{k,j,N}(n,r,\kappa,\ell)}{k} \, \mathrm{P.V.} 
\int_{-\sqrt{\frac{1}{24N}}}^{\sqrt{\frac{1}{24N}}} \cot\left(\pi \lp -\frac xk+ \frac{\kappa}{k} + \frac{r}{2Nk} \rp \right)\\
&\times  \int_{-\frac{1}{k\ell}}^{-\frac{1}{k(\ell+1)}}   e^{2\pi \left(g_{j,N}(n)\omega+\frac{1}{24 k^2\omega}-\frac{Nx^2}{k^2\omega}\right)} \, d\phi   \, dx .
\end{align*}

Since $\omega = J^{-2}-i\phi$ we obtain 
\begin{align*}
 \int_{-\frac{1}{k\ell}}^{-\frac{1}{k(\ell+1)}}   e^{2\pi \left(g_{j,N}(n)\omega+\frac{1}{24 k^2\omega}-\frac{Nx^2}{k^2\omega}\right)} \, d\phi = -i \int_{J^{-2}+\frac{i}{k(\ell+1)}}^{J^{-2}+\frac{i}{k\ell}}   e^{2\pi \left(g_{j,N}(n)\omega+\left(\frac{1}{24}-Nx^2\right)\frac{1}{k^2\omega}\right)} \, d\omega
\end{align*}
and note that for $v\coloneqq -\phi$
\begin{small}
	\begin{align*}
\operatorname{Re}\left(2\pi \left(g_{j,N}(n)\omega+\left(\frac{1}{24}-Nx^2\right)\frac{1}{k^2\omega}\right)\right) 
=& 2\pi \left(g_{j,N}(n)J^{-2}+\left(\frac{1}{24}-Nx^2\right)\frac{J^{-2}}{k^2\left(J^{-4}+v^2\right)}\right).
\end{align*}
\end{small}
\hspace{-0.25cm} Summing over all $\ell$ we see that $\frac{1}{k(J+k)}\leq v \leq \frac{1}{k(J+1)}$ and thus
\begin{align*}
	\frac{J^{-2}}{k^2\left(J^{-4}+v^2\right)} \leq \frac{J^{-2}}{k^2\left(J^{-4}+\left(\frac{1}{k(J+k)}\right)^2\right)} = \frac{(J+k)^2J^{-2}}{(J+k)^2k^2J^{-4}+1} \leq (J+k)^2 J^{-2} = \left(\frac{J+k}{J}\right)^2 <4.
\end{align*}
This gives that
\begin{align*}
	\operatorname{Re}\left(2\pi \left(g_{j,N}(n)\omega+\left(\frac{1}{24}-Nx^2\right)\frac{1}{k^2\omega}\right)\right) \leq 2\pi g_{j,N}(n)J^{-2}+8\pi\left(\frac{1}{24}-Nx^2\right)
\end{align*}
and therefore
\begin{small}
	\begin{align}\label{bound: Hl}
 H_\ell(x) \coloneqq \int_{J^{-2}+\frac{i}{k(\ell+1)}}^{J^{-2}+\frac{i}{k\ell}}   e^{2\pi \left(g_{j,N}(n)\omega+\left(\frac{1}{24}-Nx^2\right)\frac{1}{k^2\omega}\right)} \, d\omega &=  O_N \left( \left(\frac{1}{k\ell}-\frac{1}{k(\ell+1)}\right) e^{2\pi g_{j,N}(n)J^{-2}}\right).
\end{align}
\end{small}

Splitting $a_{\CI^*, j,N,1}(n)$ analogously to $E$ in the case $a_{\CI^*, j,N,0}(n)$ gives 
\begin{align*}
a_{\CI^*, j,N,1}(n) =& E^*_1 + E^*_2 + E^*_3 \\
=& - \sum_{k=1}^J \sum_{\ell=J+1}^{J+k-1} \sum_{r=1}^{\left\lceil \sqrt{\frac{N}{6}}-1\right\rceil}  \frac{ K^*_{k,j,N}(n,r,0,\ell)}{k} \, \mathrm{P.V.} 
\int_{-\sqrt{\frac{1}{24N}}}^{\sqrt{\frac{1}{24N}}} \cot\left(\pi \lp -\frac xk+ \frac{r}{2Nk} \rp \right) H_\ell(x)   \, dx \\
& - \sum_{k=1}^J \sum_{\ell=J+1}^{J+k-1} \sum_{r=\left\lceil \sqrt{\frac{N}{6}}\right\rceil}^{N-1}  \frac{ K^*_{k,j,N}(n,r,0,\ell)}{k} 
\int_{-\sqrt{\frac{1}{24N}}}^{\sqrt{\frac{1}{24N}}} \cot\left(\pi \lp -\frac xk+ \frac{r}{2Nk} \rp \right) H_\ell(x)   \, dx \\
& - \sum_{k=1}^J \sum_{\ell=J+1}^{J+k-1} \sum_{r=1}^{N-1} \sum_{\kappa=1}^{k-1} \frac{ K^*_{k,j,N}(n,r,\kappa,\ell)}{k} 
\int_{-\sqrt{\frac{1}{24N}}}^{\sqrt{\frac{1}{24N}}} \cot\left(\pi \lp -\frac xk+ \frac{\kappa}{k} + \frac{r}{2Nk} \rp \right) H_\ell(x)   \, dx.
\end{align*}

Following \cite[page 94]{estermann1929kloosterman} we set $h_0\coloneqq \ell -J$, which gives that $1\leq h_0 <k$. Additionally we define
\begin{align*}
g_1(m) \coloneqq \begin{cases}
1 & \text{ if } 0<m\leq h_0,\\
0 & \text{ if } h_0 < m \leq k,
\end{cases}
\end{align*}
and 
\begin{align*}
g_1(m+k)=g_1(m)
\end{align*}
for all integers $m$. Using this setting together with \eqref{equation: equivalences for k_1 and k_2} we obtain that 
$$ \delta_{J<k_1+k\leq \ell} = \delta_{0<k_1+k-J\leq h_0} = g_1(k_1+k-J) = g_1(k_1-J) = g_1(-h'-J) \eqqcolon \delta_{\sigma_1\leq h' < \sigma_2}, $$
for some $ 0 \leq \sigma_1 < \sigma_2\leq  k$.
This yields that the extra restriction on $k_1$ in the sum  $K^*_{k,j,N}(n,r,\kappa,\ell)$ constrains the choice of $h'$ to an interval mod $k$.
Therefore $K^*_{k,j,N}(n,r,\kappa,\ell)$ is an incomplete Kloosterman sum and can be bounded by \eqref{bound: Kloosterman sum} following the techniques by Lehner \cite[Section 10]{lehner1941partition}. We thus obtain that 
\begin{footnotesize}
	\begin{align*}
E^*_3 =& O_N\left(n e^{2\pi g_{j,N}(n)J^{-2}} \sum_{k=1}^J \sum_{\ell=J+1}^{J+k-1} \sum_{r=1}^{N-1} \sum_{\kappa=1}^{k-1} k^{-\frac{1}{2}+\varepsilon} \left(\frac{1}{k\ell}-\frac{1}{k(\ell+1)}\right) 
\int_{-\sqrt{\frac{1}{24N}}}^{\sqrt{\frac{1}{24N}}} \left|\cot\left(\pi \lp -\frac xk+ \frac{\kappa}{k} + \frac{r}{2Nk} \rp \right)\right|  \, dx\right) \\
=&  O_N\left(n e^{2\pi g_{j,N}(n)J^{-2}} \sum_{k=1}^J  \sum_{r=1}^{N-1} \sum_{\kappa=1}^{k-1} k^{-\frac{1}{2}+\varepsilon} \left(\frac{1}{k(J+1)}-\frac{1}{k(J+k)}\right) 
\int_{-\sqrt{\frac{1}{24N}}}^{\sqrt{\frac{1}{24N}}} \left|\cot\left(\pi \lp -\frac xk+ \frac{\kappa}{k} + \frac{r}{2Nk} \rp \right)\right|  \, dx\right) \\
=&  O_N\left(\frac{n}{J} e^{2\pi g_{j,N}(n)J^{-2}} \sum_{k=1}^J  \sum_{r=1}^{N-1} \sum_{\kappa=1}^{k-1} k^{-\frac{3}{2}+\varepsilon} 
\int_{-\sqrt{\frac{1}{24N}}}^{\sqrt{\frac{1}{24N}}} \left|\cot\left(\pi \lp -\frac xk+ \frac{\kappa}{k} + \frac{r}{2Nk} \rp \right)\right|  \, dx\right) \\
=& O_N\left(n e^{2\pi g_{j,N}(n)J^{-2}} J^{-\frac{1}{2}+\varepsilon} \log(J) \right) ,
\end{align*}
\end{footnotesize}
\hspace{-0.25cm} as seen before. Similar to that we redo the calculations for bounding $E_2$ to prove that
\begin{align*}
E^*_2 = O_N\left(n e^{2\pi g_{j,N}(n)J^{-2}} J^{-\frac{1}{2}+\varepsilon}  \right).
\end{align*}
Lastly we take care of $E^*_1$. We rewrite 
\begin{small}
	\begin{align*}
E^*_1 =&  - \sum_{k=1}^J \sum_{\ell=J+1}^{J+k-1} \sum_{r=1}^{\left\lceil \sqrt{\frac{N}{6}}-1\right\rceil}  \frac{ K^*_{k,j,N}(n,r,0,\ell)}{k} 
\int_{-\sqrt{\frac{1}{24N}}}^{\sqrt{\frac{1}{24N}}} \cot\left(\pi \lp -\frac xk+ \frac{r}{2Nk} \rp \right) \left(H_\ell(x) - H_\ell\left(\frac{r}{2N}\right)\right)  \, dx \\
&  - \sum_{k=1}^J \sum_{\ell=J+1}^{J+k-1} \sum_{r=1}^{\left\lceil \sqrt{\frac{N}{6}}-1\right\rceil}  \frac{K^*_{k,j,N}(n,r,0,\ell)}{k}  H_\ell\left(\frac{r}{2N}\right) \, \mathrm{P.V.} 
\int_{-\sqrt{\frac{1}{24N}}}^{\sqrt{\frac{1}{24N}}} \cot\left(\pi \lp -\frac xk+ \frac{r}{2Nk} \rp \right)    \, dx .
\end{align*}
\end{small}
\hspace{-0.25cm} Using $\delta=\min(kJ^{-\frac{3}{4}},|\sqrt{\frac{1}{24N}}-\frac{r}{2N}|)$ as before and splitting our integral yields
\begin{small}
	\begin{align*}
	E^*_1 =&  - \sum_{k=1}^J \sum_{\ell=J+1}^{J+k-1} \sum_{r=1}^{\left\lceil \sqrt{\frac{N}{6}}-1\right\rceil}  \frac{K^*_{k,j,N}(n,r,0,\ell)}{k} \\
	& \times \left(\int_{-\sqrt{\frac{1}{24N}}}^{\frac{r}{2N}-\delta} + \int_{\frac{r}{2N}+\delta}^{\sqrt{\frac{1}{24N}}} + \int_{\frac{r}{2N}-\delta}^{\frac{r}{2N}+\delta}\right) \cot\left(\pi \lp -\frac xk+ \frac{r}{2Nk} \rp \right) \left(H_\ell(x) - H_\ell\left(\frac{r}{2N}\right)\right)  \, dx \\
	&  - \sum_{k=1}^J \sum_{\ell=J+1}^{J+k-1} \sum_{r=1}^{\left\lceil \sqrt{\frac{N}{6}}-1\right\rceil}  \frac{ K^*_{k,j,N}(n,r,0,\ell)}{k}  H_\ell\left(\frac{r}{2N}\right) \, \mathrm{P.V.} 
	\int_{-\sqrt{\frac{1}{24N}}}^{\sqrt{\frac{1}{24N}}} \cot\left(\pi \lp -\frac xk+ \frac{r}{2Nk} \rp \right)    \, dx .
	\end{align*}
\end{small}
\hspace{-0.25cm} Using \eqref{bound: Kloosterman sum} and \eqref{bound: Hl} we see that 
\begin{footnotesize}
	\begin{align*}
	&- \sum_{k=1}^J \sum_{\ell=J+1}^{J+k-1} \sum_{r=1}^{\left\lceil \sqrt{\frac{N}{6}}-1\right\rceil}  \frac{K^*_{k,j,N}(n,r,0,\ell)}{k}  \left(\int_{-\sqrt{\frac{1}{24N}}}^{\frac{r}{2N}-\delta} + \int_{\frac{r}{2N}+\delta}^{\sqrt{\frac{1}{24N}}} \right) \cot\left(\pi \lp -\frac xk+ \frac{r}{2Nk} \rp \right) \left(H_\ell(x) - H_\ell\left(\frac{r}{2N}\right)\right)  \, dx \\
	=&  O_N\left(\sum_{k=1}^J \sum_{\ell=J+1}^{J+k-1} \sum_{r=1}^{\left\lceil \sqrt{\frac{N}{6}}-1\right\rceil}  \frac{nk^{\frac{1}{2}+\varepsilon}}{k}  \left(\int_{-\sqrt{\frac{1}{24N}}}^{\frac{r}{2N}-\delta} + \int_{\frac{r}{2N}+\delta}^{\sqrt{\frac{1}{24N}}} \right) \frac{k}{\delta} \left(\frac{1}{k\ell}-\frac{1}{k(\ell+1)}\right) e^{2\pi g_{j,N}(n)J^{-2}}  \, dx \right) \\
	=&  O_N\left(\frac{n}{J} e^{2\pi g_{j,N}(n)J^{-2}} \sum_{k=1}^J \sum_{r=1}^{\left\lceil \sqrt{\frac{N}{6}}-1\right\rceil}  k^{-\frac{1}{2}+\varepsilon}\delta^{-1}  \left(\int_{-\sqrt{\frac{1}{24N}}}^{\frac{r}{2N}-\delta} + \int_{\frac{r}{2N}+\delta}^{\sqrt{\frac{1}{24N}}} \right)     \, dx \right) = O_N\left(n e^{2\pi g_{j,N}(n)J^{-2}}  J^{-\frac{1}{4}} \right),
	\end{align*}
\end{footnotesize}
\hspace{-0.25cm} as before. 

Defining $h_\ell(x)\coloneqq H_\ell(x)-H_\ell(\frac{r}{2N})$ and using the Taylor expansion we obtain that
\begin{align*}
h_\ell(x) = h_\ell'\left(\xi_x\right)\left(x-\frac{r}{2N}\right),
\end{align*}
for some $\xi_x$ between $\frac{r}{2N}$ and $x$ and see that $\left|h_\ell'\left(\xi_x\right)\right|  = \left|H_\ell'\left(\xi_x\right)\right|$ as before.
For the integral close to $\frac{r}{2N}$ we first note that 
\begin{align*}
	\sum_{\ell=J+1}^{J+k-1} |H_\ell'(\xi_x)| =& O_N\left(\frac{J^\frac{1}{2}}{k^\frac{5}{2}} e^{2\pi g_{j,N}(n)J^{-2}}\right),
\end{align*}
using the techniques from before. Therefore we obtain
\begin{align*}
- \sum_{k=1}^J \sum_{\ell=J+1}^{J+k-1} &\sum_{r=1}^{\left\lceil \sqrt{\frac{N}{6}}-1\right\rceil}  \frac{ K^*_{k,j,N}(n,r,0,\ell)}{k} \int_{\frac{r}{2N}-\delta}^{\frac{r}{2N}+\delta}    \cot\left(\pi \lp -\frac xk+ \frac{r}{2Nk} \rp \right) \left(H_\ell(x) - H_\ell\left(\frac{r}{2N}\right)\right)  \, dx \\
=& O_N \left(\sum_{k=1}^J \sum_{\ell=J+1}^{J+k-1} \sum_{r=1}^{\left\lceil \sqrt{\frac{N}{6}}-1\right\rceil}   n k^{\frac{1}{2}+\varepsilon} \int_{\frac{r}{2N}-\delta}^{\frac{r}{2N}+\delta}    \left| H'_\ell (\xi_x)\right|  \, dx\right)= O_N\left(n e^{2\pi g_{j,N}(n)J^{-2}} J^{-\frac{1}{4}+\varepsilon} \right),
\end{align*}
which yields
\begin{small}
	\begin{align*}
	E^*_1 =&   - \sum_{k=1}^J \sum_{\ell=J+1}^{J+k-1} \sum_{r=1}^{\left\lceil \sqrt{\frac{N}{6}}-1\right\rceil}  \frac{ K^*_{k,j,N}(n,r,0,\ell)}{k}  H_\ell\left(\frac{r}{2N}\right) \, \mathrm{P.V.} 
	\int_{-\sqrt{\frac{1}{24N}}}^{\sqrt{\frac{1}{24N}}} \cot\left(\pi \lp -\frac xk+ \frac{r}{2Nk} \rp \right)    \, dx \\
	& + O_N\left(n e^{2\pi g_{j,N}(n)J^{-2}} J^{-\frac{1}{4}+\varepsilon} \right).
	\end{align*}
\end{small}
\hspace{-0.25cm} As seen before we have 
$$ \mathrm{P.V.} 
\int_{-\sqrt{\frac{1}{24N}}}^{\sqrt{\frac{1}{24N}}} \cot\left(\pi \lp -\frac xk+ \frac{r}{2Nk} \rp \right)    \, dx = O_N(k).$$
Additionally we obtain
\begin{align*}
\sum_{\ell=J+1}^{J+k-1}  H_\ell\left(\frac{r}{2N}\right) =& O_N\left(k^{-1}J^{-1} e^{2\pi g_{j,N}(n)J^{-2}} \right),
\end{align*}
using \eqref{bound: Hl}.
Analogously to above this yields
\begin{align*}
- \sum_{k=1}^J \sum_{\ell=J+1}^{J+k-1} \sum_{r=1}^{\left\lceil \sqrt{\frac{N}{6}}-1\right\rceil} & \frac{K^*_{k,j,N}(n,r,0,\ell)}{k}  H_\ell\left(\frac{r}{2N}\right) \, \mathrm{P.V.} 
\int_{-\sqrt{\frac{1}{24N}}}^{\sqrt{\frac{1}{24N}}} \cot\left(\pi \lp -\frac xk+ \frac{r}{2Nk} \rp \right)    \, dx \\
=& O_N\left(n e^{2\pi g_{j,N}(n)J^{-2}} J^{-\frac{1}{2}+\varepsilon} \right)
\end{align*}
and therefore overall
\begin{align*}
E^*_1 = O_N\left(n e^{2\pi g_{j,N}(n)J^{-2}} J^{-\frac{1}{4}+\varepsilon}\right).
\end{align*}

This finally gives 
\begin{align*}
a_{\CI^*, j,N,1}(n) = O_N\left( n e^{2\pi g_{j,N}(n)J^{-2} }   \max\left(J^{-\frac 12+\varepsilon}  \log(J), J^{-\frac{1}{4}+\varepsilon} \right) \right)
\end{align*}
and the analog result for $a_{\CI^*, j,N,2}(n)$.

Plugging \eqref{equation: final result a_{I^*,0}} and this results into \eqref{equation: splitting a_{I^*,j,N}} yields that
\begin{small}
	\begin{align}\label{equation: final result a_{I^*}}
	a_{\CI^*, j,N}(n)
	=& - \frac{2\pi i }{\sqrt{g_{j,N}(n)}} \sum_{k=1}^J \sum_{r=1}^{N-1} \sum_{\kappa=0}^{k-1}  \frac{K_{k,j,N}(n,r,\kappa)}{k^2} \notag \\
	& \times  \mathrm{P.V.} 
	\int_{-\sqrt{\frac{1}{24N}}}^{\sqrt{\frac{1}{24N}}} \sqrt{\frac{1}{24}-Nx^2} \cot\left(\pi \lp -\frac xk+ \frac{\kappa}{k} + \frac{r}{2Nk} \rp \right)   I_1\left(\frac{4\pi\sqrt{g_{j,N}(n)}}{k} \sqrt{\frac{1}{24}-Nx^2}\right)    \, dx \notag\\
	&+O_N\left( n e^{2\pi g_{j,N}(n)J^{-2} }   \max\left(J^{-\frac 12+\varepsilon}  \log(J), J^{-\frac{1}{4}+\varepsilon} \right) \right).
	\end{align}
\end{small}

Lastly we have to take care of $a_{\CI, j,N}(n)$ and $a_{\CI^e, j,N}(n)$.
From the definition and \eqref{representation of I^e} we have 
\begin{align*}
	a_{\CI^e, j,N}(n) =& \sum_{k=1}^J \sum_{\substack{0\leq h < k  \\ \gcd(h,k)=1 }} e^{-\frac{2\pi ih}{k}g_{j,N}(n)} \int_{-\vartheta_{h,k}'}^{\vartheta_{h,k}''} e^{2\pi g_{j,N}(n)\omega} \sum_{r=1}^{N-1} \chi_{j,r} (N,M_{h,k}) 
	\zeta_{24k}^{-h'}  \frac{i}{\pi}  e^{\frac{2 \pi}{24 k^2\omega}}  \\
	& \quad \quad \quad \quad \quad \quad \quad \quad \quad \quad \quad \quad\times \sum_{\kappa \in \Z} e^{2 \pi i N \left(\kappa+\frac{r}{2N}\right)^2\frac{h'}{k} } \,\mathrm{P.V.} 
	\int_{|x|\geq\sqrt{\frac{1}{24N}}} \frac{e^{-2 \pi N \frac{1}{k^2\omega} x^2}}{x-\lp \kappa + \frac{r}{2N} \rp } \, dx \,   d\phi .
\end{align*}
By decomposing the Farey segment $-\vartheta_{h,k}' \leq \phi \leq -\vartheta_{h,k}''$ as seen before we obtain
\begin{align*}
	a_{\CI^e, j,N}(n) =& \sum_{k=1}^J \sum_{\substack{0\leq h < k  \\ \gcd(h,k)=1 }} e^{-\frac{2\pi ih}{k}g_{j,N}(n)} \int_{-\frac{1}{k(J+k)}}^{\frac{1}{k(J+k)}} e^{2\pi g_{j,N}(n)\omega} \sum_{r=1}^{N-1} \chi_{j,r} (N,M_{h,k}) 
	\zeta_{24k}^{-h'}  \frac{i}{\pi}  e^{\frac{2 \pi}{24 k^2\omega}}  \\
	& \quad \quad \quad \quad \quad \quad \quad \quad \quad \quad \quad \quad\times \sum_{\kappa \in \Z} e^{2 \pi i N \left(\kappa+\frac{r}{2N}\right)^2\frac{h'}{k} } \,\mathrm{P.V.} 
	\int_{|x|\geq\sqrt{\frac{1}{24N}}} \frac{e^{-2 \pi N \frac{1}{k^2\omega} x^2}}{x-\lp \kappa + \frac{r}{2N} \rp } \, dx \,   d\phi \\
	&+\sum_{k=1}^J \sum_{\substack{0\leq h < k  \\ \gcd(h,k)=1 }} e^{-\frac{2\pi ih}{k}g_{j,N}(n)} \sum_{\ell=k_1+k}^{J+k-1} \int_{-\frac{1}{k\ell}}^{-\frac{1}{k(\ell+1)}} e^{2\pi g_{j,N}(n)\omega} \sum_{r=1}^{N-1} \chi_{j,r} (N,M_{h,k}) 
	\zeta_{24k}^{-h'}  \frac{i}{\pi}  e^{\frac{2 \pi}{24 k^2\omega}}  \\
	& \quad \quad \quad \quad \quad \quad \quad \quad \quad \quad \quad \quad\times \sum_{\kappa \in \Z} e^{2 \pi i N \left(\kappa+\frac{r}{2N}\right)^2\frac{h'}{k} } \,\mathrm{P.V.} 
	\int_{|x|\geq\sqrt{\frac{1}{24N}}} \frac{e^{-2 \pi N \frac{1}{k^2\omega} x^2}}{x-\lp \kappa + \frac{r}{2N} \rp } \, dx \,   d\phi  \\
	& + \sum_{k=1}^J \sum_{\substack{0\leq h < k  \\ \gcd(h,k)=1 }} e^{-\frac{2\pi ih}{k}g_{j,N}(n)} \sum_{\ell=k_2+k}^{J+k-1} \int_{\frac{1}{k(\ell+1)}}^{\frac{1}{k\ell}}  e^{2\pi g_{j,N}(n)\omega} \sum_{r=1}^{N-1} \chi_{j,r} (N,M_{h,k}) 
	\zeta_{24k}^{-h'}  \frac{i}{\pi}  e^{\frac{2 \pi}{24 k^2\omega}}  \\
	& \quad \quad \quad \quad \quad \quad \quad \quad \quad \quad \quad \quad\times \sum_{\kappa \in \Z} e^{2 \pi i N \left(\kappa+\frac{r}{2N}\right)^2\frac{h'}{k} } \,\mathrm{P.V.} 
	\int_{|x|\geq\sqrt{\frac{1}{24N}}} \frac{e^{-2 \pi N \frac{1}{k^2\omega} x^2}}{x-\lp \kappa + \frac{r}{2N} \rp } \, dx \,   d\phi \\
	\eqqcolon& a_{\CI^e, j,N,1}(n) + a_{\CI^e, j,N,2}(n) + a_{\CI^e, j,N,3}(n).
\end{align*}
We first note that 
\begin{align*}
	a_{\CI^e, j,N,1}(n) =& \frac{i}{\pi} \sum_{k=1}^J \sum_{r=1}^{N-1} \sum_{\kappa=0}^{k-1} K_{k,j,N}(n,r,\kappa) \\
	&\times \int_{-\frac{1}{k(J+k)}}^{\frac{1}{k(J+k)}} e^{2\pi g_{j,N}(n)\omega} \lim_{L\to\infty} \sum_{\mu=-L}^L   \, \mathrm{P.V.} 
	\int_{|x|\geq\sqrt{\frac{1}{24N}}} \frac{e^{-2 \pi \left(Nx^2-\frac{1}{24}\right) \frac{1}{k^2\omega} }}{x-\lp \kappa + \mu k + \frac{r}{2N} \rp } \, dx  \, d\phi
\end{align*}
completely analogously to the calculations of $a_{\CI^*, j,N,0}(n)$, while
\begin{align*}
a_{\CI^e, j,N,2}(n)  =& \frac{i}{\pi} \sum_{k=1}^J \sum_{\ell=J+1}^{J+k-1}  \sum_{r=1}^{N-1} \sum_{\kappa=0}^{k-1} K^*_{k,j,N}(n,r,\kappa,\ell)   \\
& \times  \int_{-\frac{1}{k\ell}}^{-\frac{1}{k(\ell+1)}}   e^{2\pi g_{j,N}(n)\omega}	 
\lim_{L\to\infty}\sum_{\mu =-L}^{L}  \, \mathrm{P.V.} 
\int_{|x|\geq\sqrt{\frac{1}{24N}}} \frac{e^{-2 \pi \left(Nx^2-\frac{1}{24}\right) \frac{1}{k^2\omega} }}{x-\lp \kappa+\mu k + \frac{r}{2N} \rp } \, dx \,   d\phi
\end{align*}
and 
\begin{align*}
a_{\CI^e, j,N,3}(n) =& \frac{i}{\pi} \sum_{k=1}^J \sum_{\ell=J+1}^{J+k-1}  \sum_{r=1}^{N-1} \sum_{\kappa=0}^{k-1} \widetilde{K}_{k,j,N}(n,r,\kappa,\ell)   \\
& \times  \int_{\frac{1}{k(\ell+1)}}^{\frac{1}{k\ell}}   e^{2\pi g_{j,N}(n)\omega}	 
\lim_{L\to\infty}\sum_{\mu =-L}^{L}  \, \mathrm{P.V.} 
\int_{|x|\geq\sqrt{\frac{1}{24N}}} \frac{e^{-2 \pi \left(Nx^2-\frac{1}{24}\right) \frac{1}{k^2\omega} }}{x-\lp \kappa+\mu k + \frac{r}{2N} \rp } \, dx \,   d\phi
\end{align*}
analogously to the calculation of $a_{\CI^*, j,N,1}(n)$, where
\begin{align*}
	\widetilde{K}_{k,j,N}(n,r,\kappa,\ell)\coloneqq \sum_{\substack{0\leq h < k \\ \gcd(h,k)=1 \\ J<k_2+k\leq \ell }} \chi_{j,r} (N,M_{h,k})  \zeta_{24k}^{\left(24N \left(\kappa +\frac{r}{2N}\right)^2-1\right)h' -24g_{j,N}(n)h}.
\end{align*}

For $a_{\CI,j,N}(n)$ we have 
\begin{small}
	\begin{align*}
a_{\CI, j,N}(n) =& \sum_{k=1}^J \sum_{\substack{0\leq h < k  \\ \gcd(h,k)=1 }} e^{-\frac{2\pi ih}{k}g_{j,N}(n)} \int_{-\vartheta_{h,k}'}^{\vartheta_{h,k}''} e^{2\pi g_{j,N}(n)\omega} \sum_{r=1}^{N-1} \chi_{j,r} (N,M_{h,k})   \\
& \times \left(\eta\left(\frac{h'}{k}+\frac{i}{k^2\omega}\right)^{-1} -e^{-\frac{\pi i}{12}\left(\frac{h'}{k}+\frac{i}{k^2\omega}\right)}\right) \frac{i}{\pi} \sum_{\kappa \in \Z} e^{2 \pi i N \left(\kappa+\frac{r}{2N}\right)^2\frac{h'}{k} }
\, \mathrm{P.V.}
\int_{-\infty}^\infty \frac{e^{-2 \pi \frac{N}{k^2\omega} x^2}}{x-\lp \kappa + \frac{r}{2N} \rp } \,dx \,   d\phi 
\end{align*}
\end{small}
\hspace{-0.25cm} from the definition and \eqref{representation of I}.
 Additionally we observe that
 \begin{align*}
 \eta\left(\frac{h'}{k}+\frac{i}{k^2\omega}\right)^{-1} -e^{-\frac{\pi i}{12}\left(\frac{h'}{k}+\frac{i}{k^2\omega}\right)} =  \sum_{m\geq 1} p(m) e^{-\frac{2\pi m }{k^2\omega}}  \zeta_{24k}^{(24m-1)h'} e^{\frac{2\pi}{24k^2\omega}},
 \end{align*}
 where $p(m)$ is the partition function.
 
By decomposing the Farey segment $-\vartheta_{h,k}' \leq \phi \leq -\vartheta_{h,k}''$ as seen before we obtain
\begin{scriptsize}
	\begin{align*}
a_{\CI, j,N}(n) =& \sum_{k=1}^J \sum_{\substack{0\leq h < k  \\ \gcd(h,k)=1 }} e^{-\frac{2\pi ih}{k}g_{j,N}(n)}   \int_{-\frac{1}{k(J+k)}}^{\frac{1}{k(J+k)}} \sum_{m\geq 1} p(m) e^{-\frac{2\pi m }{k^2\omega}} e^{2\pi g_{j,N}(n)\omega} \sum_{r=1}^{N-1} \chi_{j,r} (N,M_{h,k})   \zeta_{24k}^{(24m-1)h'}    \\
& \hspace{4cm}\times \frac{i}{\pi} e^{\frac{2\pi}{24k^2\omega}} \sum_{\kappa \in \Z} e^{2 \pi i N \left(\kappa+\frac{r}{2N}\right)^2\frac{h'}{k} }
\, \mathrm{P.V.}
\int_{-\infty}^\infty \frac{e^{-2 \pi \frac{N}{k^2\omega} x^2}}{x-\lp \kappa + \frac{r}{2N} \rp } \,dx \,   d\phi \\
&+ \sum_{k=1}^J \sum_{\substack{0\leq h < k  \\ \gcd(h,k)=1 }} e^{-\frac{2\pi ih}{k}g_{j,N}(n)}   \sum_{\ell=k_1+k}^{J+k-1} \int_{-\frac{1}{k\ell}}^{-\frac{1}{k(\ell+1)}} \sum_{m\geq 1} p(m) e^{-\frac{2\pi m }{k^2\omega}} e^{2\pi g_{j,N}(n)\omega} \sum_{r=1}^{N-1} \chi_{j,r} (N,M_{h,k})   \zeta_{24k}^{(24m-1)h'}    \\
& \hspace{4cm}\times \frac{i}{\pi} e^{\frac{2\pi}{24k^2\omega}} \sum_{\kappa \in \Z} e^{2 \pi i N \left(\kappa+\frac{r}{2N}\right)^2\frac{h'}{k} }
\, \mathrm{P.V.}
\int_{-\infty}^\infty \frac{e^{-2 \pi \frac{N}{k^2\omega} x^2}}{x-\lp \kappa + \frac{r}{2N} \rp } \,dx \,   d\phi \\
&+ \sum_{k=1}^J \sum_{\substack{0\leq h < k  \\ \gcd(h,k)=1 }} e^{-\frac{2\pi ih}{k}g_{j,N}(n)}   \sum_{\ell=k_2+k}^{J+k-1}\int_{\frac{1}{k(\ell+1)}}^{\frac{1}{k\ell}} \sum_{m\geq 1} p(m) e^{-\frac{2\pi m }{k^2\omega}} e^{2\pi g_{j,N}(n)\omega} \sum_{r=1}^{N-1} \chi_{j,r} (N,M_{h,k})   \zeta_{24k}^{(24m-1)h'}   \\
& \hspace{4cm}\times  \frac{i}{\pi} e^{\frac{2\pi}{24k^2\omega}} \sum_{\kappa \in \Z} e^{2 \pi i N \left(\kappa+\frac{r}{2N}\right)^2\frac{h'}{k} }
\, \mathrm{P.V.}
\int_{-\infty}^\infty \frac{e^{-2 \pi \frac{N}{k^2\omega} x^2}}{x-\lp \kappa + \frac{r}{2N} \rp } \,dx \,   d\phi\\
\eqqcolon& a_{\CI, j,N,1}(n) + a_{\CI, j,N,2}(n) + a_{\CI, j,N,3}(n).
\end{align*}
\end{scriptsize}

Define
\begin{align*}
K_{k,j,N}(n,m,r,\kappa)\coloneqq& \sum_{\substack{0\leq h < k \\ \gcd(h,k)=1  }} \chi_{j,r} (N,M_{h,k})  \zeta_{24k}^{\left(24N \left(\kappa +\frac{r}{2N}\right)^2+24m-1\right)h' -24g_{j,N}(n)h}, \\
K^*_{k,j,N}(n,m,r,\kappa,\ell)\coloneqq& \sum_{\substack{0\leq h < k \\ \gcd(h,k)=1 \\ J<k_1+k\leq \ell }} \chi_{j,r} (N,M_{h,k})  \zeta_{24k}^{\left(24N \left(\kappa +\frac{r}{2N}\right)^2+24m-1\right)h' -24g_{j,N}(n)h}, \\
\widetilde{K}_{k,j,N}(n,m,r,\kappa,\ell)\coloneqq& \sum_{\substack{0\leq h < k \\ \gcd(h,k)=1 \\ J<k_2+k\leq \ell }} \chi_{j,r} (N,M_{h,k})  \zeta_{24k}^{\left(24N \left(\kappa +\frac{r}{2N}\right)^2+24m-1\right)h' -24g_{j,N}(n)h},
\end{align*}
which are all well-defined Kloosterman sums of modulus $k$.
We thus note that 
\begin{align*}
a_{\CI, j,N,1}(n) =& \frac{i}{\pi} \sum_{k=1}^J \sum_{r=1}^{N-1} \sum_{\kappa=0}^{k-1}  \int_{-\frac{1}{k(J+k)}}^{\frac{1}{k(J+k)}} \sum_{m\geq 1} p(m) e^{-\frac{2\pi\left( m-\frac{1}{24}\right) }{k^2\omega}} K_{k,j,N}(n,m,r,\kappa) e^{2\pi g_{j,N}(n)\omega} \\
&\times  \lim_{L\to\infty} \sum_{\mu=-L}^L   \, \mathrm{P.V.} 
\int_{-\infty}^\infty \frac{e^{-2 \pi N x^2 \frac{1}{k^2\omega}}}{x-\lp \kappa + \mu k + \frac{r}{2N} \rp } \, dx  \, d\phi
\end{align*}
completely analogously to the calculations of $a_{\CI^*, j,N,0}(n)$, while
\begin{align*}
a_{\CI, j,N,2}(n)  =& \frac{i}{\pi} \sum_{k=1}^J \sum_{\ell=J+1}^{J+k-1}  \sum_{r=1}^{N-1} \sum_{\kappa=0}^{k-1}   \int_{-\frac{1}{k\ell}}^{-\frac{1}{k(\ell+1)}} \sum_{m\geq 1} p(m) e^{-\frac{2\pi \left( m-\frac{1}{24}\right) }{k^2\omega}} K^*_{k,j,N}(n,m,r,\kappa,\ell) e^{2\pi g_{j,N}(n)\omega}	   \\
& \times 
\lim_{L\to\infty}\sum_{\mu =-L}^{L}  \, \mathrm{P.V.} 
\int_{-\infty}^\infty \frac{e^{-2 \pi N x^2 \frac{1}{k^2\omega}}}{x-\lp \kappa+\mu k + \frac{r}{2N} \rp } \, dx \,   d\phi
\end{align*}
and 
\begin{align*}
a_{\CI, j,N,3}(n) =& \frac{i}{\pi} \sum_{k=1}^J \sum_{\ell=J+1}^{J+k-1}  \sum_{r=1}^{N-1} \sum_{\kappa=0}^{k-1}   \int_{\frac{1}{k(\ell+1)}}^{\frac{1}{k\ell}} \sum_{m\geq 1} p(m) e^{-\frac{2\pi\left( m-\frac{1}{24}\right) }{k^2\omega}} \widetilde{K}_{k,j,N}(n,m,r,\kappa,\ell)  e^{2\pi g_{j,N}(n)\omega}	 
  \\
& \times \lim_{L\to\infty}\sum_{\mu =-L}^{L}  \, \mathrm{P.V.} 
\int_{-\infty}^\infty \frac{e^{-2 \pi N x^2 \frac{1}{k^2\omega}}}{x-\lp \kappa+\mu k + \frac{r}{2N} \rp } \, dx \,   d\phi
\end{align*}
analogously to the calculation of $a_{\CI^*, j,N,1}(n)$.

To be able to bound all parts of $a_{\CI^e, j,N}(n) $ resp. $a_{\CI, j,N}(n) $ we need the following lemma.
\begin{lemma}\label{lemma: bound pv integral}
	For $0\leq d <N$, $N,~k,~\omega,~\kappa$, and $r$ as above, and $2\sqrt{Nd}\notin\Z\backslash\{0\}$ we have
	\begin{align*}
		\lim_{L\to\infty}\sum_{\mu =-L}^{L}  \, \mathrm{P.V.} 
	\int_{|x|\geq\sqrt{\frac{d}{N}}} \frac{e^{-2 \pi \left(Nx^2-d\right) \frac{1}{k^2\omega} }}{x-\lp \kappa+\mu k + \frac{r}{2N} \rp } \, dx = O_N\left(\frac{1}{\min\left(\kappa+\frac{r}{2N}, k-\kappa-\frac{r}{2N}\right)}\right),
	\end{align*}
	as $k\to\infty$.
\end{lemma}
\begin{proof}
We follow the ideas of \cite[Proof of Lemma 3.3]{bringmann2019framework}. Combining the integral over negative and positive reals gives us that
\begin{align*}
	 \mathcal{P}_{d,N}(k,\omega,\kappa,r) \coloneqq& \mathrm{P.V.} 
	\int_{|x|\geq\sqrt{\frac{d}{N}}} \frac{e^{-2 \pi \left(Nx^2-d\right) \frac{1}{k^2\omega} }}{x-\lp \kappa+\mu k + \frac{r}{2N} \rp } \, dx  \\
	=&  \mathrm{P.V.} 
	\int_{\sqrt{\frac{d}{N}}}^\infty e^{-2 \pi \left(Nx^2-d\right) \frac{1}{k^2\omega} }\left( \frac{1}{x-\lp \kappa+\mu k + \frac{r}{2N} \rp } -  \frac{1}{x+\lp \kappa+\mu k + \frac{r}{2N} \rp }\right) \, dx \\
	=& 2 \lp \kappa+\mu k + \frac{r}{2N} \rp   \, \mathrm{P.V.} 
	\int_{\sqrt{\frac{d}{N}}}^\infty  \frac{e^{-2 \pi N\left(x^2 -\frac{d}{N}\right)\frac{1}{k^2\omega} }}{x^2-\lp \kappa+\mu k + \frac{r}{2N} \rp^2 } \, dx \\
	=& \lp \kappa+\mu k + \frac{r}{2N} \rp   \, \mathrm{P.V.} 
	\int_{0}^\infty  \frac{e^{-2 \pi Nu \frac{1}{k^2\omega} }}{\sqrt{u+\frac{d}{N}} \left(u+\frac{d}{N}-\lp \kappa+\mu k + \frac{r}{2N} \rp^2\right) } \, du,
\end{align*}
where we substituted $u=x^2-\frac{d}{N}$ in the last step. We go on by writing
\begin{align*}
	\frac{1}{u+\frac{d}{N}-\lp \kappa+\mu k + \frac{r}{2N} \rp^2} =&  \left(\frac{1}{u+\frac{d}{N}-\lp \kappa+\mu k + \frac{r}{2N} \rp^2} + \frac{1}{\lp \kappa+\mu k+\frac{r}{2N} \rp^2}\right) - \frac{1}{\lp \kappa+\mu k+\frac{r}{2N} \rp^2} \\
	=& \frac{u+\frac{d}{N}}{\lp \kappa+\mu k+\frac{r}{2N} \rp^2 \left(u+\frac{d}{N}-\lp \kappa+\mu k + \frac{r}{2N} \rp^2\right)}  - \frac{1}{\lp \kappa+\mu k+\frac{r}{2N} \rp^2}
\end{align*}
and consider the contribution of each term seperately, where we denote them by $ \mathcal{P}_{d,N,1}(k,\omega,\kappa,r)$, resp. $ \mathcal{P}_{d,N,2}(k,\omega,\kappa,r)$. We start by looking at 
\begin{align*}
	\mathcal{P}_{d,N,2}(k,\omega,\kappa,r) =&  - \frac{1}{ \kappa+\mu k+\frac{r}{2N} } 
	\int_{0}^\infty  \frac{e^{-2 \pi Nu \frac{1}{k^2\omega} }}{\sqrt{u+\frac{d}{N}} } \, du.
\end{align*}
Using that $\operatorname{Re}(\frac{2N}{k^2\omega})\geq N$ together with $u\geq 0$ we see that
\begin{align*}
	\left| \int_{0}^\infty  \frac{e^{-2 \pi Nu \frac{1}{k^2\omega} }}{\sqrt{u+\frac{d}{N}} } \, du \right| \leq  \int_{0}^\infty  \frac{e^{-\pi Nu}}{\sqrt{u+\frac{d}{N}} } \, du = O_{N}(1).
\end{align*}
Using \eqref{equation: cotangens} we additionally see that
\begin{align*}
	-\lim_{L\to\infty}\sum_{\mu =-L}^{L} \frac{1}{ \kappa+\mu k+\frac{r}{2N} } = - \frac{1}{k} \lim_{L\to\infty}\sum_{\mu =-L}^{L} \frac{1}{ \mu + \frac{\kappa}{k}+\frac{r}{2Nk} } = -\frac{\pi}{k} \cot\left(\pi\left(\frac{\kappa}{k}+\frac{r}{2Nk}\right)\right).
\end{align*}
Since $0 < \frac{\kappa}{k}+\frac{r}{2Nk} < 1$ we obtain that (see \cite[page 13]{bringmann2019framework})
$$\left|\cot\left(\pi \left(\frac{\kappa}{k}+\frac{r}{2Nk}\right)\right)\right| \ll \frac{1}{\frac{\kappa}{k}+\frac{r}{2Nk}}+\frac{1}{1-\frac{\kappa}{k}-\frac{r}{2Nk}},$$ 
which yields that
\begin{align*}
	\lim_{L\to\infty}\sum_{\mu =-L}^{L} \mathcal{P}_{d,N,2}(k,\omega,\kappa,r) =& O_N\left(\frac{1}{k} \left(\frac{1}{ \frac{\kappa}{k}+\frac{r}{2Nk}} + \frac{1}{1- \frac{\kappa}{k}-\frac{r}{2Nk}}\right) \right) \\
	=& O_N\left(\frac{1}{\min\left(\kappa+\frac{r}{2N}, k-\kappa-\frac{r}{2N}\right)}\right).
\end{align*}

Next we look at $\mathcal{P}_{d,N,1}(k,\omega,\kappa,r)$. We start by writing
\begin{align*}
	 \mathcal{P}_{d,N,1}(k,\omega,\kappa,r) =&  \frac{1}{\kappa+\mu k + \frac{r}{2N}}   \, \mathrm{P.V.} 
	 \int_{0}^\infty  \frac{e^{-2 \pi Nu \frac{1}{k^2\omega} } \sqrt{u+\frac{d}{N}} }{ u+\frac{d}{N}-\lp \kappa+\mu k + \frac{r}{2N} \rp^2} \, du.
\end{align*}
Our poles thus lie in $u=\lp \kappa+\mu k + \frac{r}{2N} \rp^2 -\frac{d}{N}\in\R$. We further investigate that since $d< N$ we only have a pole in $0$ if $\kappa=\mu =0$ and $r=2\sqrt{Nd}$, which cannot happen since  $r\geq 1$ and we assumed that $2\sqrt{Nd}\notin\Z\backslash\{0\}$.

We rewrite our principal value integral as the average of the paths $\gamma_{\varepsilon,+}$ and $\gamma_{\varepsilon,-}$, where $\gamma_{\varepsilon,+}$, resp. $\gamma_{\varepsilon,-}$, is the path of integration along the positive real axis taking a semicircular path of radius $\varepsilon$ above, resp. below, the pole (see Figure \ref{pathgamma+}).

\begin{figure}[ht] 
	\centering
	\begin{tikzpicture}

		\begin{scope}[thick]
			\draw                  (0,0) circle (.5pt) node[left] {$0$} -- (2.25,0); 
	
			\draw[->] (2.25,0)  arc (180:90:.5cm); 	
			\draw[] (3.25,0) arc (0:90:0.5cm);		
			
			\draw[] (2.75,0) circle (.5pt);
			
			\draw[->]     (3.25,0) --  (5.75,0) node[right] {$\infty$} ;

			\draw                  (7,0) --  (9.5,0);
			
			\draw[->] (9.5,0)  arc (-180:-90:.5cm); 	
			\draw[] (10.5,0) arc (0:-90:0.5cm);
			
			\draw[] (10,0) circle (.5pt);
			
			\draw[->]               (10.5,0) --  (13,0)   node[right] {$\infty$};
		\end{scope}
	\end{tikzpicture}
	\caption{The paths of integration $\gamma_{\varepsilon,+}$, resp. $\gamma_{\varepsilon,-}$.} 
	\label{pathgamma+}
\end{figure}
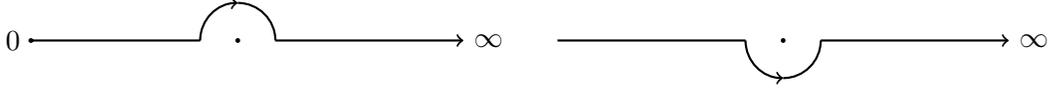
We obtain that 
\begin{align*}
	\mathcal{P}_{d,N,1}(k,\omega,\kappa,r) =&  \frac{1}{\kappa+\mu k + \frac{r}{2N}}  \lim_{\varepsilon\to 0 } \left( \frac{1}{2} \left(	\int_{\gamma_{\varepsilon,+}} + 	\int_{\gamma_{\varepsilon,-}} \right)  \frac{e^{-2 \pi Nu \frac{1}{k^2\omega} } \sqrt{u+\frac{d}{N}} }{ u+\frac{d}{N}-\lp \kappa+\mu k + \frac{r}{2N} \rp^2} \, du\right).
\end{align*}
We note that 
\begin{align*}
	\operatorname{Re}\left(\frac{2N}{k^2\omega}e^{\pm \frac{\pi i}{4}}\right) = \frac{\operatorname{Re}\left(\frac{2N}{k^2\omega}\right)}{\sqrt{2}} \mp \frac{\operatorname{Im}\left(\frac{2N}{k^2\omega}\right)}{\sqrt{2}},
\end{align*}
since $e^{\pm\frac{\pi i}{4}}= \frac{1}{\sqrt{2}}(1\pm i)$. Choosing the $\pm$ to be $-\sgn(\operatorname{Im}(\frac{2N}{k^2\omega}))$ thus gives us that
\begin{align*}
	\operatorname{Re}\left(\frac{2N}{k^2\omega}e^{\pm \frac{\pi i}{4}}\right) = \frac{\operatorname{Re}\left(\frac{2N}{k^2\omega}\right)}{\sqrt{2}} + \frac{\left|\operatorname{Im}\left(\frac{2N}{k^2\omega}\right)\right|}{\sqrt{2}}\geq \frac{\operatorname{Re}\left(\frac{2N}{k^2\omega}\right)}{\sqrt{2}} \geq \frac{N}{\sqrt{2}} ,
\end{align*}
since $\operatorname{Re}(\frac{2N}{k^2\omega})\geq N$. This means we either have $\operatorname{Re}(\frac{2N}{k^2\omega}e^{ \frac{\pi i}{4}})\geq \frac{N}{\sqrt{2}}$ or $\operatorname{Re}(\frac{2N}{k^2\omega}e^{- \frac{\pi i}{4}})\geq \frac{N}{\sqrt{2}}$.
Using Cauchy's Theorem we now want to rotate our paths of integration either to $e^{\frac{\pi i}{4}}\R^+$ if we have $\operatorname{Re}(\frac{2N}{k^2\omega}e^{ \frac{\pi i}{4}})\geq \frac{N}{\sqrt{2}}$ or to $e^{-\frac{\pi i}{4}}\R^+$ if $\operatorname{Re}(\frac{2N}{k^2\omega}e^{- \frac{\pi i}{4}})\geq \frac{N}{\sqrt{2}}$ picking up the residues from the poles that lie on the real line. Since both rotations follow the same argument we from now on assume that we have $\operatorname{Re}(\frac{2N}{k^2\omega}e^{ \frac{\pi i}{4}})\geq \frac{N}{\sqrt{2}}$. Note that for this rotation we only pick up poles by performing the rotation on $\gamma_{\varepsilon,-}$ as can by easily seen in Figure \ref{paths of integration}.

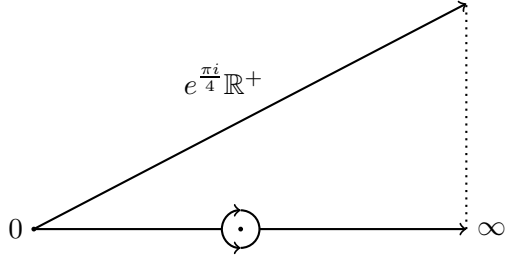
\begin{figure}[ht] 
	\centering
	\begin{tikzpicture}

		\begin{scope}[thick]
			\draw[]                  (0,0) circle (.5pt) node[left] {$0$} -- (2.5,0); 
			
			\draw[->] (2.5,0)  arc (180:90:.25cm); 	
			\draw[] (3,0) arc (0:90:0.25cm);	
			\draw[->] (2.5,0)  arc (-180:-90:.25cm); 	
			\draw[] (3,0) arc (0:-90:0.25cm);	
			
			\draw[] (2.75,0) circle (.5pt);
			
			\draw[->]     (3,0) --  (5.75,0) node[right] {$\infty$} ;
			
			\draw[->]   (0,0) --  (5.75,3);
			
			\draw   (3.2,2) node[left] {$e^{\frac{\pi i}{4}}\R^+$};
			 \draw[dotted] (5.75,0) -- (5.75,3);
		\end{scope}
	\end{tikzpicture}
	\caption{Rotation of the paths of integration.} 
	\label{paths of integration}
\end{figure}

We compute
\begin{align*}
	&\operatorname{Res}_{u=\lp \kappa+\mu k + \frac{r}{2N} \rp^2 -\frac{d}{N}} \frac{e^{-2 \pi Nu \frac{1}{k^2\omega} } \sqrt{u+\frac{d}{N}} }{ u+\frac{d}{N}-\lp \kappa+\mu k + \frac{r}{2N} \rp^2} \\
	=& \lim_{u\to \lp \kappa+\mu k + \frac{r}{2N} \rp^2 -\frac{d}{N}} \left(u- \lp \lp \kappa+\mu k + \frac{r}{2N} \rp^2 -\frac{d}{N}\rp\right)  \frac{e^{-2 \pi Nu \frac{1}{k^2\omega} } \sqrt{u+\frac{d}{N}} }{ u+\frac{d}{N}-\lp \kappa+\mu k + \frac{r}{2N} \rp^2} \\
	=&  e^{-2 \pi N\left(\lp \kappa+\mu k + \frac{r}{2N} \rp^2 -\frac{d}{N}\right) \frac{1}{k^2\omega} } \sqrt{\lp \kappa+\mu k + \frac{r}{2N} \rp^2 } =  e^{-2 \pi N\left(\lp \kappa+\mu k + \frac{r}{2N} \rp^2 -\frac{d}{N}\right) \frac{1}{k^2\omega} } \left| \kappa+\mu k + \frac{r}{2N} \right|,
\end{align*}
which gives us that the contribution of this poles on the positive real line and with respect to  $\lim_{L\to\infty}\sum_{\mu =-L}^{L} \mathcal{P}_{d,N,1}(k,\omega,\kappa,r) $ sums up to
\begin{align*}
	\pi i \lim_{L\to\infty}\sum_{\mu =-L}^{L} \delta_{\lp \kappa+\mu k + \frac{r}{2N} \rp^2 -\frac{d}{N}\geq 0}  e^{-2 \pi N\left(\lp \kappa+\mu k + \frac{r}{2N} \rp^2 -\frac{d}{N}\right)  \frac{1}{k^2\omega} } \sgn\left(\kappa+\mu k + \frac{r}{2N}\right).
\end{align*}
Next we define the \textit{Jacobi theta function} as 
$$\vartheta_3(z;\tau)\coloneqq \sum_{n\in\Z}e^{\pi in^2\tau +2\pi inz}.$$ 
Noting that it transforms modular by (see e.g.\@, \cite[page 32]{mumford1983Tata})
$$\vartheta_3(z;\tau)= \frac{1}{\sqrt{-i\tau}} e^{-\frac{\pi i z^2}{\tau}} \vartheta_3\left(\frac{z}{\tau}; -\frac{1}{\tau}\right)$$
we can bound the absolute value of this contribution against
\begin{align*}
\pi e^{\pi d} \sum_{\mu\in\Z}  e^{- \pi Nk^2\lp \mu +\frac{\kappa}{k}  + \frac{r}{2Nk} \rp^2 } 
=& \pi e^{\pi d} \frac{1}{\sqrt{Nk^2}} \, \vartheta_3\left(\frac{\kappa}{k}+\frac{r}{2Nk};-\frac{1}{iNk^2}\right) = O_N\left(\frac{1}{k}\right),
\end{align*}
where we used \cite[Figure 20.3.4]{nist} in the last step\footnote{Note that $\vartheta_3(z;\tau)=\theta_3(\pi z; e^{\pi i\tau})$ from \cite[Figure 20.3.4]{nist}.}.

We are left with bounding the integrals on the rotated paths. We see that they sum up to
\begin{align*}
		\mathcal{Q}_{d,N}(k,\omega,\kappa,r) \coloneqq&  \frac{1}{\kappa+\mu k + \frac{r}{2N}}   	\int_{e^{\frac{\pi i}{4}}  \R^+}  \frac{e^{-2 \pi Nu \frac{1}{k^2\omega} } \sqrt{u+\frac{d}{N}} }{ u+\frac{d}{N}-\lp \kappa+\mu k + \frac{r}{2N} \rp^2} \, du \\
		=&  \frac{1}{\kappa+\mu k + \frac{r}{2N}}   e^{\frac{\pi i}{4}} 	\int_{0}^\infty  \frac{e^{-2 \pi N \frac{1+i}{\sqrt{2}}u \frac{1}{k^2\omega} } \sqrt{\frac{1+i}{\sqrt{2}}u+\frac{d}{N}} }{ \frac{1+i}{\sqrt{2}}u+\frac{d}{N}-\lp \kappa+\mu k + \frac{r}{2N} \rp^2} \, du,
\end{align*}
by changing variables $u \mapsto e^{\frac{\pi i}{4}} u $. Since $ \operatorname{Re}(\frac{2N}{k^2\omega}e^{\frac{\pi i}{4}})\geq \frac{N}{\sqrt{2}}$ and $u\geq 0$ we note that 
\begin{align*}
	\left|e^{-2 \pi N \frac{1+i}{\sqrt{2}}u \frac{1}{k^2\omega} } \right| = e^{- \pi u \operatorname{Re}\left(\frac{2N}{k^2\omega}e^{\frac{\pi i}{4}}\right) } \leq e^{- \frac{\pi uN}{\sqrt{2}} }.
\end{align*}
Furthermore we have
\begin{align*}
	\frac{1}{\left| \frac{1+i}{\sqrt{2}}u+\frac{d}{N}-\lp \kappa+\mu k + \frac{r}{2N} \rp^2\right|} \leq \frac{\sqrt{2}}{\left|\lp \kappa+\mu k + \frac{r}{2N} \rp^2-\frac{d}{N}\right|},
\end{align*}
which, together with the other bound, yields
\begin{align*}
	\left| \mathcal{Q}_{d,N}(k,\omega,\kappa,r) \right| \leq& \frac{\sqrt{2}}{\left|\kappa+\mu k + \frac{r}{2N}\right|\left|\lp \kappa+\mu k + \frac{r}{2N} \rp^2-\frac{d}{N}\right|}  	\int_{0}^\infty  e^{- \frac{\pi uN}{\sqrt{2}} } \left|\frac{1+i}{\sqrt{2}}u+\frac{d}{N} \right|^\frac{1}{2} \, du.
\end{align*}
Noting that $e^{- \frac{\pi uN}{\sqrt{2}} } |\frac{1+i}{\sqrt{2}}u+\frac{d}{N} |^\frac{1}{2}$ is integrable and only depending on $d$ and $N$ gives us that the integral occuring above is $O_N(1)$.
Therefore we are left with showing
\begin{align*}
	\lim_{L\to\infty}\sum_{\mu =-L}^{L} \frac{\sqrt{2}}{\left|\kappa+\mu k + \frac{r}{2N}\right|\left|\lp \kappa+\mu k + \frac{r}{2N} \rp^2-\frac{d}{N}\right|}= O_N\left(\frac{1}{\min\left( \kappa  + \frac{r}{2N},k - \kappa  - \frac{r}{2N} \right)}\right).  
\end{align*}
Using elementary estimates and the fact that $d<N$ we see that
\begin{small}
	\begin{align}\label{bound: last term against minimum}
	 \frac{\sqrt{2}}{k^3} \sum_{\mu \in\Z} &\frac{1}{\left|\mu + \frac{\kappa}{k}  + \frac{r}{2Nk}\right|\left|\lp \mu + \frac{\kappa}{k} + \frac{r}{2Nk} \rp^2-\frac{d}{Nk^2}\right|} \notag \\
	 <& \frac{\sqrt{2}}{k^3}\left[ \frac{k^3}{\left( \kappa  + \frac{r}{2N}\right)\left|\lp  \kappa + \frac{r}{2N} \rp^2-\frac{d}{N}\right|} + \frac{k^3}{\left(k - \kappa  - \frac{r}{2N}\right)\left|\lp k - \kappa- \frac{r}{2N} \rp^2-\frac{d}{N}\right|} \right. \notag \\
	 &\left. \hspace{1cm} + \frac{k^3}{\left( k + \kappa  + \frac{r}{2N}\right)\left(\lp k + \kappa + \frac{r}{2N} \rp^2-1\right)} + \frac{k^3}{\left(2k - \kappa - \frac{r}{2N}\right)\left(\lp 2k - \kappa - \frac{r}{2N} \rp^2- 1\right)}\right. \notag \\
	 &\left. \hspace{1cm} + \sum_{\mu \geq 2} \frac{1}{\mu \lp \mu^2-1\rp}  +  \sum_{\mu \geq 3} \frac{1}{\left(\mu - 1\right)\left(\lp \mu - 1 \rp^2-1\right)} \right]\notag \\
	 =& O_N\left(\frac{1}{\min\left(\left( \kappa  + \frac{r}{2N}\right)\left|\lp  \kappa + \frac{r}{2N} \rp^2-\frac{d}{N}\right|,\left(k - \kappa  - \frac{r}{2N}\right)\left|\lp k - \kappa- \frac{r}{2N} \rp^2-\frac{d}{N}\right| \right)}\right).
\end{align}
\end{small}
\hspace{-0.25cm} For $\kappa \in\{0,1,k-2,k-1\}$ the required bound holds, so we can restrict to the case $2\leq \kappa\leq k-3$.
Here we have 
$$ \lp  \kappa + \frac{r}{2N} \rp^2-\frac{d}{N}>1 \qquad \text{and} \qquad \lp k - \kappa- \frac{r}{2N} \rp^2-\frac{d}{N}>1,$$ 
which yields that we can also bound \eqref{bound: last term against minimum} against
$$ \frac{1}{\min\left( \kappa  + \frac{r}{2N},k - \kappa  - \frac{r}{2N} \right)}$$
and finishes the proof.
\end{proof}

Using Lemma \ref{lemma: bound pv integral} with $d=\frac{1}{24}$, \eqref{bound: Kloosterman sum} for all the Kloosterman sums, and noting that
$$ \sum_{\kappa=0}^{k-1}  \frac{1}{\min\left(\kappa+\frac{r}{2N}, k-\kappa-\frac{r}{2N}\right)} = O(\log(k))$$
yields
\begin{align*}
	a_{\CI^e, j,N,1}(n) =& O_N\left(\sum_{k=1}^J \sum_{r=1}^{N-1} nk^{\frac{1}{2}+\varepsilon} \log (k)\int_{-\frac{1}{k(J+k)}}^{\frac{1}{k(J+k)}} \left|e^{2\pi g_{j,N}(n)\omega}\right|  \, d\phi\right) \\
	=& O_N\left(n e^{2\pi g_{j,N}(n)J^{-2}} J^{-1} \sum_{k=1}^J  k^{-\frac{1}{2}+\varepsilon} \log(k) \right) = O_N\left( n e^{2\pi g_{j,N}(n)J^{-2} }   J^{-\frac 12+\varepsilon}  \log(J)\right),
\end{align*}
\begin{align*}
	a_{\CI^e, j,N,2}(n)  =& O_N\left( \sum_{k=1}^J \sum_{\ell=J+1}^{J+k-1}  \sum_{r=1}^{N-1}  nk^{\frac{1}{2}+\varepsilon} \log(k)   \int_{-\frac{1}{k\ell}}^{-\frac{1}{k(\ell+1)}}   \left|e^{2\pi g_{j,N}(n)\omega}\right|
	\,   d\phi\right) \\
	=& O_N\left(n e^{2\pi g_{j,N}(n)J^{-2}} J^{-1} \sum_{k=1}^J  k^{-\frac{1}{2}+\varepsilon} \log(k) \right) = O_N\left( n e^{2\pi g_{j,N}(n)J^{-2} }   J^{-\frac 12+\varepsilon}  \log(J)\right),
\end{align*}
and analogously $a_{\CI^e, j,N,3}(n) =O_N( n e^{2\pi g_{j,N}(n)J^{-2} }   J^{-\frac 12+\varepsilon}  \log(J)).$
Overall we thus obtain
\begin{align}\label{equation: final result a_{I^e}}
	a_{\CI^e, j,N}(n) 
	=& O_N\left(n e^{2\pi g_{j,N}(n)J^{-2} }   J^{-\frac 12+\varepsilon}  \log(J)\right).
\end{align}

Using the same bounds as used above for Lemma \ref{lemma: bound pv integral} with $d=0$, $\operatorname{Re}(\frac{2}{k^2\omega})\geq 1$, and noting that
$$ \sum_{m\geq1} p(m) e^{-\pi\left(m-\frac{1}{24}\right)} = \eta\left(\frac{i}{2}\right)^{-1} - e^{\frac{\pi}{24}}  =O(1)$$
yields
\begin{align*}
	a_{\CI, j,N,1}(n) =& O_N\left(  \sum_{k=1}^J \sum_{r=1}^{N-1}   \int_{-\frac{1}{k(J+k)}}^{\frac{1}{k(J+k)}} \sum_{m\geq 1} p(m) \left|e^{-\frac{2\pi\left( m-\frac{1}{24}\right) }{k^2\omega}}\right|  nk^{\frac{1}{2}+\varepsilon} \log(k) \left|e^{2\pi g_{j,N}(n)\omega} \right|  \, d\phi\right) \\
	=&  O_N\left( n e^{2\pi g_{j,N}(n)J^{-2}} \sum_{k=1}^J \sum_{r=1}^{N-1} k^{\frac{1}{2}+\varepsilon} \log(k)  \int_{-\frac{1}{k(J+k)}}^{\frac{1}{k(J+k)}}  \, d\phi\right) \\
	=& O_N\left( n e^{2\pi g_{j,N}(n)J^{-2} }   J^{-\frac 12+\varepsilon}  \log(J)\right),
\end{align*}
and analogously to above
\begin{align*}
	a_{\CI, j,N,2}(n)  =&  O_N\left( \sum_{k=1}^J \sum_{\ell=J+1}^{J+k-1}  \sum_{r=1}^{N-1}   \int_{-\frac{1}{k\ell}}^{-\frac{1}{k(\ell+1)}} \sum_{m\geq 1} p(m) \left|e^{-\frac{2\pi \left( m-\frac{1}{24}\right) }{k^2\omega}} \right| nk^{\frac{1}{2}+\varepsilon} \log(k) \left|e^{2\pi g_{j,N}(n)\omega}\right|	  \,   d\phi\right)\\
	=& O_N\left( n e^{2\pi g_{j,N}(n)J^{-2} }   J^{-\frac 12+\varepsilon}  \log(J)\right),
\end{align*}
and $a_{\CI, j,N,3}(n)= O_N( n e^{2\pi g_{j,N}(n)J^{-2} }   J^{-\frac 12+\varepsilon}  \log(J)).$
Overall we thus obtain
\begin{align}\label{equation: final result a_{I}}
a_{\CI, j,N}(n) 
=&  O_N\left( n e^{2\pi g_{j,N}(n)J^{-2} }   J^{-\frac 12+\varepsilon}  \log(J)   \right) .
\end{align}

\subsection{Combining the results}
Plugging \eqref{equation: final result a_{I^*}}, \eqref{equation: final result a_{I^e}}, and \eqref{equation: final result a_{I}} into \eqref{equation: splitting a_{j,N}}, using the definition of $g_{j,N}(n)$, and taking $J\to\infty$ gives
\begin{small}
\begin{align*}
&\!\!\!\!a_{j,N}(n) = -\frac{2\pi i }{\sqrt{n+\frac{j^2}{4N}-\frac{1}{24}}}\sum_{k\geq1} \sum_{r=1}^{N-1} \sum_{\kappa=0}^{k-1} \frac{ K_{k,j,N}(n,r,\kappa)}{k^2} \notag \\
&~~~~~\times  \mathrm{P.V.} 
\int_{-\sqrt{\frac{1}{24N}}}^{\sqrt{\frac{1}{24N}}} \sqrt{\frac{1}{24}-Nx^2} \cot\left(\pi \lp -\frac xk+ \frac{\kappa}{k} + \frac{r}{2Nk} \rp \right)  I_1\left(\frac{4\pi\sqrt{n+\frac{j^2}{4N}-\frac{1}{24}}}{k} \sqrt{\frac{1}{24}-Nx^2}\right)  \, dx .
\end{align*}
\end{small}
\hspace{-0.25cm} This finishes the proof of Theorem \ref{main Thm}.

Again by noting that the integral over $x$ only has a simple pole in $x= \frac{r}{2N}$ for $\kappa=0$ and $r<\sqrt{\frac N6}$ we additionally obtain 
\begin{scriptsize}
	\begin{align}\label{equation: coefficients final without P.V.}
	a_{j,N}(n) =& -\frac{2\pi i }{\sqrt{n+\frac{j^2}{4N}-\frac{1}{24}}}\sum_{k\geq1} \sum_{r=1}^{\left\lceil\sqrt{\frac N6}-1\right\rceil}  \frac{ K_{k,j,N}(n,r,0)}{k^2}   \\
	&\quad \times  
	\lim\limits_{\varepsilon\to0}\left(\int_{-\sqrt{\frac{1}{24N}}}^{\frac{r}{2N}-\varepsilon}  + \int_{\frac{r}{2N}+\varepsilon}^{\sqrt{\frac{1}{24N}}}\right) \sqrt{\frac{1}{24}-Nx^2} \cot\left(\pi \lp -\frac xk+  \frac{r}{2Nk} \rp \right)  I_1\left(\frac{4\pi\sqrt{n+\frac{j^2}{4N}-\frac{1}{24}}}{k} \sqrt{\frac{1}{24}-Nx^2}\right)  \, dx  \notag \\
	& -\frac{2\pi i }{\sqrt{n+\frac{j^2}{4N}-\frac{1}{24}}}\sum_{k\geq1} \sum_{r=\left\lceil\sqrt{\frac N6}\right\rceil}^{N-1}  \frac{ K_{k,j,N}(n,r,0)}{k^2}  \notag \\
	&\quad \times   
	\int_{-\sqrt{\frac{1}{24N}}}^{\sqrt{\frac{1}{24N}}} \sqrt{\frac{1}{24}-Nx^2} \cot\left(\pi \lp -\frac xk+  \frac{r}{2Nk} \rp \right)  I_1\left(\frac{4\pi\sqrt{n+\frac{j^2}{4N}-\frac{1}{24}}}{k} \sqrt{\frac{1}{24}-Nx^2}\right)  \, dx \notag  \\
	&-\frac{2\pi i }{\sqrt{n+\frac{j^2}{4N}-\frac{1}{24}}}\sum_{k\geq1} \sum_{r=1}^{N-1} \sum_{\kappa=1}^{k-1} \frac{ K_{k,j,N}(n,r,\kappa)}{k^2} \notag \\
	&\quad \times 
	\int_{-\sqrt{\frac{1}{24N}}}^{\sqrt{\frac{1}{24N}}} \sqrt{\frac{1}{24}-Nx^2} \cot\left(\pi \lp -\frac xk+ \frac{\kappa}{k} + \frac{r}{2Nk} \rp \right)  I_1\left(\frac{4\pi\sqrt{n+\frac{j^2}{4N}-\frac{1}{24}}}{k} \sqrt{\frac{1}{24}-Nx^2}\right)  \, dx. \notag 
	\end{align}
\end{scriptsize}

\begin{remark}
	Note that we had to exclude $n=0$ in our calculation. This is not caused by the fact that we have $\sqrt{g_{j,N}(n)}$ in the denominator of \eqref{equation: final result a_{I^*,0}} (which equals $0$ if and only if $n=0$ and $j=\sqrt{\frac N6}$), but because the estimates of our Kloosterman sums would break down for this special case (see \cite[Section 8]{rademacher1938fourier}). 
\end{remark}

\section{Numerical Results}\label{section: Numerical results}
In this section we offer some numerical results and compare the value of $a_{j,N}(n)$ for a number of cases to the results from Theorem \ref{main Thm}, where we numerically perform the sum over $k$ from $1$ to $J$.

\begin{table}[H]
	\centering
	\tabulinesep=1mm
	\begin{tabu}{l|c|c|c|c|c}
	 & J=1 & J=3 & J=20 & J=25 & J=50\\
	\hline
	\phantom{\rule{0mm}{4mm}}
	$a_{1,3}(3)= 2 $ & $2.3181\dots$    & $2.2886\dots$   & $2.0990\dots$ & $2.0875\dots$ & $2.0527\dots$ \\
	$a_{1,3}(10)= 30$ & $29.8989\dots$  & $30.2442\dots$  & $30.0866\dots$ & $30.0789\dots$ & $30.0418\dots$ \\
	$a_{1,3}(18)= 272$ & $271.3098\dots$   & $272.2656\dots$ & $272.0720\dots$ & $272.0651\dots$ & $272.0408\dots$ \\
	\hline
	\phantom{\rule{0mm}{4mm}}
	$a_{5,8}(3)= 2$ & $2.5197\dots$  & $2.2200\dots$  & $1.9993\dots$ & $1.9830\dots$ & $1.9892\dots$ \\
	$a_{5,8}(10)= 27$ & $26.2697\dots$ & $26.9853\dots$  & $26.9856\dots$ & $26.9997\dots$ & $26.9991\dots$ \\
	$a_{5,8}(18)= 216$ & $214.4979\dots$ & $216.0557\dots$ & $215.9830\dots$ & $215,9893\dots$ & $216.0044\dots$ \\
	\hline
	\phantom{\rule{0mm}{4mm}}
	$a_{3,10}(3)= 3$ & $3.1624\dots$  & $3.0544\dots$ & $3.0307\dots$ & $3.0222\dots$ & $2.9985\dots$ \\
	$a_{3,10}(10)= 39$ & $38.5337\dots$ & $38.9965\dots$ & $39.0080\dots$ & $39.0001\dots$ & $38.9982\dots$ \\
	$a_{3,10}(18)= 336$ & $334.3940\dots$  & $336.0237\dots$  & $336.0058\dots$ & $336.0254\dots$ & $336.0111\dots$ \\
	\hline
	\hline 
\end{tabu}
\vspace{4mm}
\caption{Numerical results for Fourier coefficients of $q^{\frac{1}{24}-\frac{j^2}{4N}}\CA_{j,N} (\tau)$.}
\end{table}

\section{Further Questions}
To end this paper we want to briefly mention some related questions that could be the topic for possible follow up projects.
\begin{enumerate}
	\item By splitting the Mordell-type integral in Section \ref{subsection: Splitting of the Mordell-type integral} we avoided the special case $2\sqrt{Nd}\in\Z\backslash\{0\}$ to verify the well-definedness of the principal value integral. Is there a possibility to get rid of the extra condition? What happens if we have a pole right at the edge of our integration path?
	\item If one is interested in figuring out what happens if we let $N$ tend to $\infty$ one could be more precise about the dependence of $N$ in the error terms while running the Circle Method as well as in the bound of the Kloosterman sum.
\end{enumerate}

\end{document}